\documentclass[reqno,12pt]{amsart} %amsproc/amsbook
\usepackage{amssymb,amsthm,amsmath,mathrsfs,graphicx,amsfonts,bbm}

\usepackage{esint}
\usepackage{bookmark}
\usepackage{xcolor}
\usepackage{tikz}
\usetikzlibrary{arrows,chains,matrix,positioning,scopes}
\makeatletter
\tikzset{join/.code=\tikzset{after node path={%
			\ifx\tikzchainprevious\pgfutil@empty\else(\tikzchainprevious)%
			edge[every join]#1(\tikzchaincurrent)\fi}}}
\makeatother
\tikzset{>=stealth',every on chain/.append style={join},
	every join/.style={->}}
\tikzstyle{labeled}=[execute at begin node=$\scriptstyle,
execute at end node=$]
\usetikzlibrary{matrix}
\usepackage{float}
\usepackage[font=footnotesize]{caption}
\usepackage[left=1.3in, right=1.3in, bottom=1.5in]{geometry}
\newtheorem{theorem}{Theorem}[section]
\newtheorem{lemma}[theorem]{Lemma}
\newtheorem{proposition}[theorem]{Proposition}

\theoremstyle{definition}
\newtheorem{definition}[theorem]{Definition}

\theoremstyle{remark}
\newtheorem{remark}[theorem]{Remark}

\numberwithin{equation}{section}

\newcommand{\norm}[1]{\lVert#1\rVert}
\newcommand{\Ag}[1]{\langle#1\rangle}

\newcommand{\X}{\mathcal{X}}

\newcommand{\R}{\mathbb{R}}
\newcommand{\Z}{\mathbb{Z}}
\newcommand{\E}{\mathbb{E}}
\newcommand{\N}{\mathbb{N}}

\newcommand{\e}{\varepsilon}

\newcommand{\txt}[1]{\text{\rm #1}}

\newcommand{\argmin}{\mathop{\mathrm{argmin}}\limits}
\newcommand{\argmax}{\mathop{\mathrm{argmax}}\limits}
\def\Xint#1{\mathchoice
  {\XXint\displaystyle\textstyle{#1}}%
  {\XXint\textstyle\scriptstyle{#1}}%
  {\XXint\scriptstyle\scriptscriptstyle{#1}}%
  {\XXint\scriptscriptstyle\scriptscriptstyle{#1}}%
  \!\int}
\def\XXint#1#2#3{{\setbox0=\hbox{$#1{#2#3}{\int}$}
  \vcenter{\hbox{$#2#3$}}\kern-.5\wd0}}

\def\average{\Xint-}

\newcommand{\vertiii}[1]{{\left\vert\kern-0.25ex\left\vert\kern-0.25ex\left\vert #1
		\right\vert\kern-0.25ex\right\vert\kern-0.25ex\right\vert}}
	
\newcommand{\RN}[1]{%
	\textup{\uppercase\expandafter{\romannumeral#1}}%
}
\begin{document}

\title{Large-scale Regularity of Nearly Incompressible Elasticity in Stochastic Homogenization}


%    author one information

\author{Shu Gu\quad and\quad Jinping Zhuge}
% \author{Shu Gu}
% \thanks{}

% \author{Jinping Zhuge}
% \thanks{}

%\date{\today}

%    The 2010 edition of the Mathematics Subject Classification is
%    now available.  If you are citing a classification from the
%    new scheme, use the following input coding instead.
%\subjclass[2010]{Primary }


\maketitle

\begin{abstract}
In this paper, we systematically study the regularity theory of the linear system of nearly incompressible elasticity. In the setting of stochastic homogenization, we develop new techniques to establish the large-scale estimates of displacement and pressure, which are uniform in both the scale parameter and the incompressibility parameter. In particular, we obtain the boundary estimates in a new class of Lipschitz domains whose boundaries are smooth at large scales and bumpy at small scales.
\end{abstract}

\pagestyle{plain}
\tableofcontents

\section{Introduction}
\subsection{Motivations}
The system of linear elasticity for a homogeneous isotropic material is called Lam\'{e} system given by
\begin{equation}\label{eq.lame}
\mu \Delta u + (\mu +\lambda) \nabla(\nabla\cdot u) = F,
\end{equation}
where the vector-valued function $u = (u^1,u^2,\cdots, u^d)$ represents the displacement field in equilibrium state, and the scalar constants $\mu$ and $\lambda$ are usually referred as Lam\'{e}'s first and second parameters. In terms of the modulus of elasticity $E$ and Poisson ratio $\nu$ ($-1<\nu<\frac{1}{2}$), one has
\begin{equation*}
\mu = \frac{E}{2(1+\nu)}, \qquad \lambda = \frac{E\nu}{(1+\nu)(1-2\nu)}.
\end{equation*}
The quantity $\frac{2}{3}\mu + \lambda = \frac{1}{3}E/(1-2\nu)$ is called the bulk modulus, which measures the compressibility. A material tends to be incompressible, if the bulk modulus is large, or equivalently, if the Poisson ratio $\nu$ is close to $\frac{1}{2}$. The endpoint $\nu = \frac{1}{2}$ corresponds to the incompressible materials and reduces the system into a Stokes system. In the real world, all materials are more or less compressible, while the Poisson ratio for some materials could be very close to $\frac{1}{2}$. For example, the typical nearly incompressible material, natural rubber, has a Poisson ratio of $0.4999$ \cite{MR09}. Due to many applications of nearly incompressible materials in engineering, the Lam\'{e} system (\ref{eq.lame}) with large $\lambda$ has been studied extensively in physics and numerical analysis (see, e.g., \cite{Temam01,Her65,Vog83,KS95,HL02,MDR08,DWO09} and references therein). However, to the best of the authors' knowledge, only sparse results are known on the theoretical analysis of the nearly incompressible elasticity.

In this paper, we study the system of linear elasticity for inhomogeneous, anisotropic, nearly incompressible materials \cite{OSY92,JikovKozlovOleinik94}. Precisely, let $D\subset \R^d$ be a bounded domain occupied by a material body and consider the system with Dirichlet boundary condition
\begin{equation}\label{eq.main}
\left\{
\begin{aligned}
\nabla\cdot (A(x) \nabla u) +  \nabla (\lambda(x) \nabla\cdot u) &= F \qquad &\txt{in } &D, \\
u  &= f \qquad &\txt{on } &\partial D,
\end{aligned}
\right.
\end{equation}
where $A(\cdot)=(a_{ij}^{\alpha\beta}(\cdot)): \mathbb{R}^d \mapsto \mathbb{R}^{d^2\times d^2}$ is a tensor-valued function and $\lambda(\cdot):\R^d\mapsto [0,\infty)$ is a scalar function. We point out that $\lambda(\cdot)$ plays a role, similar as Lam\'{e}'s second parameter or bulk modulus, in measuring the incompressibility of the material.

Our primary hypothesis for the coefficients $A$ and $\lambda$ (measurable and deterministic) are as follows:
\begin{itemize}
	\item Ellipticity condition: there exists a fixed constant $\Lambda>0$ so that
	\begin{equation}\label{cond.ellipticity}
	\Lambda^{-1} |\xi|^2 \le  a_{ij}^{\alpha\beta}(x)\xi_i^\alpha \xi_j^\beta \le \Lambda |\xi|^2 \quad \text{for any } x\in \R^d, \xi \in \R^{d\times d}. 
	\end{equation}
	(The Einstein summation convention will be used throughout the paper.)
	
	\item Symmetry condition:
	\begin{equation}\label{cond.symmetry}
	a_{ij}^{\alpha\beta} = a_{ji}^{\beta\alpha} \quad \text{for any } 1\le i,j,\alpha,\beta \le d.
	\end{equation}
	
	\item Compressibility condition: there is a constant $\lambda_0 \ge 0$ so that
	\begin{equation}\label{cond.incompr}
	\lambda_0 \le \lambda(x) \le \lambda_0 + \Lambda.
	\end{equation}
\end{itemize}
Note that (\ref{cond.ellipticity}) and (\ref{cond.symmetry}) are the usual ellipticity and symmetry conditions for the system of elasticity. However, the ``incompressibility parameter'' $\lambda_0$ in the compressibility condition (\ref{cond.incompr}) is allowed to be arbitrarily large which makes the system (\ref{eq.main}) very ill-conditioned. In this paper, we are interested in the uniform regularity of the solutions that are independent of $\lambda_0$.

Observe that the upper bound of (\ref{cond.incompr}) also implies that the oscillation of $\lambda$ is bounded by a fixed constant $\Lambda$, though its magnitude could be arbitrarily large. This assumption is purely technical but crucial for our analysis, because it allows us to reduce (\ref{eq.main}) to a simpler situation. Indeed, we may write $\lambda(x) = \lambda_0 + b(x)$ and
\begin{equation}\label{eq.reduction}
\begin{aligned}
\nabla\cdot (A(x) \nabla u) + \nabla(\lambda(x) \nabla\cdot u) &= \Big[ \nabla\cdot (A(x) \nabla u) +  \nabla(b(x) \nabla\cdot u) \Big]+ \nabla(\lambda_0 \nabla\cdot u) \\
& = \nabla\cdot (\widetilde{A}(x) \nabla u) + \nabla(\lambda_0 \nabla\cdot u),
\end{aligned}
\end{equation}
where $\widetilde{A} = (\widetilde{a}_{ij}^{\alpha\beta})$ is defined by
\begin{equation}\label{eq.splitting}
\widetilde{a}_{ij}^{\alpha\beta}(x) = a_{ij}^{\alpha\beta}(x) + b(x)\delta_{i}^\alpha \delta_{j}^\beta.
\end{equation}
Clearly, $\widetilde{A}$ still satisfies the ellipticity and symmetry conditions. Hence, without loss of generality, we may simply assume $\lambda\ge 0$ is constant and concentrate on the following Dirichlet boundary value problem
\begin{equation}\label{eq.elast}
\left\{
\begin{aligned}
\nabla\cdot (A(x) \nabla u_\lambda) +  \lambda \nabla (\nabla\cdot u_\lambda) &= F \qquad &\txt{in } &D, \\
u_\lambda  &= f \qquad &\txt{on } &\partial D,
\end{aligned}
\right.
\end{equation}
where $A$ satisfies (\ref{cond.ellipticity}) and (\ref{cond.symmetry}) and $\lambda \ge 0$ is an arbitrary constant. 

\subsection{General regularity theory}
Nowadays, the regularity theory for elliptic equation/system and Stokes system has been well-understood. The system (\ref{eq.elast}) can be viewed as an intermediate state between elliptic system and Stokes system. Intuitively, the uniform regularity should be expected thanks to the fine regularity of the endpoint systems. However, the regularity for (\ref{eq.elast}) uniform in $\lambda$ seems to be a mathematically different and harder problem, compared to the usual elliptic or Stokes systems. For example, the fundamental Caccioppoli inequality for (\ref{eq.elast}) does not hold uniform in $\lambda$. Actually, if $\nabla\cdot (A(x) \nabla u_\lambda) +  \lambda \nabla (\nabla\cdot u_\lambda) = 0$ in $B_{2r}$, we can only show
\begin{equation}\label{est.Caccioppoli}
\average_{B_r}|\nabla u_\lambda|^2 \le \frac{C\lambda}{r^2} \average_{B_{2r}}|u_\lambda|^2.
\end{equation}
Note that $\lambda$, appearing in front of the $L^2$ norm of $u_\lambda$, makes the ienquality useless in the study of the uniform regularity. Fortunately, we invent a novel variation of (\ref{est.Caccioppoli}), which will be called the generalized Caccioppoli inequality,
\begin{equation}\label{est.G.Cacc}
\begin{aligned}
\average_{B_r} |\nabla u_\lambda|^2 & \le \frac{C}{r^2} \average_{B_{2r}} | u_\lambda|^2 + \frac{C}{r^2} \norm{\lambda \nabla\cdot u_\lambda - \average_{B_{2r}}\lambda \nabla\cdot u_\lambda  }_{\underline{H}^{-1}(B_{2r})}^2 \\
& \qquad + C\sup_{k,\ell \in [1/4,1]} \bigg| \average_{B_{2kr}} \lambda\nabla \cdot u_\lambda - \average_{B_{2\ell r}} \lambda\nabla \cdot u_\lambda \bigg|^2.
\end{aligned}
\end{equation}
We emphasize that, in (\ref{est.G.Cacc}) and all the estimates involved in this paper, $\lambda$ is harmless when it comes together with $\nabla\cdot u_\lambda$. As a whole, $\lambda\nabla\cdot u_\lambda$ has obvious physical and mathematical meaning, namely, the ``pressure''. The additional structures with the pressure in (\ref{est.G.Cacc}) gives a taste why the uniform regularity of (\ref{eq.elast}) should be expected, but meanwhile more complicated.

The first part of this paper is devoted to the uniform regularity of (\ref{eq.main}) or (\ref{eq.elast}) in the non-homogenization setting without $\e$. Besides the energy estimates and the generalized Caccippoli inequality mentioned above, our main tool to study the uniform regularity for large $\lambda$ is the following asymptotic expansion (also see \cite{Temam01})
\begin{equation}\label{est.exp.intr}
u_\lambda = \sum_{k = 0}^\infty \lambda^{-k} v_k \quad \text{in } H^1 \quad \txt{and} \quad \lambda \nabla\cdot u_\lambda - \lambda \Ag{f}_D = \sum_{k=0}^\infty \lambda^{-k} p_k \quad \text{in } L^2,
\end{equation}
where $(v_k,p_k), k\ge 0$, are the solutions of certain $\lambda$-independent iterative Stokes systems whose regularity are known (see Theorem \ref{thm.expansion}), and $\Ag{f}_D$ is a constant defined in (\ref{def.Agf}). In particular, the pair $(u_\lambda,\lambda\nabla\cdot u_\lambda - \lambda\Ag{f}_D)$, as $\lambda\to \infty$, converges quantitatively to $(v_0,p_0)$ which is the solution of a Stokes system
\begin{equation*}
\left\{
\begin{aligned}
\nabla\cdot (A(x) \nabla v_0) + \nabla p_0 &= F \qquad &\txt{in }& D, \\
\nabla\cdot v_0 & = \Ag{f}_D \qquad &\txt{in }& D, \\
v_0 & = f \qquad &\txt{on }& \partial D.
\end{aligned}
\right.
\end{equation*}
The above asymptotic behavior of $u_\lambda$ provides us at least two approaches to study the uniform regularity: (1) exploring the regularity for all $(v_k,p_k)$ in (\ref{est.exp.intr}); (2) a real variable perturbation argument. It turns out these two approaches are effective in different situations. For instance, the first approach is useful in Schauder estimate (Theorem \ref{thm.C1a}), while the second one is more efficient in Calder\'{o}n-Zygmund estimate (Theorem \ref{thm.Meyers.Elasticity}). These basic regularity estimates uniform in $\lambda$ are of independent interest and crucial for studying the uniform regularity in homogenization.

\subsection{Regularity in homogenization}
The larger part of this paper will be devoted to the uniform regularity in stochastic homogenization.
Recently, there have been lots of interests in the uniform regularity in homogenization theory in either random (e.g., \cite{AS16,AKM16,AM16,AKM17,AKM19,GNO15,GO17,MN17}) or deterministic settings (e.g., \cite{AL8701,AL89,KenigLinShen1301,ArmstrongShen15,GuShen15,GuZhuge17}). In particular, the uniform regularity for Stokes system in periodic homogenization has been studied in \cite{GuShen15,GuXu17,GuZhuge17,Xu17}. In this paper, we consider the system of nearly incompressible elasticity in a bounded Lipschitz domain with an additional tiny scale parameter $\e>0$:
\begin{equation}\label{eq.elasticity.Le}
\nabla\cdot (A^\e \nabla u^\e_\lambda) + \nabla ( \lambda^\e \nabla\cdot u^\e_\lambda) = 0 \qquad \txt{in } D,
\end{equation}
where $A^\e(x) = A(x/\e), \lambda^\e(x) = \lambda(x/\e)$ and the solution $u_\lambda^\e$ depends both on $\e$ and $\lambda_0$.\footnote{Without ambiguity, we write $u^\e_\lambda$, instead of $u^\e_{\lambda_0}$, for short.} We are interested in the interior and boundary uniform estimates that are independent of both $\e\in (0,1)$ and $\lambda_0 \in (0,\infty)$. Notice that the expansion (\ref{est.exp.intr}) also applies to the system (\ref{eq.elasticity.Le}). Therefore, we expect to obtain the uniform estimate for the system (\ref{eq.elasticity.Le}) by Theorem \ref{thm.expansion}, as long as the same uniform estimate holds for the Stokes system. This straightforward strategy actually works for the $W^{1,p}$ estimate with $p\in (1,\infty)$; see Theorem \ref{thm.W1p.periodic} for example (in periodic setting). However, it fails for the Lipschitz estimate of $u_\lambda^\e$, due to the following two essential reasons: (1) Lipschitz estimate is optimal in homogenization. This means that it is impossible to prove a higher regularity, say $C^{1,\alpha}$ estimate, that implies the Lipschitz estimates. (2) The $L^\infty$ estimate of the pressure (corresponding to the Lipschitz estimate of $u_\lambda^\e$) is not preserved through the iterative Stokes system (\ref{eq.iterated.Stokes}), since the map $p_{k-1} \mapsto p_k$ is a singular  integral which definitely is not bounded in $L^\infty$.

In this paper, we will make substantial modifications to the excess decay method developed recently to establish the uniform Lipschitz and pressure estimates for (\ref{eq.elasticity.Le}) in stochastic homogenization. Moreover, we generalize the large-scale boundary estimates to a class of Lipschitz domains whose boundaries are ``smooth'' only above $\e$-scale and could be very rough at or below $\e$-scale.

Before stating the main result, we first set up
the random environment in quantitative homogenization theory. We will follow the approach developed recently in, e.g., \cite{AS16,AKM16,AM16,AKM17}, which is based on the natural stationarity and ergodicity assumptions on the coefficient fields. 
Precisely, denote the set of all the possible coefficient fields by
\begin{equation}\label{def.Omega}
\Omega := \{ (A,\lambda): \mathbb{R}^d \mapsto \mathbb{R}^{d^2\times d^2} \times \R \text{ satisfying } (\ref{cond.ellipticity}) - (\ref{cond.incompr}) \},
\end{equation}
which is endowed by a family of $\sigma$-algebras $\{\mathcal{F}_D\}$, where $\mathcal{F}_D$ is the $\sigma$-algebra representing the information of the coefficient fields restricted in $D$, formally generated by 
\begin{equation*}
\begin{aligned}
&\mathcal{F}_D := \text{the $\sigma$-algebra generated by the random elements} \\
&\quad (A,\lambda) \mapsto \bigg( \int_{\R^d}a_{ij}^{\alpha\beta}(x)\phi(x), \int_{\R^d} \lambda(x)\psi(x) \bigg),  \quad \phi,\psi \in C_0^\infty(D), 1\le i,j,\alpha,\beta\le d.
\end{aligned}
\end{equation*}
Let $\mathcal{F} = \mathcal{F}_{\R^d}$ be the largest $\sigma$-algebra in the family $\{\mathcal{F}_D\}$. We further assume that there is a probability measure $\mathbb{P}$ satisfying the following assumptions:
\begin{itemize}
	
	\item Stationarity with respect to $\mathbb{Z}^d$-translations:
	\begin{equation}\label{def.stationarity}
	\mathbb{P}\circ T_z = \mathbb{P}, \quad\text{where }(T_z(A,\lambda))(x)= (A(x+z),\lambda(x+z)).
	\end{equation}
	
	\item Unit range of dependence:
	\begin{equation}\label{def.unit-range-dependence}
	\begin{aligned}
	&\text{$\mathcal{F}_D$ and $\mathcal{F}_E$ are $\mathbb{P}$-independent for every Borel}\\
	&\qquad\text{subset pair $D, \, E\subset \mathbb{R}^d$ satisfying dist$(D,E)\ge 1$}.
	\end{aligned}
	\end{equation}
	
\end{itemize}

Throughout this paper, we will use the following notation, which has been commonly used in many recent references, to control the size of a random variable. For a random variable $X:\Omega\to [1,\infty)$, we write $X\le \mathcal{O}_s(\theta)$ for some $s>0,\theta>0$, if
\begin{equation}\label{def.XO}
\E\Big[ \exp \Big( (X/\theta)^s \Big) \Big] \le 2,
\end{equation}
where $\mathbb{E}[Y]$ denotes the expectation of the random variable $Y$. Note that (\ref{def.XO}) implies that, for any $t\ge 1$,
\begin{equation}\label{est.prob}
\mathbb{P}[ X\ge t] \le 2\exp((-t/\theta)^s).
\end{equation}
Conversely, (\ref{est.prob}) implies $X\le \mathcal{O}_s(2^{1/s}\theta)$. As a convention, we write $X\le Y + \mathcal{O}_s(\theta)$ if $X-Y \le \mathcal{O}_s(\theta)$.

Now we state the main theorem for the interior estimate.
\begin{theorem}\label{thm.main}
	Let $(\Omega,\mathcal{F},\mathbb{P})$ be as above. For any $s\in (0,d)$ and $\lambda_0 \in [0,\infty)$, there exist a constant $C_0 = C_0(s,d,\Lambda)$ and a random variable $\X = \X_{s,\lambda}: \Omega\mapsto [1,\infty)$ satisfying
	\begin{equation}\label{est.Xs}
	\X \le \mathcal{O}_{s}(C_0),
	\end{equation}
	such that if $u_\lambda^\e\in H^1(B_2;\R^d)$ is a weak solution of
	\begin{equation}\label{eq.elasticity.B2}
	\nabla\cdot (A^\e \nabla u^\e_\lambda) + \nabla ( \lambda^\e \nabla\cdot u^\e_\lambda) = 0 \qquad\txt{in } B_2 = B_2(0),
	\end{equation}
	then for any $r \in [\e \X, 1]$, we have
	\begin{equation}\label{est.Lip.Osc}
	\bigg( \average_{B_r} |\nabla u_\lambda^\e|^2 \bigg)^{1/2} + \bigg( \fint_{B_r} |\lambda^\e \nabla\cdot u_\lambda^\e -  \fint_{B_2} \lambda^\e \nabla\cdot u_\lambda^\e|^2 \bigg)^{1/2} \le C \bigg( \average_{B_2} |\nabla u_\lambda^\e|^2 \bigg)^{1/2},
	\end{equation}
	where $B_r = B_r(0)$ and $C$ depends only on $s,d$ and $\Lambda$.
\end{theorem}

Before all else, we should emphasize that the nontrivial point of Theorem \ref{thm.main} is that the constants involved are independent of both $\e$ and $\lambda_0$. The first part of (\ref{est.Lip.Osc}) is called the large-scale Lipschitz estimate of the displacement. In other words, the ``average deformation'' of the material at relatively large scales ($r \ge \e \X$) is controlled by the average deformation at macroscopic scale ($r=2$). Particularly, this guarantees the continuity of the material and no ``cracks'' will be seen at scales greater than $\e\X$. Since there is no regularity assumption on the coefficients $(A,\lambda)$, we cannot expect any continuity for $u_\lambda^\e$ at scales less than $\e$. The second part of (\ref{est.Lip.Osc}) is called the large-scale oscillation estimate of the pressure, which is an exclusive feature of the system of elasticity or Stokes system. This particularly implies that the ``average pressure'' at a relatively large scale ($r\ge \e\X$) has uniform bounded oscillation, i.e.,
\begin{equation*}
\bigg| \fint_{B_r(x)} \lambda^\e \nabla\cdot u_\lambda^\e - \fint_{B_r(y)} \lambda^\e \nabla\cdot u_\lambda^\e \bigg| \le C \bigg( \average_{B_2} |\nabla u_\lambda^\e|^2 \bigg)^{1/2},
\end{equation*}
for any $x,y\in B_1(0)$ and $r\in [\e\X,1]$. If $\lambda_0$ is large, this estimate actually implies that the spatial change of $\nabla\cdot u_\lambda^\e$ is small as we expected for the nearly incompressible materials. We also mention that the estimate (\ref{est.Lip.Osc}) together with the ranges of $s\in (0,d)$ and $r\in [\e\X,1]$ is optimal in stochastic integrability \cite{AKM19}. Of course, because of (\ref{est.Xs}) and (\ref{est.prob}), the random variable $\X$ is large only with a small probability (decaying exponentially). Precisely, Theorem \ref{thm.main} implies that
\begin{equation}\label{est.Prob.Lip}
\mathbb{P} \big[ \text{(\ref{est.Lip.Osc}) holds for } r \big] \ge 1-2\exp\Big(-\Big( \frac{r}{C_0 \e}\Big)^s \Big).
\end{equation}
Note that this probability (depending on $r/\e$) is independent of $\lambda_0$.

Now, let us consider the boundary estimates. As we mentioned earlier, we will establish the boundary estimates in a class of Lipschitz domains, defined as follows.

\begin{definition}\label{def.C1a.e}
	Let $\alpha \in (0,1]$ and $D$ be a bounded Lipschitz domain with $0\in \partial D$. We say that $D$ satisfies the $\e$-scale $C^{1,\alpha}$ condition at $0$, if there exist $C_0>0$ and $r_0>0$ such that for any $r\in (\e,r_0)$, there exists a unit vector $n_r$ such that
	\begin{equation}
	\begin{aligned}
	& \big\{y\in \R^d: y\cdot n_r < -C_0 r \big[ r^\alpha + (\e/r)^\alpha \big] \big\} \cap B_r(0) \\
	& \qquad \subset D\cap B_r(0) \subset \big\{y\in \R^d: y \cdot n_r < C_0 r \big[ r^\alpha + (\e/r)^\alpha \big] \big\}\cap B_r(0).
	\end{aligned}
	\end{equation}
\end{definition}

From the above definition, we see that the local boundary $\partial D \cap B_r(0)$ is contained between two parallel hyperplanes with distance comparable to $r \zeta_\alpha(r,\e)$, where $\zeta_\alpha(r,\e):= r^\alpha + (\e/r)^\alpha$.
In particular, this class of domains covers the classical $C^{1,\alpha}$ domains and the so called bumpy Lipschitz domains. Obviously, the $C^{1,\alpha}$ domains correspond to the case $r\zeta_\alpha(r,\e) = r^{1+\alpha}$. On the other hand, in \cite{KP15} and \cite{KP18}, Kenig and Prange studied the Lipschitz estimate by the compactness argument in the bumpy Lipschitz domain whose boundary is the graph of the function 
\begin{equation}\label{eq.bdryGraph}
x_d = \e \psi(x'/\e),
\end{equation}
where $\psi \in W^{1,\infty}(\R^{d-1})$. This actually corresponds to the special case $r\zeta_1(r,\e) = \e = r(\e/r)$ in Definition \ref{def.C1a.e}. From these two typical exampls, we notice that the two parts of the function $\zeta_\alpha$ come from two different sources, namely, smoothness and small bumps, which dominates at large-scales and small scales, respectively. Particularly, Definition \ref{def.C1a.e} includes the domain whose boundary is the local graph of
\begin{equation*}
x_d = \psi_0(x') + \e\psi_1(x'/\e),
\end{equation*}
where $\psi_0\in C^{1,\alpha}(\R^{d-1})$ and $\psi_1\in W^{1,\infty}(\R^{d-1})$.

\begin{figure}[h]
	\begin{center}
		\includegraphics[scale =0.3]{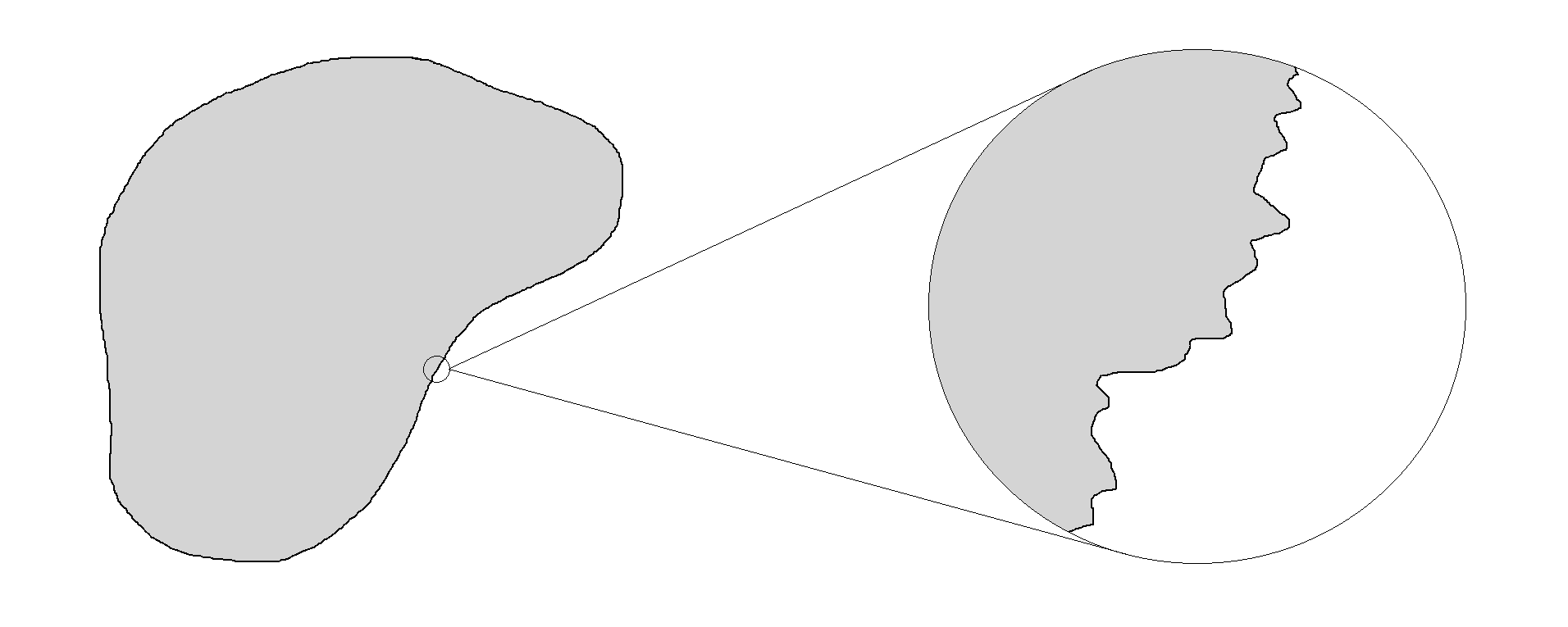}
	\end{center}
	\vspace{-1 em}
	\caption{A Lipschitz domain bumpy at small scales}\label{fig_1}
\end{figure}

\begin{remark}
	In Definition \ref{def.C1a.e}, we assume in priori that $D$ is a Lipschitz domain. This assumption actually is not essential in the proof of the large-scale regularity. It is only required for Lemma \ref{lem.Nabla} which affects the basic energy estimate for the Stokes system and the system of elasticity with large $\lambda_0$. It is possible to relax this assumption to even more general domains with fractals (such as John domains \cite{Jones81,ADM06,Rogers06}). For simplicity, however, we will not explore this direction in the present paper.
\end{remark}

\begin{remark}
	Oscillating boundaries given by (\ref{eq.bdryGraph}), with additional structure such as periodicity, have been widely studied in the analysis of the wall law for the Navier-Stokes equations with rough boundaries; see \cite{BG08,GD09,GM10,DG11,DP14} for some recent references. Most recently, Higaki and Prange \cite{HP20} obtained the large-scale Lipschitz estimate for the stationary Navier-Stokes equations over bumpy Lipschitz boundaries without any structure by a compactness method.
\end{remark}

Our main result for the boundary estimate is stated as follows.  
\begin{theorem}\label{thm.main1}
	Let $D$ be a bounded Lipschitz domain satisfying the $\e$-scale $C^{1,\alpha}$ condition at $0\in \partial D$. Define $D_t = D\cap B_t(0)$ and $\Delta_t = \partial D\cap B_t(0)$. Let $(\Omega,\mathcal{F},\mathbb{P})$ be as before. For any $s\in (0,d)$ and $\lambda_0\in [0,\infty)$, there exist  a constant $C_0 = C_0(s,d,\Lambda)$ and a random variable $\X = \X_{s,\lambda}: \Omega\mapsto [1,\infty)$ satisfying
	\begin{equation*}
	\X \le \mathcal{O}_{s}(C_0),
	\end{equation*}
	such that if $u_\lambda^\e\in H^1(D_{2};\R^d)$ is a weak solution of
	\begin{equation}
	\left\{
	\begin{aligned}
	\nabla\cdot (A^\e \nabla u^\e_\lambda) +  \nabla (\lambda^\e \nabla\cdot u^\e_\lambda) &= 0 \qquad &\txt{in } &D_{2}, \\
	u_\lambda^\e  &=0  \qquad &\txt{on } & \Delta_{2},
	\end{aligned}
	\right.
	\end{equation}
	then for any $r \in [\e \X, 1]$, 
	\begin{equation}\label{est.bdryLip.Osc}
	\bigg( \average_{D_r} |\nabla u_\lambda^\e|^2 \bigg)^{1/2} + \bigg( \fint_{D_r} |\lambda^\e \nabla\cdot u_\lambda^\e -  \fint_{D_2} \lambda^\e \nabla\cdot u_\lambda^\e|^2 \bigg)^{1/2} \le C \bigg( \average_{D_2} |\nabla u_\lambda^\e|^2 \bigg)^{1/2}.
	\end{equation}
\end{theorem}

The above theorem gives the expected boundary estimates parallel to Theorem \ref{thm.main}. The main novelty of Theorem \ref{thm.main1} is that the boundary is not necessarily smooth below $\e$-scale. This phenomenon seems physically and experimentally natural, as the microscopic structure of the bumpy boundary (which is always the case in reality) should not have a visible influence if only the large-scale or macroscopic regularity is concerned. In other words, the following philosophy should be valid:
\begin{equation*}
\begin{aligned}
&\text{The large-scale smoothness of the boundary}\\ & \qquad \Longrightarrow \text{the large-scale smoothness of the solutions.}
\end{aligned}
\end{equation*}
Furthermore, it seems very promising that the quantitative method in this paper may also apply to other types of equations, such as Navier-Stokes equations.

\begin{remark}
	By forcing $\lambda\equiv 0$, our main theorems recover the results for the usual elliptic system. Particularly, Theorem \ref{thm.main1} recovers Kenig and Prange's work \cite{KP18} for the large-scale boundary estimate. On the other hand,
	by taking $\lambda_0\to \infty$ (literally, replacing $\lambda^\e\nabla\cdot u_\lambda^\e$ by $p^\e$ in (\ref{est.Lip.Osc}) and (\ref{est.bdryLip.Osc})), the results in this paper also implies the large-scale regularity for Stokes system (namely, the system of completely incompressible elasticity).
\end{remark}

\subsection{New ingredients of the proofs}
We will prove our main theorems by a method of excess decay iteration. The key step in this method is to show an algebraic rate of convergence for (\ref{eq.elasticity.Le}) (which will first be reduced to the case with constant $\lambda$). Precisely, if $u_\lambda^\e\in H^1(D;\R^d)$ is a weak solution of
\begin{equation}\label{eq.elast.e}
\left\{
\begin{aligned}
\nabla\cdot (A^\e \nabla u^\e_\lambda) + \lambda \nabla ( \nabla\cdot u^\e_\lambda) &= 0 \qquad &\txt{in } &D, \\
u^\e_\lambda  &= f \qquad &\txt{on } &\partial D,
\end{aligned}
\right.
\end{equation}
for some $f\in W^{1,2+\delta}(D;\R^d)$ with $\delta>0$, then we show that this system homogenizes to
\begin{equation}\label{eq.elast.0}
\left\{
\begin{aligned}
\nabla\cdot (\overline{A}_\lambda \nabla u^0_\lambda) + \lambda \nabla ( \nabla\cdot u^0_\lambda) &= 0 \qquad &\txt{in } &D, \\
u^0_\lambda  &= f \qquad &\txt{on } &\partial D,
\end{aligned}
\right.
\end{equation}
with a rate of convergence
\begin{equation}\label{est.global.rate}
\begin{aligned}
& \norm{u^\e_\lambda - u^0_\lambda}_{L^2(D)} + \norm{\lambda \nabla\cdot u^\e_\lambda - \lambda\nabla \cdot u^0_\lambda}_{H^{-1}(D)} \\
&\quad \le C \big(\e^{\beta(d-s)}  + (\X \e)^{\alpha s} \big)\norm{\nabla f}_{L^{2+\delta}(D)},
\end{aligned}
\end{equation}
for some $\alpha,\beta\in (0,1)$, where the random variable satisfies $\X = \X_{s,\lambda} \le \mathcal{O}_s(C_0)$. Note that the pressures only have a weak convergence in $H^{-1}(D)$. Surprisingly, this is sufficient for us to establish the optimal pressure estimates.
Also, it should be pointed out that the homogenized matrix $\overline{A}_\lambda$, depending on $\lambda$ implicitly, satisfies the ellipticity condition (\ref{cond.ellipticity}) uniform in $\lambda$. Moreover, $\overline{A}_\lambda = \widehat{A} + O(\lambda^{-1})$ as $\lambda\to \infty$, where $\widehat{A}$ is the homogenized matrix of a Stokes system; see Section 4.

A crucial principle in our mind to prove (\ref{est.global.rate}) is that (\ref{eq.elasticity.Le}) could be viewed as an elliptic system for relatively small $\lambda$ and could be approximated by a Stokes system for relatively large $\lambda$, due to (\ref{est.exp.intr}). The precise threshold we will use is $\lambda = \e^{-\sigma}$ for some small $\sigma\in (0,1)$ independent of $\e$ and $\lambda$. If $\lambda < \e^{-\sigma}$, the convergence rate follows from the result for elliptic system. In this case, we need to track how the constant $C$ depends on $\lambda$. If $\lambda>\e^{-\sigma}$, with an explicit error, (\ref{est.exp.intr}) may be first reduced to a Stokes system for which a convergence rate may be obtained similarly as elliptic system (Theorem \ref{thm.Stokes.rate}). This process may be described by the diagram in Figure 2.

\begin{figure}[h]
	\centering
	\begin{tikzpicture}
	\matrix (m) [matrix of math nodes, row sep=3em, column sep=3em]
	{ \nabla\cdot A^\e\nabla u^\e_\lambda + \nabla(\lambda\nabla\cdot u^\e_\lambda) &   & \nabla\cdot \overline{A}_\lambda u^0_\lambda + \nabla(\lambda\nabla\cdot u^0_\lambda)  \\
		\nabla\cdot A^\e\nabla v^\e + \nabla p^\e & \nabla\cdot \widehat{A}\nabla v^0 + \nabla p^0 & \nabla\cdot \overline{A}_\lambda \nabla v^0_\lambda + \nabla p^0_\lambda  \\ };
	
	{ [start chain] \chainin (m-1-1);
		
		{ [start branch=B] \chainin (m-2-1)
			[join={node[right,labeled] {
					\lambda > \e^{-\sigma} }
			}];}
		
		\chainin (m-1-3) [join={node[above,labeled] {\lambda < \e^{-\sigma} }
			node[below,labeled] {\text{Homogenization of elliptic system}}
		}];
	}
	{ [start chain] \chainin (m-2-1);
		
		\chainin (m-2-2)
		[join={node[above,labeled] {\text{Homogenization}} 
			node[below,labeled] {\text{Stokes system}}
		}];
		
		\chainin (m-2-3) [join={node[above,labeled] {|\widehat{A} - \overline{A}_\lambda| \lesssim \lambda^{-1}  }}];
		{ [start branch=B] \chainin (m-1-3)
			[join={node[left,labeled] { \text{Error}\lesssim \lambda^{-1} }}];}
	}
	
	\end{tikzpicture}
	\vspace{-1 em}
	\caption{A sketch of the proof of the convergence rate}\label{fig_2}
\end{figure}
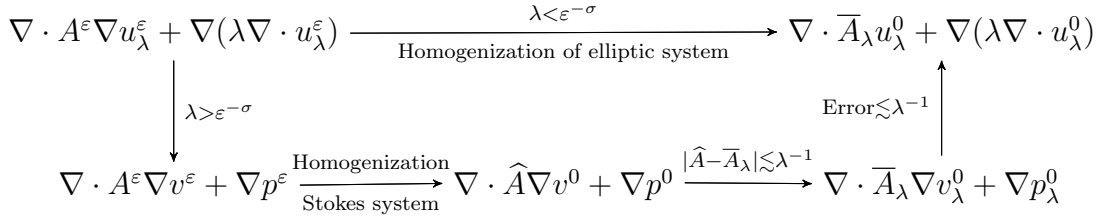

With the explicit convergence rate, we are able to control the excess decays in an iterative argument. Our method follows from Shen's elegant framework in \cite{Shen17}, which originates from \cite{AS16,ArmstrongShen15}. Because of the generalized Caccioppoli inequality (\ref{est.G.Cacc}), however, our argument is much more complicated than the usual elliptic/Stokes system.
%In particular, we have to prove a new iteration lemma, Lemma \ref{lem.iteration}, which generalizes \cite[Lemma 8.5]{Shen17}. Once the large-scale Lipschitz estimate for the displacement is obtained, the pressure estimate follows by another subtle iterative argument by considering the average pressures over dyadic balls.
In the following context, we would like to describe our main idea to tackle the boundary estimate (the interior estimate is similar). In \cite{KP15} and \cite{KP18}, Kenig and Prange introduced the boundary layer correctors to prove the large-scale Lipschitz estimate by a compactness argument in bumpy Lipschitz domains. In this paper, we adopt a more effective quantitative perturbation argument which can be beautifully unified into the aforementioned method of excess decay iteration. The quantified excess we are going to use for the boundary estimate is defined by
\begin{equation}\label{def.Ht.Intro}
\begin{aligned}
H(t) &= \frac{1}{t}\inf_{q\in\R^d} \bigg( \average_{D_{t} } |u^\e_\lambda - (n_t\cdot x)q|^2 \bigg)^{1/2} + \frac{1}{t} \norm{\lambda\nabla\cdot u^\e_\lambda - \average_{D_t} \lambda\nabla\cdot u^\e_\lambda }_{\underline{H}^{-1}(D_t)} \\
& \qquad + \sup_{k,\ell \in [1/4,1]} \bigg| \average_{D_{kt}} \lambda\nabla\cdot u^\e_\lambda - \average_{D_{\ell t}} \lambda\nabla\cdot u^\e_\lambda \bigg|,
\end{aligned}
\end{equation}
where $n_t$ is the unit vector given in Definition \ref{def.C1a.e} which may be understood as an approximate outer normal to the large-scale smooth boundary. We do need this specified directions because of the lack of regularity of the real normal or tangential directions along the Lipschitz boundary at small scales. Also, the particular structure is designed corresponding to the generalized Caccioppoli inequality (\ref{est.G.Cacc}). The advantage of this new structure in (\ref{def.Ht.Intro}) involving the pressure is that we can obtain the Lipschitz estimate (for the displacement) and the pressure estimate simultaneously, which were proved separately in the previous work \cite{GuShen15,GuZhuge17} for the Stokes system.

With the excess quantity given as above, we show that there exists some constant $\theta\in (0,1/4)$ such that for $r\in (\e\X,1)$
\begin{equation}\label{est.ExcessDecay}
H(\theta r) \le \frac{1}{2}H(r) + \text{small error}.
\end{equation}
This eventually leads to the desired estimates by Lemma \ref{lem.iteration}, which is an iteration argument generalizing \cite[Lemma 8.5]{Shen17}. Finally, let us explain the key idea for proving (\ref{est.ExcessDecay}). First of all, Definition \ref{def.C1a.e} implies that for any mesoscopic scale $r\in (\e,1)$, the localized boundary is close to be flat with a controllable error. This fact allows us to construct an approximate solution $v^\e_\lambda$ in a nicer domain with flat boundary, in which the excess decay estimate for the approximate solution could be established in a familiar way. Meanwhile, the errors between the approximate and real solutions could be estimated quantitatively via the Meyers' estimate which holds in any Lipschitz domains. Collecting all these errors, we obtain (\ref{est.ExcessDecay}).

\subsection{Organization of the paper}
The organization of the paper is as follows. In Section 2, we give the definitions of variational solutions, energy estimates and the generalized Caccioppoli inequality. In Section 3, we prove the asymptotic expansion and apply it to the regularity theory in non-homogenization setting. In Section 4 and 5, we establish the algebraic convergence rates for Stokes system and system of elasticity, respectively. Finally, in Section 6 and 7, we prove Theorem \ref{thm.main} and \ref{thm.main1}, respectively.

\textbf{Acknowledgment}. Both of the authors would like to thank Professor Fanghua Lin for the helpful comments after the second author reporting the results of this paper in the SUSTech PDE Workshop in Shenzhen.

\section{Variational Solutions and Energy Estimates}

In this section, we will define the variational solution for the system of elasticity and establish the enery estimates. The classical theory for the Stokes system may be found in \cite{GiaquintaModica82,Temam01}. We begin with an important lemma for the Stokes system.

Denote by $H^{-1}(D)$ and $W^{-1,p'}$ the dual spaces of $H_0^1(D)$ and $W^{1,p}_0(D)$, respectively. Define $L^2_0(D) = \{ f\in L^2(D): \int_D f = 0 \}$.

\begin{lemma}\label{lem.Nabla}
	Let $D$ be a Lipschitz domain and $f\in L^p(D)$. Then
	\begin{equation*}
	C^{-1} \norm{\nabla f}_{W^{-1,p'}(D)} \le \norm{f - \fint_D f}_{L^p(D)} \le C\norm{\nabla f}_{W^{-1,p'}(D)},
	\end{equation*}
	where $C$ depends only on $d$ and $D$. In particular, if $p = 2$,
	\begin{equation*}
	C^{-1} \norm{\nabla f}_{H^{-1}(D)} \le \norm{f - \fint_D f}_{L^2(D)} \le C\norm{\nabla f}_{H^{-1}(D)}.
	\end{equation*}
\end{lemma}
\begin{proof}
	This is the duality of the solvability of the divergence equation $\nabla \cdot u = g\in L^p(D), u\in W_0^{1,p}(D)$, which has been proved in \cite{ADM06} in any bounded John domains (including Lipschitz domains).
\end{proof}

\subsection{Stokes system}
Let $A = A(x):D\mapsto \R^{d^2\times d^2} $ satisfy (\ref{cond.ellipticity}). Consider the general (compressible) Stokes system
\begin{equation}\label{eq.SS}
\left\{
\begin{aligned}
\nabla\cdot (A(x) \nabla v) + \nabla p &= F \qquad &\txt{in }& D, \\
\nabla\cdot v & = g \qquad &\txt{in }& D, \\
v & = f \qquad &\txt{on }& \partial D.
\end{aligned}
\right.
\end{equation}
We say a pair $(v,p)\in H^1(D;\R^d)\times L^2_0(D)$ is a weak solution of (\ref{eq.SS}) if it holds
\begin{equation*}
\int_{D} A(x)\nabla v\cdot \nabla w + \int_{D} p\nabla\cdot w = -\Ag{F, w}
\end{equation*}
for any $w\in H^1_0(D;\R^d)$, and $\nabla\cdot v = g$ in $L^2(D)$, $v - f \in H_0^1(D;\R^d)$.

The following theorem includes the wellposedness of (\ref{eq.SS}) and the energy estimate.

\begin{theorem}\label{thm.Stokes.Energy}
	Let $D$ be a bounded Lipschitz domain and the compatibility condition
	\begin{equation}\label{def.compatibility}
	\int_{D} g \,dx = \int_{\partial D} f\cdot n \,d\sigma
	\end{equation}
	be satisfied. Then the Stokes system (\ref{eq.SS}) has a unique weak solution $(v,p) \in H^1(D;\R^d)\times L^2_0(D)$. Moreover,
	\begin{equation*}
	\norm{\nabla v}_{L^2(D)} + \norm{p}_{L^2(D)} \le C\big( \norm{F}_{H^{-1}(D)} + \norm{g}_{L^2(D)} + \norm{f}_{H^{1/2}(D)} \big),
	\end{equation*}
	where $C$ depends only on $d,\Lambda$ and $D$.
\end{theorem}

\subsection{System of elasticity}
Let $A$ satisfy (\ref{cond.ellipticity}) and consider (\ref{eq.elast}) with constant $\lambda \ge 0$. We say $u_\lambda\in H^1(D;\R^d)$ is the weak solution of (\ref{eq.elast}), if
\begin{equation*}
\int_{D} A(x)\nabla u_\lambda\cdot \nabla w + \int_D \lambda \nabla u_\lambda\cdot \nabla w = -\Ag{F,w}
\end{equation*}
for any $w\in H_0^1(D;\R^d)$ and $u_\lambda - f\in H_0^1(D;\R^d)$. The following theorem gives the energy estimate of the elasticity system with arbitrary $\lambda \ge 0$.

\begin{theorem}\label{lem.energy.elasticity}
	Let $D$ be a bounded Lipschitz domain. Then the elasticity system (\ref{eq.elast}) has a unique weak solution $u_\lambda\in H^1(D;\R^d)$ satisfying
	\begin{equation}\label{est.energy}
	\norm{u_\lambda}_{H^1(D)} + \norm{\lambda \nabla\cdot u_\lambda - \fint_D \lambda \nabla\cdot u_\lambda}_{L^2(D)} \le C\big( \norm{F}_{H^{-1}(D)} + \norm{f}_{H^1(D)} \big),
	\end{equation}
	where $C$ depends only on $d,\Lambda$ and $D$.
\end{theorem}

\begin{proof}
	If we view (\ref{eq.elast}) as an elliptic system with a large ellipticity constant, then the classical Lax-Milgram theorem implies that $u_\lambda \in H^1(D;\R^d)$. It suffices to show (\ref{est.energy}) with constant $C$ independent of $\lambda$. By Lemma \ref{lem.Nabla},
	\begin{equation}\label{est.lambda-Divu}
	\norm{\lambda \nabla\cdot u_\lambda - \average_D \lambda \nabla\cdot u_\lambda }_{L^2(D)} \le C\big( \norm{F}_{H^{-1}(D)} + \norm{\nabla u_\lambda}_{L^2( D)}\big).
	\end{equation}
	Now, by adding a constant, we write the system in (\ref{eq.elast}) as
	\begin{equation*}
	\nabla\cdot (A(x) \nabla u_\lambda) +  \nabla \bigg( \lambda\nabla\cdot u_\lambda - \fint_D \lambda \nabla\cdot u_\lambda \bigg) = F \qquad \text{in } D.
	\end{equation*}
	Integrating this system against $u_\lambda - f$ and using the integration by parts, we arrive at
	\begin{equation}\label{eq.energy.int}
	\begin{aligned}
	& \int_D A(x)\nabla u_\lambda\cdot \nabla u_\lambda + \int_D \bigg( \lambda\nabla\cdot u_\lambda - \fint_D \lambda \nabla\cdot u_\lambda \bigg) \nabla\cdot (u_\lambda - f) \\
	& \qquad = \Ag{F, u_\lambda -f} + \int_D A(x)\nabla u_\lambda\cdot \nabla f.
	\end{aligned}
	\end{equation}
	Substituting
	\begin{equation*}
	\nabla\cdot (u_\lambda - f) = \bigg( \nabla\cdot u_\lambda - \fint_D \nabla\cdot u_\lambda \bigg) + \fint_D \nabla\cdot u_\lambda - \nabla\cdot f
	\end{equation*}
	into the second term of (\ref{eq.energy.int}), we obtain
	\begin{equation}\label{est.energy.1}
	\begin{aligned}
	& \int_D A(x)\nabla u_\lambda\cdot \nabla u_\lambda +  \int_D \lambda \bigg( \nabla\cdot u_\lambda - \fint_D \nabla\cdot u_\lambda \bigg)^2 \\
	& \qquad \le C\norm{F}_{H^{-1}(D)} ( \norm{\nabla u_\lambda}_{L^2(D)} + \norm{\nabla f}_{L^2(D)} ) \\
	& \qquad \qquad + \norm{\lambda \nabla\cdot u_\lambda - \fint_D \lambda \nabla\cdot u_\lambda}_{L^2(D)} \bigg( \bigg|\fint_D \nabla\cdot u_\lambda \bigg| + \norm{\nabla f}_{L^2(D)} \bigg) \\
	& \qquad\qquad + \Lambda \norm{\nabla u_\lambda}_{L^2(D)} \norm{\nabla f}_{L^2(D)} \\
	& \qquad \le \frac{1}{2\Lambda}  \norm{\nabla u_\lambda}_{L^2(D)}^2 + C\big( \norm{F}_{H^{-1}(D)}^2 + \norm{\nabla f}_{L^2(D)}^2 \big),
	\end{aligned}
	\end{equation}
	where in the second inequality, we have used (\ref{est.lambda-Divu}), the Cauchy-Schwarz inequality and the fact
	\begin{equation*}
	\fint_D \nabla\cdot u_\lambda = \frac{1}{|D|} \int_{\partial D} f\cdot n d\sigma = \fint_{D} \nabla\cdot f.
	\end{equation*}
	It follows from the ellipticity condition that
	\begin{equation}\label{est.du2df}
	\norm{\nabla u_\lambda}_{L^2(D)} \le C\big( \norm{F}_{H^{-1}(D)} + \norm{\nabla f}_{L^2(D)} \big).
	\end{equation}
	Finally, the estimate (\ref{est.energy}) follows from the Poincar\'{e} inequality and (\ref{est.lambda-Divu}).
\end{proof}

\subsection{A generalized Caccioppoli inequality}
We introduce the scale-invariant space $\underline{H}^{-1}$. Let $D$ be a Lipschitz domain. Define the scale-invariant $H^{1}$ norm by
\begin{equation*}
\norm{v}_{\underline{H}^1(D)} := |D|^{-1/d} \bigg( \average_{D} |v|^2 \bigg)^{1/2} +\bigg( \average_{D} |\nabla v|^2 \bigg)^{1/2}.
\end{equation*}
Then, we define
\begin{equation*}
\norm{u}_{\underline{H}^{-1}(D)}:= \sup \bigg\{ \average_D u v: v\in H^1_0(D) \text{ and } \norm{v}_{\underline{H}^{1}} \le 1 \bigg\}.
\end{equation*}
Observe that if $u\in L^2(D)$, then
\begin{equation}\label{est.H-1.L2}
\norm{u}_{\underline{H}^{-1}(D)} \le |D|^{1/d} \bigg( \average_{D} |u|^2 \bigg)^{1/2}.
\end{equation}

\begin{theorem}\label{thm.weak.Caccioppoli}
	Let $u_\lambda\in H^1(B_2;\R^d)$ be a weak solution of
	\begin{equation}\label{eq.uB2}
	\nabla\cdot A(x)\nabla u_\lambda + \nabla(\lambda\nabla\cdot u_\lambda) = 0 \quad \text{in } B_2.
	\end{equation}
	Then there exists a constant $C$ depending only on $\Lambda$ and $d$ such that
	\begin{equation*}
	\begin{aligned}
	& \average_{B_1} |\nabla u_\lambda|^2+ { \average_{B_1} |\lambda \nabla\cdot u_\lambda-\average_{B_1} \lambda \nabla \cdot u_\lambda|^2} \\
	&\quad  \le C \average_{B_2} | u_\lambda|^2 + C \norm{\lambda \nabla\cdot u_\lambda - \average_{B_2}\lambda \nabla\cdot u_\lambda  }_{\underline{H}^{-1}(B_2)}^2 \\
	& \qquad + C\sup_{k,\ell \in [1/4,1]} \bigg| \average_{B_{2k}} \lambda\nabla \cdot u_\lambda - \average_{B_{2\ell}} \lambda\nabla \cdot u_\lambda \bigg|^2.
	\end{aligned}
	\end{equation*}
\end{theorem}
\begin{proof}
	Since $\frac{x}{|x|}\cdot u_\lambda \in L^2(B_2)$, the co-area formula implies
	\begin{equation*}
	\int_{1/2}^2 \int_{\partial B_t} \Big(\frac{x}{|x|}\cdot u_\lambda \Big)^2 dt = \int_{1/2}^2 \int_{\partial B_t} (n\cdot u_\lambda)^2 dt \le C\int_{B_2} |u_\lambda|^2,
	\end{equation*}
	where $n$ is the unit outer normal of $\partial B_t$.
	Let
	\begin{equation*}
	T(t) = \average_{B_t} \nabla\cdot u_\lambda.
	\end{equation*}
	Then the divergence theorem yields
	\begin{equation}\label{est.Tt.L2}
	\int_{1/2}^2 T(t)^2 dt \le C\int_{1/2}^2 \int_{\partial B_t} (n\cdot u_\lambda)^2 dt \le C\int_{B_2} |u_\lambda|^2.
	\end{equation}
	%	By Lemma \ref{lem.Nabla},
	%	\begin{equation}\label{est.du-T}
	%	\int_{B_2} |\lambda\nabla\cdot u_\lambda - \lambda T|^2 \le C\int_{B_2} |\nabla u_\lambda|^2.
	%	\end{equation}
	
	Let $\phi\in C_0^\infty(B_t)$ be a nonnegative cut-off function so that $\phi(x) = 1 \text{ for } x\in B_1$ and $\sum_{k = 0}^2 |\nabla^k \phi| \le C$. By inserting the constant $\lambda T(t)$ to the divergence part of the equation, we obtain
	\begin{equation*}
	\nabla\cdot A(x)\nabla u_\lambda + \nabla(\lambda\nabla\cdot u_\lambda - \lambda T(t)) = 0 \quad \text{in } B_2.
	\end{equation*}	
	Now, integrating the above system against $u_\lambda \phi^2$ and using the integration by parts, we obtain
	\begin{equation*}
	\begin{aligned}
	&\int_{B_2}A\nabla u_\lambda \cdot \nabla u_\lambda \phi^2 + \int_{B_2} \lambda (\nabla\cdot u_\lambda - T(t) )^2 \phi^2 \\
	&= - 2\int_{B_2} A\nabla u_\lambda \phi \cdot (u_\lambda\otimes \nabla \phi) - \int_{B_2} \lambda(\nabla\cdot u_\lambda)(u_\lambda\cdot 2\phi\nabla \phi)
	\\
	& = - 2\int_{B_2} A\nabla u_\lambda \phi \cdot (u_\lambda\otimes \nabla \phi) - \int_{B_2} \lambda(\nabla\cdot u_\lambda - T(t) )(u_\lambda\cdot 2\phi\nabla \phi) - \int_{B_2} \lambda(\nabla\cdot u_\lambda - T(t) ) T(t) \phi^2.
	\end{aligned}
	\end{equation*}
	This implies that
	\begin{equation*}
	\begin{aligned}
	& \int_{B_2} |\nabla u_\lambda|^2 \phi^2 \\
	& \le C\int_{B_2} |u_\lambda|^2 + \norm{\lambda \nabla\cdot u_\lambda - \lambda T(t)}_{\underline{H}^{-1}(B_2)}^2 + CT(t)^2  \\
	& \le C\int_{B_2} |u_\lambda|^2 |\nabla \phi|^2 + \norm{\lambda \nabla\cdot u_\lambda - \lambda T(2)}_{\underline{H}^{-1}(B_2)}^2 + C|\lambda T(t) - \lambda T(2)|^2 + CT(t)^2.
	\end{aligned}
	\end{equation*}
	Finally, integrating in $t$ over $[1/2,2]$ and using (\ref{est.Tt.L2}), we obtain the desired estimate for $\nabla u_\lambda$. The estimate for the pressure follows from (\ref{lem.Nabla}).
\end{proof}

\begin{remark}
	By considering rescaling and the fact that $u_\lambda - q$ with any constant $q\in \R^d$ is also a solution, we actually prove that
	\begin{equation}\label{est.WCaccioppoli.Int}
	\begin{aligned}
	& \average_{B_r} |\nabla u_\lambda|^2+ { \average_{B_r} |\lambda \nabla\cdot u_\lambda-\average_{B_r} \lambda \nabla \cdot u_\lambda|^2} \\
	& \quad \le \frac{C}{r^2} \inf_{q\in \R^d} \average_{B_{2r}} | u_\lambda - q|^2 + \frac{C}{r^2} \norm{\lambda \nabla\cdot u_\lambda - \average_{B_{2r}}\lambda \nabla\cdot u_\lambda  }_{\underline{H}^{-1}(B_{2r})}^2 \\
	& \qquad + C\sup_{k,\ell \in [1/4,1]} \bigg| \average_{B_{2kr}} \lambda\nabla \cdot u_\lambda - \average_{B_{2\ell r}} \lambda\nabla \cdot u_\lambda \bigg|^2.
	\end{aligned}
	\end{equation}
	Similarly, one may also show the generalized boundary Caccioppoli inequality as follows. Let $0\in \partial D, D_t = D\cap B_t(0)$ and $\Delta_t = \partial D\cap B_t(0)$.Let $u_\lambda\in H^1(D_2;\R^d)$ be a weak solution of
	\begin{equation}\label{eq.bdry.uB2}
	\left\{
	\begin{aligned}
	\nabla\cdot A(x)\nabla u_\lambda + \lambda\nabla(\nabla\cdot u_\lambda)&= 0 \qquad &\txt{in }& D_{2r}, \\
	u_\lambda & = 0 \qquad &\txt{on }& \Delta_{2r}.
	\end{aligned}
	\right.
	\end{equation}
	Then, we have
	\begin{equation}\label{est.WCacc.bdry}
	\begin{aligned}
	& \average_{D_r} |\nabla u_\lambda|^2+ { \average_{D_r} |\lambda \nabla\cdot u_\lambda-\average_{D_r} \lambda \nabla \cdot u_\lambda|^2} \\
	& \quad \le \frac{C}{r^2} \average_{D_{2r}} | u_\lambda|^2 + \frac{C}{r^2} \norm{\lambda \nabla\cdot u_\lambda - \average_{D_{2r}}\lambda \nabla\cdot u_\lambda  }_{\underline{H}^{-1}(D_{2r})}^2 \\
	& \qquad + C\sup_{t\in [1/4,1]} \bigg| \average_{D_{2kr}} \lambda\nabla \cdot u_\lambda - \average_{D_{2\ell r}} \lambda\nabla \cdot u_\lambda \bigg|^2.
	\end{aligned}
	\end{equation}
	The generalized Caccioppoli inequalities (\ref{est.WCaccioppoli.Int}) and (\ref{est.WCacc.bdry}) will be useful for us.
\end{remark}

\section{Asymptotic Behaviors and General Regularity}
It is well-known in physics and numerical analysis that the solution $u_\lambda$ of 
\begin{equation}\label{eq.elast1}
\left\{
\begin{aligned}
\nabla\cdot (A(x) \nabla u_\lambda) +  \lambda \nabla (\nabla\cdot u_\lambda) &= F \qquad &\txt{in } &D, \\
u_\lambda  &= f \qquad &\txt{on } &\partial D,
\end{aligned}
\right.
\end{equation}
converges to the solution of a Stokes system as constant $\lambda$ approaches infinity. This property allows people to design efficient numerical algorithm to solve the system of nearly incompressible elasticity \cite{Temam01}. In this section, we will prove a complete asymptotic expansion in terms of the solutions of certain iterative Stokes systems and use it to study the uniform regularity of the system of nearly incompressible elasticity.

\subsection{A proof of asymptotic expansion}
%We start with the proof of the zero-order term of the expansion, which is the solution of the limiting system (\ref{eq.Stokes}), as $\lambda\to \infty$.
To describe the limiting system of (\ref{eq.elast1}), we define
\begin{equation}\label{def.Agf}
\Ag{f}_D = \frac{1}{|D|} \int_{\partial D} f\cdot n d\sigma.
\end{equation}
Roughly speaking, the quantity $\Ag{f}_D$ represents the averaging volume change of the material body. For nearly incompressible materials, $\Ag{f}_D$ is small under a mild physical condition, and could be large under high pressure.

%Also, the pressure part $\lambda \nabla\cdot u_\lambda$ converges to $p_0$, which indicates that $\nabla\cdot u_\lambda$ vanishes as $\lambda \to \infty$.

\begin{lemma}\label{lem.lambda.rate}
	Let $D$ be a Lipschitz domain and $u_\lambda$ the weak solution of (\ref{eq.elast1}). Let $v_0$ be the weak solution of 
	\begin{equation}\label{eq.Stokes}
	\left\{
	\begin{aligned}
	\nabla\cdot (A(x) \nabla v_0) + \nabla p_0 &= F \qquad &\txt{in }& D, \\
	\nabla\cdot v_0 & = \Ag{f}_D \qquad &\txt{in }& D, \\
	v_0 & = f \qquad &\txt{on }& \partial D.
	\end{aligned}
	\right.
	\end{equation}
	Then
	\begin{equation}\label{est.0order}
	\norm{u_\lambda - v_0}_{H^1(D)} + \norm{\lambda \nabla\cdot u_\lambda - \average_D \lambda \nabla\cdot u_\lambda - p_0 }_{L^2(D)} \le C\lambda^{-1} \big( \norm{F}_{H^{-1}(D)} + \norm{f}_{H^{1/2}(\partial D)}\big),
	\end{equation}
	where $C$ depends only on $d, \Lambda$ and $D$.
\end{lemma}

\begin{proof}
	The proof is well-known. We provide a proof for completeness. By the divergence theorem, $\Ag{f}_D =\average_D \nabla\cdot u_\lambda $. Let $w_0 = u_\lambda - v_0$ and consider the Stokes system for $w_0$
	\begin{equation}\label{eq.0order}
	\left\{
	\begin{aligned}
	\nabla\cdot (A(x) \nabla w_0) + \nabla (\lambda \nabla\cdot w_0 -p_0) &= 0 \qquad &\txt{in }& D, \\
	\nabla\cdot w_0 & = \nabla\cdot u_\lambda - \average_D \nabla\cdot u_\lambda \qquad &\txt{in }& D, \\
	w_0 & = 0 \qquad &\txt{on }& \partial D.
	\end{aligned}
	\right.
	\end{equation}
	As a consequence, it follows from Theorem \ref{thm.Stokes.Energy} and Lemma \ref{lem.energy.elasticity} that
	\begin{equation*}
	\begin{aligned}
	&\norm{u_\lambda - v_0}_{H^1(D)} + \norm{\lambda \nabla\cdot u_\lambda - \average_D \lambda \nabla\cdot u_\lambda - p_0}_{L^2(D)} \\
	& \le C\norm{ \nabla\cdot u_\lambda - \average_D \nabla\cdot u_\lambda }_{L^2(D)} \\
	& \le C\lambda^{-1} \big( \norm{F}_{H^{-1}(D)} + \norm{f}_{H^{1/2}(\partial D)}\big).
	\end{aligned}
	\end{equation*}
	This completes the proof.
\end{proof}

%The following is the main theorem of this section. We give a complete asymptotic expansion for the solution $u_\lambda$ as well as $\lambda \nabla\cdot u_\lambda$ in terms of the solutions of a sequence of iterated Stokes systems.

Divided by $\lambda$, (\ref{est.0order}) yields
\begin{equation}\label{est.1order.helper}
\nabla\cdot u_\lambda - \average_D \nabla\cdot u_\lambda - \lambda^{-1} p_0 \le O(\lambda^{-2}).
\end{equation}
Observe that $\lambda^{-1}p_0$ may be used to correct system (\ref{eq.0order}) with a higher-order error $O(\lambda^{-2})$ and thus the first-order term may be determined. It turns out that iterating this argument leads to a complete asymptotic expansion in $\lambda$ for the solution $u_\lambda$. 

\begin{theorem}\label{thm.expansion}
	Let $D$ be a bounded Lipschitz domain and $A$ satisfy (\ref{cond.ellipticity}) and (\ref{cond.symmetry}). Suppose $u_\lambda$ is the weak solution of (\ref{eq.elast1}) with $F\in H^{-1}(D;\R^d)$ and $f\in H^{1/2}(D;\R^d)$. Then there exists $C_0>0$ depending only on $d,\Lambda$ and $D$, such that if $\lambda>C_0$,
	\begin{equation}\label{eq.asym.expansion}
	u_\lambda = \sum_{k = 0}^\infty \lambda^{-k} v_k \quad \text{in } H^1 \quad \txt{and} \quad \lambda \nabla\cdot u_\lambda - \lambda \Ag{f}_D = \sum_{k=0}^\infty \lambda^{-k} p_k \quad \text{in } L^2_0.
	\end{equation}
	where $(v_0, p_0)$ is the weak solution of (\ref{eq.Stokes}) and $(v_k,p_k)$ with $k\ge 1$ are the solutions of a sequence of iterative Stokes systems
	\begin{equation}\label{eq.iterated.Stokes}
	\left\{
	\begin{aligned}
	\nabla\cdot (A(x) \nabla v_k) + \nabla p_k &= 0 \qquad &\txt{in }& D, \\
	\nabla\cdot v_k & = p_{k-1} \qquad &\txt{in }& D, \\
	v_k & = 0 \qquad &\txt{on }& \partial D.
	\end{aligned}
	\right.
	\end{equation}
\end{theorem}

\begin{proof}
	The theorem is proved by induction with Lemma \ref{lem.lambda.rate} being the base case. Let $w_\ell = u_\lambda - \sum_{k=0}^\ell \lambda^{-k} v_k$ and $\pi_\ell = \lambda\nabla\cdot u_\lambda-\lambda \Ag{f}_D - \sum_{k=0}^\ell \lambda^{-k} p_k$. We prove that
	\begin{equation}\label{est.wl.pil}
	\norm{w_{\ell}}_{H^1(D)} + \norm{\pi_{\ell}}_{L^2(\Omega)} \le C_0 C_1^{\ell} \lambda^{-\ell-1},
	\end{equation}
	where $C_0$ depends only on $d,\Lambda, D$ and the data $(F,f)$, and $C_1$ depends only on $d,\Lambda$ and $D$. Definitely, if we let $\lambda>\lambda_0:=C_0$, then the right-hand side of (\ref{est.wl.pil}) converges to zero as $\ell\to\infty$, which leads to (\ref{eq.asym.expansion}).
	
	To show (\ref{est.wl.pil}), we first consider the base case $\ell = 0$. Note that $(w_0,\pi_0) = (u_\lambda - v_0, \lambda\nabla\cdot u_\lambda-\lambda\Ag{f}_D-p_0)$ satisfies
	\begin{equation*}
	\left\{
	\begin{aligned}
	\nabla\cdot (A(x) \nabla w_0) + \nabla \pi_0 &= 0 \qquad &\txt{in }& D, \\
	\nabla\cdot w_0 & = \nabla\cdot u_\lambda - \Ag{f}_D \qquad &\txt{in }& D, \\
	w_0 & = 0 \qquad &\txt{on }& \partial D.
	\end{aligned}
	\right.
	\end{equation*}
	Thus, Lemma \ref{lem.lambda.rate} implies
	\begin{equation*}
	\norm{w_0}_{H^1(D)} + \norm{\pi_0}_{L^2(\Omega)} \le C_0\lambda^{-1},
	\end{equation*}
	where $C_0$ depends only on $d,\Lambda, D$ and the data $(F,f)$.
	
	To clearly see our idea, let us work out the first iteration step for $(w_1,\pi_1) = (w_0,\pi_0) - \lambda^{-1}(v_1, p_1)$. By the definition of $(v_1, p_1)$ in (\ref{eq.iterated.Stokes}), one has
	\begin{equation*}
	\left\{
	\begin{aligned}
	\nabla\cdot (A(x) \nabla w_1) + \nabla \pi_1 &= 0 \qquad &\txt{in }& D, \\
	\nabla\cdot w_1 & = \lambda^{-1} \pi_0 \qquad &\txt{in }& D, \\
	w_1 & = 0 \qquad &\txt{on }& \partial D.
	\end{aligned}
	\right.
	\end{equation*}
	Then the energy estimate for the Stoke system yields
	\begin{equation*}
	\norm{w_1}_{H^1(D)} + \norm{\pi_1}_{L^2(\Omega)} \le C_1\lambda^{-1} \norm{\pi_0}_{L^2(D)} \le C_0 C_1 \lambda^{-2},
	\end{equation*}
	where $C_1$ depends only on $d,\Lambda$ and $D$.
	
	In general, assume that $(w_\ell, \pi_\ell) = (w_{\ell-1}, \pi_{\ell-1}) - \lambda^{-\ell} (v_\ell,p_\ell) $ with $\ell \ge 1$ satisfies
	\begin{equation}\label{eq.iterated.ell}
	\left\{
	\begin{aligned}
	\nabla\cdot (A(x) \nabla w_\ell) + \nabla \pi_\ell &= 0 \qquad &\txt{in }& D, \\
	\nabla\cdot w_\ell & = \lambda^{-1} \pi_{\ell-1} \qquad &\txt{in }& D, \\
	w_\ell & = 0 \qquad &\txt{on }& \partial D,
	\end{aligned}
	\right.
	\end{equation}
	and
	\begin{equation}\label{est.wpi.ell}
	\norm{w_\ell}_{H^1(D)} + \norm{\pi_\ell}_{L^2(\Omega)} \le C_0 C_1^{\ell} \lambda^{-\ell-1}.
	\end{equation}
	Now, let $(w_{\ell+1}, \pi_{\ell+1}) = (w_{\ell}, \pi_{\ell}) - \lambda^{-\ell-1} (v_{\ell+1},p_{\ell+1}) $. In view of (\ref{eq.iterated.ell}) and (\ref{eq.iterated.Stokes}), we see that $(w_{\ell+1}, \pi_{\ell+1})$ is the solution of
	\begin{equation*}
	\left\{
	\begin{aligned}
	\nabla\cdot (A(x) \nabla w_{\ell+1}) + \nabla \pi_{\ell+1} &= 0 \quad &\txt{in }& D, \\
	\nabla\cdot w_{\ell+1} & = \lambda^{-1} \pi_{\ell} \quad &\txt{in }& D, \\
	w_{\ell+1} & = 0 \quad &\txt{on }& \partial D,
	\end{aligned}
	\right.
	\end{equation*}
	where we have used the construction $\pi_\ell = \pi_{\ell-1} - \lambda^{-\ell} p_\ell$. By (\ref{est.wpi.ell}), the last system implies
	\begin{equation*}
	\norm{w_{\ell+1}}_{H^1(D)} + \norm{\pi_{\ell+1}}_{L^2(\Omega)} \le C_1\lambda^{-1} \norm{\pi_{\ell}}_{L^2(D)} \le C_0 C_1^{\ell+1} \lambda^{-\ell-2}.
	\end{equation*}
	This proves (\ref{est.wl.pil}).
\end{proof}

\subsection{Global estimates}
The asymptotic expansion in Theorem \ref{thm.expansion} is a powerful tool to study the global regularity of the system of nearly incompressible elasticity, provided the same regularity holds for the Stokes systems. The result may be described in an abstract setting. Let $X_0(D;\R^d)$ be a subspace of $L^1(D;\R^d)$ endowed with norm $\norm{\cdot}_{X_0}$, namely,
\begin{equation}
X_0(D;\R^d) = \{ f\in L^1(D;\R^d): \norm{f}_{X_0} < \infty \}.
\end{equation}
Define $X_1(D;\R^d) = \{ f\in X_0(D;\R^d): \nabla f\in X_0(D;\R^{d\times d}) \}$ and $\norm{f}_{X_1} = \norm{f}_{X_0} + \norm{\nabla f}_{X_0}$.

\begin{theorem}\label{thm.abs.expansion}
	Let $D$ be a Lipschitz domain. Suppose there exists a constant $M>0$ such that for any $h\in X_0(D;\R^{d\times d}), g\in X_0(D;\R)$ and $f\in X_1(D;\R^d)$, the solution $(v,p)$ of the Stokes system
	\begin{equation}
	\left\{
	\begin{aligned}
	\nabla\cdot (A(x) \nabla v) + \nabla p &= \nabla\cdot h \qquad &\txt{in }& D, \\
	\nabla\cdot v & = g \qquad &\txt{in }& D, \\
	v & = f \qquad &\txt{on }& \partial D
	\end{aligned}
	\right.
	\end{equation}
	satisfies
	\begin{equation}\label{est.abs.Stokes}
	\norm{v}_{X_1} + \norm{p}_{X_0} \le M(\norm{f}_{X_1} + \norm{h}_{X_0} + \norm{g}_{X_0}).
	\end{equation}
	Then if $\lambda > 2M$, and $u_\lambda$ is the weak solution of
	\begin{equation}\label{eq.uhf}
	\left\{
	\begin{aligned}
	\nabla\cdot (A(x) \nabla u_\lambda) +  \nabla (\lambda \nabla\cdot u_\lambda) &= \nabla\cdot h \qquad &\txt{in } &D, \\
	u_\lambda  &= f \qquad &\txt{on } &\partial D,
	\end{aligned}
	\right.
	\end{equation}
	then
	\begin{equation}
	\norm{u_\lambda}_{X_1} + \lambda \norm{\nabla\cdot u_\lambda - \Ag{f}_D }_{X_0} \le 2M(\norm{f}_{X_1} + \norm{h}_{X_0}).
	\end{equation}
\end{theorem}

Theorem \ref{thm.abs.expansion} may be proved directly by using Theorem \ref{thm.expansion}.  Note that Theorem \ref{thm.abs.expansion} applies only for $\lambda\ge 2M$. However, for $0\le \lambda\le 2M$, the system (\ref{eq.uhf}) is the classical elliptic system whose regularity theory is well-understood.

%Moreover, all of these results may stated in a local form, such as interior or boundary local Meyers' estimate, $C^{1,\alpha}$ estimate, etc.
Particularly, Theorem \ref{thm.abs.expansion} applies to $X_0 = W^{k,p}, X_1 = W^{k+1,p}$ or $X_0 = C^{k,\alpha}, X_1 = C^{k+1,\alpha}$ for $p\in (1,\infty)$ and $k\in \N, \alpha\in (0,1)$.
To give a concrete example, in the following, we apply Theorem \ref{thm.abs.expansion} to show the $W^{1,p}$ (or $C^\alpha$) regularity of (\ref{eq.uhf}) in the context of periodic homogenization, for the Stokes system has been well-studied in \cite{GuShen15,Xu17,GuXu17,GuZhuge17}. Precisely, we assume that the coefficient pair $(A,\lambda)$ satisfies (\ref{cond.ellipticity}), (\ref{cond.incompr}) and the following assumptions:
\begin{itemize}
	\item Periodicity:
	\begin{equation*}
	A(x+z) = A(x),\ \lambda(x+z) =\lambda(x), \quad \text{for any } x\in \R^d, z\in \Z^d.
	\end{equation*}
	
	\item Uniform VMO condition:
	\begin{equation}
	\lim_{r\to 0} \sup_{y\in [0,1]^d} \bigg( \fint_{B_r(y)} | A(x) - \fint_{B_r(y)} A |dx + \fint_{B_r(y)} | \lambda(x) - \fint_{B_r(y)} \lambda |dx \bigg)= 0.
	\end{equation}
\end{itemize}
Let $D$ be a bounded $C^1$ domain and consider the system
\begin{equation}\label{eq.ue.hf}
\left\{
\begin{aligned}
\nabla\cdot (A^\e \nabla u_\lambda^\e) +  \nabla (\lambda^\e \nabla\cdot u_\lambda^\e) &= \nabla\cdot h \qquad &\txt{in } &D, \\
u_\lambda^\e  &= f \qquad &\txt{on } &\partial D,
\end{aligned}
\right.
\end{equation}
where $A^\e(x) = A(x/\e)$ and $\lambda^\e(x) = \lambda(x/\e)$.

\begin{theorem}\label{thm.W1p.periodic}
	Let $D$ and $(A,\lambda)$ satisfy the above assumptions. Suppose $p\in (1,\infty), h\in L^p(D;\R^{d\times d})$ and $f\in W^{1,p}(D;\R^d)$. Then the weak solution of (\ref{eq.ue.hf}) satisfies
	\begin{equation}
	\norm{u_\lambda^\e}_{W^{1,p}(D)} + \norm{\lambda^\e\nabla\cdot  u_\lambda^\e - \lambda_0 \Ag{f}_D}_{L^p(D)} \le C(\norm{f}_{W^{1,p}(D)} + \norm{h}_{L^p(D)}),
	\end{equation}
	where $C$ depends only on $d,D,\Lambda$ and the VMO modulus of $(A,\lambda)$. In particular, $C$ is independent of $\lambda_0$ and $\e$.
\end{theorem}
\begin{proof}
	This follows from \ref{thm.abs.expansion} and   \cite[Theorem 1.4]{GuShen15}.
\end{proof}

\begin{remark}
	We should emphasize that the asymptotic expansion does not apply to the Lipschitz estimate of $u_\lambda^\e$ for (\ref{eq.ue.hf}), because the uniform boundedness of $|\nabla u_\lambda^\e|$ is not preserved under the iterative Stokes system (\ref{eq.iterated.Stokes}). In other words, (\ref{est.abs.Stokes}) is generally wrong for $X_0 = L^\infty$. This failure is due to a well-known fact that the singular integral (or Riesz transform) is not bounded in $L^\infty$. In the rest of the paper, we will develop a new approach, using the regularity theory for Stokes system only in the case of constant coefficients, to resolve the problem.
\end{remark}

\subsection{Local estimates}
In this subsection, we will prove two local estimates, i.e., the Meyers' estimate and $C^{1,\alpha}$ estimate, which are crucial in the study of quantitative homogenization. For the system of elasticity with large $\lambda$, the local estimates are technically more involving.

To show the local $C^{1,\alpha}$ regularity of the system of elasticity, we need the same regularity for the Stokes system. For convenience, define
\begin{equation*}
[f]_{C^{\alpha}(D)} := \sup_{x,y\in D} \frac{|f(x) - f(y)|}{|x-y|^\alpha}.
\end{equation*}
\begin{theorem}[\cite{GiaquintaModica82}]\label{thm.Stokes.C1a}
	Let $D$ be a bounded $C^{1,\alpha}$ domain and $A$ be constant. There exits $C>0$ depending only on $d,\Lambda$ and $D$ such that if $(v,p)$ is a weak solution of
	\begin{equation}\label{eq.const.Stokes}
	\left\{
	\begin{aligned}
	\nabla\cdot (A \nabla v) + \nabla p &= \nabla\cdot h \qquad &\txt{in }& D_{2r}, \\
	\nabla\cdot v & = g \qquad &\txt{in }& D_{2r},\\
	v & = f \qquad &\txt{on }& \Delta_{2r},
	\end{aligned}
	\right.
	\end{equation}
	with $h\in C^{\alpha}(D_{2r};\R^{d\times d}), g\in C^\alpha(D_{2r})$ and $f\in C^{1,\alpha}(\Delta_{2r};\R^d)$, then $(\nabla v, p) \in C^{\alpha}(D_r;\R^{d}\times \R)$ and
	\begin{equation*}
	\begin{aligned}
	&[\nabla v]_{C^\alpha(D_r)} + [p]_{C^\alpha(D_r)} \\
	&\quad \le C\bigg( \frac{1}{r^\alpha} \bigg( \average_{D_{2r}} |\nabla v|^{2} \bigg)^{1/2} + [h]_{C^\alpha(D_{2r})} + [g]_{C^\alpha(D_{2r})} + [\nabla f]_{C^\alpha(\Delta_{2r})} \bigg).
	\end{aligned}
	\end{equation*}
\end{theorem}

As a corollary, we may show the following lemma.
\begin{lemma}\label{lem.C1a}
	Let the same conditions as (\ref{thm.Stokes.C1a}) hold. Let $(v,p)$ be the solution of (\ref{eq.const.Stokes}) with $r = 1$. Then for any $s\in (0,2)$
	\begin{align*}
	[\nabla v]_{C^\alpha(D_s)} + [p]_{C^\alpha(D_s)} &\le \frac{C}{(2-s)^{d/2+\alpha}} \bigg( \average_{D_{2}} |\nabla v|^{2} \bigg)^{1/2} \\
	& \quad + C\Big( [h]_{C^\alpha(D_{(2+s)/2})} + [g]_{C^\alpha(D_{(2+s)/2})} + [\nabla f]_{C^\alpha(\Delta_{(2+s)/2})} \Big).
	\end{align*}
\end{lemma}
\begin{proof}
	The case $s\in (0,1)$ is obvious. Fix $s\in (1,2)$. Let $x,y\in D_s$. Then $\text{dist}(x,\partial D_{(2+s)/2}\setminus \Delta_2) \ge (2-s)/2$ and $\text{dist}(y,\partial D_{(2+s)/2}\setminus \Delta_2) \ge (2-s)/2$. Now, if $|x-y|<(2-s)/4$, we can find a ball $\widetilde{B}$ with radius $(2-s)/8$ containing both $x$ and $y$ so that $2\widetilde{B} \subset B_{(2+s)/2}$. It follows from Theorem \ref{thm.Stokes.C1a} that
	\begin{equation*}
	\begin{aligned}
	&[\nabla v]_{C^\alpha(\widetilde{B}\cap D_2)} + [p]_{C^\alpha(\widetilde{B}\cap D_2)} \\
	& \le C\bigg( \frac{1}{(2-s)^\alpha} \bigg( \average_{2\widetilde{B}\cap D_2} |\nabla v|^{2} \bigg)^{1/2} + [h]_{C^\alpha(2\widetilde{B}\cap D_2)} + [g]_{C^\alpha(2\widetilde{B}\cap D_2)} + [\nabla f]_{C^\alpha(2\widetilde{B}\cap \Delta_2)} \bigg) \\
	& \le C\bigg( \frac{1}{(2-s)^{d/2+\alpha}} \bigg( \average_{ D_2} |\nabla v|^{2} \bigg)^{1/2} + [h]_{C^\alpha(D_{(2+s)/2})} + [g]_{C^\alpha(D_{(2+s)/2})} + [\nabla f]_{C^\alpha(\Delta_{(2+s)/2})} \bigg).
	\end{aligned}
	\end{equation*}
	
	Now, assume $|x-y|>(2-s)/4$. Because $D_{(2+s)/2}$ is a Lipschitz domain, we can find a sequence of balls $\{ B_{r_k}(x_k)\}_{k=M}^N$ (a Harnack chain) connecting the points $x$ and $y$. Moreover, the radius $r_k$ are comparable to $\theta^k(2-s)$ for some $\theta>1$ and
	\begin{equation*}
	2B_{r_k}(x_k) \subset B_{(2-s)/2} \qquad \text{for all } k = M,M+1,\cdots, N.
	\end{equation*}
	The largest radius is comparable to $|x-y|$, i.e., $\theta^N(2-s) \simeq |x-y|$. The idea is that we apply Theorem \ref{thm.Stokes.C1a} on each ball $B_{r_k}(x_k)$ and then estimate $|\nabla v(x) - \nabla v(y)|$ and $|p(x) - p(y)|$ by connecting a path through the chain of balls. Precisely, we have
	\begin{equation*}
	\begin{aligned}
	& |\nabla v(x) - \nabla v(y)| + |p(x) - p(y)| \\
	& \le C\sum_{k}  \bigg( \average_{2B_{r_k}(x_k)\cap D_2} |\nabla v|^{2} \bigg)^{1/2} \\
	& \qquad + C\sum_{k} r_k^\alpha  \Big( [h]_{C^\alpha(2B_{r_k}(x_k)\cap D_2)} + [g]_{C^\alpha(2B_{r_k}(x_k)\cap D_2)} + [\nabla f]_{C^\alpha(2B_{r_k}(x_k)\cap \Delta_2)} \Big) \\
	& \le \frac{C}{(2-s)^{d/2}} \bigg( \average_{D_{2}} |\nabla v|^{2} \bigg)^{1/2} \\
	& \qquad + C|x-y|^\alpha \Big( [h]_{C^\alpha(D_{(2+s)/2})} + [g]_{C^\alpha(D_{(2+s)/2})} + [\nabla f]_{C^\alpha(\Delta_{(2+s)/2})} \Big).
	\end{aligned}
	\end{equation*}
	This implies the desired estimate since $|x-y|>(2-s)/4$.
\end{proof}

\begin{theorem}\label{thm.C1a}
	Let $D$ be a bounded $C^{1,\alpha}$ domain and $A$ be constant. There exists $C>0$ depending only on $d,\Lambda$ and $D$ such that if $u_\lambda$ is a weak solution of
	\begin{equation}
	\left\{
	\begin{aligned}
	\nabla\cdot (A \nabla u_\lambda) +  \nabla (\lambda \nabla\cdot u_\lambda) &= \nabla\cdot h \qquad &\txt{in } &D_{2r}, \\
	u_\lambda  &= f \qquad &\txt{on } &\partial D_{2r},
	\end{aligned}
	\right.
	\end{equation}
	with $h\in C^{\alpha}(D_{2r};\R^{d\times d})$ and $f\in C^{1,\alpha}(\Delta_{2r};\R^d)$, then $u_\lambda \in C^{1,\alpha}(D_r;\R^{d})$ and
	\begin{equation*}
	\begin{aligned}
	&[\nabla u_\lambda]_{C^\alpha(D_r)} + [\lambda\nabla\cdot u_\lambda]_{C^\alpha(D_r)} \\
	& \qquad \le C\bigg\{ \frac{1}{r^\alpha} \bigg( \average_{D_{2r}} |\nabla u_\lambda|^{2} \bigg)^{1/2} + [h]_{C^\alpha(D_{2r})} + [\nabla f]_{C^\alpha(\Delta_{2r})} \bigg\}.
	\end{aligned}
	\end{equation*}
\end{theorem}
\begin{proof}
	It suffices to show the estimate for $\lambda>2C$ for some constant $C>0$, while the case with small $\lambda$ follows from the classical Schauder estimate for the elliptic system. 
	
	%We will fully use Theorem \ref{thm.expansion}. 
	
	By rescaling and normalization, we assume $r = 1$ and
	\begin{equation*}
	\bigg( \average_{D_{2}} |\nabla u_\lambda|^{2} \bigg)^{1/2} + [h]_{C^\alpha(D_{2})} + [\nabla f]_{C^\alpha(\Delta_{2})} = 1.
	\end{equation*}
	Applying Theorem \ref{thm.expansion} in $D_2$, we may write
	\begin{equation*}
	u_\lambda = \sum_{k = 0}^\infty \lambda^{-k} v_k \quad \text{in } H^1 \quad \txt{and} \quad \lambda \nabla\cdot u_\lambda - \lambda \Ag{f}_D = \sum_{k=0}^\infty \lambda^{-k} p_k \quad \text{in } L^2.
	\end{equation*}
	where $(v_0,p_0)$ solves (\ref{eq.Stokes}) with $F = \nabla\cdot h$ and $(v_k,p_k)$ with $k\ge 1$ solves (\ref{eq.iterated.Stokes}). By the energy estimate, it is not hard to see
	\begin{equation}\label{est.vkpk.L2}
	\norm{\nabla v_k}_{L^2(D_2)} + \norm{p_k}_{L^2(D_2)} \le C^{k+1}.
	\end{equation}
	
	Now, applying Lemma \ref{lem.C1a} to (\ref{eq.Stokes}), we have
	\begin{equation*}
	[\nabla v_0] _{C^\alpha(D_s)} + [p_0]_{C^\alpha(D_s)} \le \frac{C^2}{(2-s)^{d/2+\alpha}}.
	\end{equation*}
	On the other hand, applying Lemma \ref{lem.C1a} to (\ref{eq.iterated.Stokes}) and using (\ref{est.vkpk.L2}), we obtain, for any $s\in (0,2)$,
	\begin{equation*}
	[\nabla v_k] _{C^\alpha(D_s)} + [p_k]_{C^\alpha(D_s)} \le \frac{C^{k+2}}{(2-s)^{d/2+\alpha}} + C [p_{k-1}]_{C^\alpha(D_{(2+s)/2})}.
	\end{equation*}
	Next, without much difficulty, we may prove by induction that
	\begin{equation*}
	[\nabla v_k] _{C^\alpha(D_s)} + [p_k]_{C^\alpha(D_s)} \le \frac{(k+1)C^{k+2}}{(2-s)^{d/2+\alpha}} \qquad \text{for all } k\ge 0, s\in (0,2).
	\end{equation*}
	This implies the desired estimate if $\lambda$ is large. Indeed, if $\lambda > 2C$,
	\begin{align*}
	[\nabla u_\lambda]_{C^\alpha(D_s)} + [\lambda\nabla\cdot u_\lambda]_{C^\alpha(D_s)} &\le \sum_{k=0}^\infty \lambda^{-k} \Big( [\nabla v_k] _{C^\alpha(D_s)} + [p_k]_{C^\alpha(D_s)} \Big) \\
	&\le \sum_{k=0}^\infty (2C)^{-k} \frac{(k+1)C^{k+2}}{(2-s)^{d/2+\alpha}} \\
	& \le \frac{C}{(2-s)^{d/2+\alpha}}.
	\end{align*}
	This implies the desired estimate by setting $s = 1$.
\end{proof}

Next, we are going to show the Meyers' estimate for the system of elasticity which is independent of $\lambda$. To this end, let us recall the local Meyers' estimate of Stokes system with bounded measurable coefficients.

\begin{theorem}[Meyers' estimate for Stokes system]\label{thm.Meyers.Stokes}
	Let $D$ be a Lipschitz domain. There exists $p_0>2$, depending only on $d,\Lambda$ and $\text{Lip}(D)$ so that if $(v,p)$ is a weak solution of
	\begin{equation*}
	\left\{
	\begin{aligned}
	\nabla\cdot (A(x) \nabla v) + \nabla p &= \nabla\cdot h \qquad &\txt{in }& D_{2r}, \\
	\nabla\cdot v & = g \qquad &\txt{in }& D_{2r},\\
	v & = f \qquad &\txt{on }& \Delta_{2r},
	\end{aligned}
	\right.
	\end{equation*}
	with  $h\in W^{1,p_0}(D_{2r};\R^{d\times d})$, $g\in L^{p_0}(D_{2r})$ and $f\in W^{1,p_0}(D_{2r};\R^d)$, then $(v,p)\in W^{1,p_0}(D_r;\R^d) \times L^{p_0}(D_r)$ and
	\begin{equation*}
	\begin{aligned}
	& \bigg( \average_{D_r} |\nabla v|^{p_0} \bigg)^{1/p_0} + \bigg( \average_{D_r} |p - \average_{D_r} p |^{p_0} \bigg)^{1/p_0}
	\\
	& \quad \le C\bigg\{  \bigg( \average_{D_{2r}} |\nabla v|^{2} \bigg)^{1/2} + \bigg( \average_{D_{2r}} |h|^{p_0} \bigg)^{1/p_0}  + \bigg( \average_{D_{2r}} |g|^{p_0} \bigg)^{1/p_0} + \bigg( \average_{D_{2r}} |\nabla f|^{p_0} \bigg)^{1/p_0} \bigg\},
	\end{aligned}
	\end{equation*}
	where $C$ depends only on $d,\Lambda$ and $\text{Lip}(D)$.
\end{theorem}

Of course, the Meyers' estimate for the system of elasticity could be proved by the similar strategy as Theorem \ref{thm.C1a}. Here we will use an alternative approach which takes advantage of a real variable perturbation argument by Shen \cite[Chapter 3]{ShenBook18}.

\begin{theorem}\label{thm.Meyers.Elasticity}
	Let $D$ be a Lipschitz domain. There exists $q_0>2$, depending only on $d,\Lambda$ and $D$ so that if $u_\lambda$ is the weak solution of
	\begin{equation}\label{eq.local.Elasticity}
	\left\{
	\begin{aligned}
	\nabla\cdot (A(x) \nabla u_\lambda) + \nabla(\lambda\nabla\cdot u_\lambda) &= \nabla\cdot h \qquad &\txt{in }& D_{2r}, \\
	u_\lambda & = f \qquad &\txt{on }& \Delta_{2r},
	\end{aligned}
	\right.
	\end{equation}
	with  $h\in W^{1,q_0}(D_{2r};\R^{d\times d})$ and $f\in W^{1,q_0}(D_{2r};\R^d)$, then $u_\lambda \in W^{1,q_0}(D_r;\R^d)$ and
	\begin{equation*}
	\begin{aligned}
	& \bigg( \average_{D_r} |\nabla u_\lambda|^{q_0} \bigg)^{1/q_0} + \bigg( \average_{D_r} |\lambda\nabla\cdot u_\lambda - \average_{D_r} \lambda\nabla\cdot u_\lambda |^{q_0} \bigg)^{1/q_0}
	\\
	& \quad \le C\bigg\{  \bigg( \average_{D_{2r}} |\nabla u_\lambda|^{2} \bigg)^{1/2} + \bigg( \average_{D_{2r}} |h|^{q_0} \bigg)^{1/q_0} + \bigg( \average_{D_{2r}} |\nabla f|^{q_0} \bigg)^{1/q_0} \bigg\},
	\end{aligned}
	\end{equation*}
	where $C$ depends only on $d,\Lambda$ and $D$.
\end{theorem}
\begin{proof}
	Again, it suffices to consider the case when $\lambda$ is large. By rescaling, assume $r = 1$. Now, let $x\in \Delta_2$ and $D_{2s}(x) \subset D_2$. We construct an approximation of $u_\lambda$ in $D_{2s}(x)$. Actually, let $(v_x^r,p_x^r)$ be the weak solution of
	\begin{equation}\label{eq.approx.Stokes}
	\left\{
	\begin{aligned}
	\nabla\cdot (A(x) \nabla v_x^r) + \nabla p_x^r &= \nabla\cdot h \qquad &\txt{in }& D_{2s}(x), \\
	\nabla\cdot v_x^r & = \Ag{u_\lambda}_{D_{2s}(x)} \qquad &\txt{in }& D_{2s}(x),\\
	v_x^r & = u_\lambda \qquad &\txt{on }& \partial D_{2s}(x).
	\end{aligned}
	\right.
	\end{equation}
	Then Lemma \ref{lem.lambda.rate} implies
	\begin{equation}\label{est.vxr}
	\bigg( \average_{D_{2s}(x)} |\nabla u_\lambda - \nabla v_x^r|^2 \bigg)^{1/2} \le \frac{C}{\lambda} \bigg\{\bigg( \average_{D_{2s}(x)} |\nabla u_\lambda|^2 \bigg)^{1/2} + \bigg( \average_{D_{2s}(x)} |h|^2 \bigg)^{1/2} \bigg\}.
	\end{equation}
	
	Next, we will reduce (\ref{eq.approx.Stokes}) to a homogeneous system. To this end, let $w_x^r$ be the solution of
	\begin{equation}
	\left\{
	\begin{aligned}
	\nabla\cdot (A(x) \nabla w_x^r) + \nabla \pi_x^r &= \nabla\cdot h \qquad &\txt{in }& D_{2s}(x), \\
	\nabla\cdot w_x^r & = \Ag{f}_{D_{2s}(x)} \qquad &\txt{in }& D_{2s}(x),\\
	w_x^r & = f \qquad &\txt{on }& \partial D_{2s}(x).
	\end{aligned}
	\right.
	\end{equation}
	Clearly, the energy estimate implies
	\begin{equation*}
	\bigg( \average_{D_{2s}(x)} |\nabla w^r_x|^2 \bigg)^{1/2} \le C\bigg\{ \bigg( \average_{D_{2s}(x)} |h|^{2} \bigg)^{1/2} + \bigg( \average_{D_{2s}(x)} |\nabla f|^{2} \bigg)^{1/2} \bigg\}
	\end{equation*}
	Combined with (\ref{est.vxr}), this leads to
	\begin{equation}\label{est.Shen.cond1}
	\begin{aligned}
	&\bigg( \average_{D_{2s}(x)} |\nabla u_\lambda - \nabla (v_x^r - w_x^r)|^2 \bigg)^{1/2} \\
	&\quad \le \frac{C}{\lambda} \bigg( \average_{D_{2s}(x)} |\nabla u_\lambda|^2 \bigg)^{1/2} + C\bigg\{ \bigg( \average_{D_{2s}(x)} |h|^2\bigg)^{1/2} + \bigg( \average_{D_{2s}(x)} |\nabla f|^{2}  \bigg)^{1/2} \bigg\}.
	\end{aligned}
	\end{equation}

	On the other hand, observe that the difference $v_x^r - w_x^r$ satisfies
	\begin{equation}
	\left\{
	\begin{aligned}
	\nabla\cdot (A(x) \nabla (v_x^r - w_x^r)) + \nabla (p_x^r -\pi_x^r) &= 0 \quad &\txt{in }& D_{2s}(x), \\
	\nabla\cdot (v_x^r - w_x^r) & = \Ag{u_\lambda - f}_{D_{2s}(x)} \quad &\txt{in }& D_{2s}(x),\\
	v_x^r - w_x^r & = 0 \quad &\txt{on }& \partial D_{2s}(x)\cap D_2.
	\end{aligned}
	\right.
	\end{equation}
	Now, Theorem \ref{thm.Meyers.Stokes} implies that there exists some $p_0>2$ depending only on $d,\Lambda$ and $D$ so that
	\begin{equation}\label{est.Shen.cond2}
	\begin{aligned}
	& \bigg( \average_{D_{s}(x)} |\nabla (v_x^r - w_x^r)|^{p_0} \bigg)^{1/{p_0}} \\
	&\qquad \le C\bigg\{ \bigg( \average_{D_{2s}(x)} | \nabla (v_x^r - w_x^r) |^2 \bigg)^{1/2} + |\Ag{u_\lambda - f}_{D_{2s}(x)}|  \bigg\} \\
	&\qquad  \le C\bigg\{ \bigg( \average_{D_{2s}(x)} |\nabla u_\lambda|^2 \bigg)^{1/2} + \bigg( \average_{D_{2s}(x)} |h|^2 \bigg)^{1/2} + \bigg( \average_{D_{2s}(x)} |\nabla f|^2 \bigg)^{1/2} \bigg\}.
	\end{aligned}
	\end{equation}
	
	In view of (\ref{est.Shen.cond1}) and (\ref{est.Shen.cond2}), which actually holds in any $D_{2s}(x) \subset D_2$, we may apply \cite[Theorem 4.2.6]{ShenBook18} to conclude that for any $q_0\in (2,p_0)$, there exists $\eta>0$ such that if $C/\lambda < \eta$, then
	\begin{equation*}
	\begin{aligned}
	& \bigg( \average_{D_s(x)} |\nabla u_\lambda|^{q_0} \bigg)^{1/q_0} \\
	& \qquad \le C\bigg\{  \bigg( \average_{D_{2s}(x)} |\nabla u_\lambda|^{2} \bigg)^{1/2} + \bigg( \average_{D_{2s}(x)} |h|^{q_0} \bigg)^{1/q_0} + \bigg( \average_{D_{2s}(x)} |\nabla f|^{q_0} \bigg)^{1/q_0} \bigg\},
	\end{aligned}
	\end{equation*}
	for any $D_s(x)$ so that $D_{4s}(x)\subset D_2$. This particularly gives the desired estimate for $\nabla u_\lambda$ with $r = 1$. Finally, the estimate of the pressure follows from Lemma \ref{lem.Nabla}.
\end{proof}

\section{Homogenization of Stokes System}
In this section, we study the quantitative homogenization of Stokes system with coefficient matrix in a probability measure space $(\Omega,\mathcal{F},\mathbb{P})$, where
\begin{equation*}
\Omega := \{ A: \mathbb{R}^d \mapsto \mathbb{R}^{d^2\times d^2} \text{ satisfying } (\ref{cond.ellipticity}) - (\ref{cond.symmetry}) \}.
\end{equation*}
Notice that we redefined $(\Omega,\mathcal{F},\mathbb{P})$ with a slight abuse of notation since it has been defined as the probability measure space for the system of elasticity in Section 1.3. Fortunately, this will cause no ambiguity in this section.

Given an open set $D\subset \R^d$. Let $ \mathcal{F}_D$ be the $\sigma$-algebra generated by the random elements
\begin{equation*}
A \mapsto  \int_{\R^d}a_{ij}^{\alpha\beta}(x)\phi(x),  \quad \phi \in C_0^\infty(D), 1\le i,j,\alpha,\beta\le d.
\end{equation*}
Let $\mathcal{F}$ be the largest $\sigma$-algebra containing all $\mathcal{F}_D$ with $D\subset \R^d$. We assume the probability measure $\mathbb{P}$ satisfies the following assumptions:
\begin{itemize}
	
	\item Stationarity with respect to $\mathbb{Z}^d$-translations:
	\begin{equation*}
	\mathbb{P}\circ T_z = \mathbb{P}, \quad\text{where } T_z(A)(x)= A(x+z).
	\end{equation*}
	
	\item Unit range of dependence:
	\begin{equation*}
	\begin{aligned}
	&\text{$\mathcal{F}_D$ and $\mathcal{F}_E$ are $\mathbb{P}$-independent for every Borel}\\
	&\qquad\text{subset pair $D, \, E\subset \mathbb{R}^d$ satisfying dist$(D,E)\ge 1$}.
	\end{aligned}
	\end{equation*}
	
\end{itemize}

\subsection{Finite volume correctors}
In this subsection, we will introduce several correctors and obtain some quantitative estimates in the case of Stokes system, following the standard work in \cite{AKM19} for the elliptic equation. Since the key results and their proofs are actually very similar to the elliptic equation, we will briefly describe this process with a few key steps and skip the most details.

First of all, we define the solenoidal space $H^{\text{sol}}_0$ by
$$
H^{\text{sol}}_0(D):= H^1_0(D)\cap L^2_{\text{sol}}(D) =\{f\in H^1_0(D,\R^d): \nabla \cdot f =0\}
$$
and $H^{\text{sol}}$ by
$$
H^{\text{sol}}(D):= H^1(D)\cap L^2_{\text{sol}}(D) =\{f\in H^1(D,\R^d): \nabla \cdot f =0\}.
$$

%Consider the general Stokes system in bounded Lipschitz domain $D\subset \R^d$,
%\begin{equation}\label{def.stokes.e}
%\left\{
%\begin{aligned}
%\nabla\cdot(A(x/\varepsilon)\nabla u^\varepsilon)+ \nabla p^\varepsilon &=F &\qquad\text{ in }D,\\
%\nabla \cdot u^\varepsilon &=g &\qquad\text{ in }D,
%\end{aligned}
%\right.
%\end{equation}
For each $P \in \R^{d\times d}$, we introduce the subadditive quantity $\mu(D,P)$ defined by
\begin{equation}\label{def.mu}
\begin{aligned}
\mu(D,P):&= \inf_{\nu \in \ell_P+H^{\text{sol}}_0(D)} \average_D \frac{1}{2} \nabla \nu  \cdot A\nabla \nu  \\
& =\inf_{w\in H^{\text{sol}}_0(D)} \average_D \frac{1}{2}(P+\nabla w)\cdot A(P+\nabla w),
\end{aligned}
\end{equation}
where $\ell_P: = P x$ is the affine function with slope $P$.  The quantity $\mu(D,P)$ is the energy of its unique minimizer 
\begin{equation*}
\nu (\cdot,D,P):= \argmin_{\nu \in \ell_P+H^{\text{sol}}_0(D)}\mu(D,P),
\end{equation*}
which turns out to be the Dirichlet corrector, i.e., the weak solution of
\begin{equation}\label{def.dirichlet-corrector}
\left\{
\begin{aligned}
\nabla\cdot(A\nabla \nu )+ \nabla \varsigma &=0 &\qquad\text{ in }D,\\
\nabla \cdot \nu  &=\text{Tr}(P) &\qquad\text{ in }D,\\
\nu &= \ell_P &\qquad\text{ on }\partial D.
\end{aligned}
\right.
\end{equation}
Note that the pressure $\varsigma$ in the above system is not defined directly in (\ref{def.mu}).

%Similarly, we define the dual subadditive quantity $\mu^*(D,Q)$ by
%\begin{equation}\label{def.mu*}
%\begin{aligned}
%\mu^*(D,Q):& =\sup_{w\in H^{\text{sol}}(D)} \average_D \left(-\frac{1}{2}\nabla w\cdot A\nabla w + Q\nabla w\right),
%\end{aligned}
%\end{equation}
%The quantity $\mu^*(D,Q)$ is the energy of its unique maxmizer $\nu^*(\cdot,D,Q):= \argmax_{w\in H^{\text{sol}}(D)}\mu^*(D,Q)$, where $(\nu^*,\varsigma^*)$ is the Neumann corrector that satisfies
%\begin{equation}\label{def.neumann-corrector}
%\left\{
%\begin{aligned}
%\nabla\cdot(A\nabla \nu^*)+ \nabla \varsigma^* &=0 &\qquad\text{ in }D,\\
%\nabla \cdot \nu^* &=0 &\qquad\text{ in }D,\\
%n\cdot A\nabla \nu^*+ n\varsigma^*&= nQ &\qquad\text{ on }\partial D.
%\end{aligned}
%\right.
%\end{equation}
%By tesing the minimizer of $\mu(D,P)$ in the definition of $\mu^*(D,Q)$, we easily observe that
%$$
%\mu(D,P)+\mu^*(D,Q)\ge P\cdot Q,
%$$
%this motivates us to define the quantity
%\begin{equation}\label{def.J}
%J(D,P,Q) = \mu(D,P)+\mu^*(D,Q)- P\cdot Q.
%\end{equation}

\begin{proposition}\label{prop.mu}
	Let $D$ be a bounded Lipschitz domain in $\R^d$. Then $\mu(D,P)$ and its minimizer satisfy the following properties:
	\begin{itemize}
		\item Representation as quadratic form: there exist an symmetric $A_D$ such that
		$$
		\Lambda^{-1}I \le A_D\le \Lambda I,
		$$
		and, for $P\in \R^{d\times d}$,
		\begin{equation}\label{mu.quadratic-form}
		\mu(D,P) = \frac{1}{2}P\cdot A_DP.
		\end{equation}
		
		\item Subadditivity. Let $\{D_i\}_{i=1}^N \subset D$ be a partition of $D$ of bounded Lipschitz domains, in the sense that $D_i\cap D_j = \emptyset$ if $i\neq j$ and
		$$
		\left|D \setminus \bigcup\limits_{i=1}^N D_i\right| =0.
		$$
		Then, for every $P\in\mathbb{R}^{d\times d}$,
		$$
		\mu(D,P) \le\sum\limits_{i=1}^N \dfrac{|D_i|}{|D|}\mu(D_i,P).
		$$
		\item Quadratic response. For every $w\in \ell_P+H^{\text{sol}}_0(D)$,
		$$
		\begin{aligned}
		\frac{1}{2\Lambda}\fint_D |\nabla w -\nabla \nu(\cdot, D,P)|^2 & \le \fint_D \frac{1}{2}\nabla w \cdot A\nabla w-\mu(D,P)\\
		& \le \frac{\Lambda}{2} \fint_D |\nabla w -\nabla \nu(\cdot, D,P)|^2.
		\end{aligned}
		$$
		
	\end{itemize}
\end{proposition}

For each integer $m \ge 1$, define the triadic cube
\begin{equation*}
\square_m := \Big( -\frac{1}{2}3^m, \frac{1}{2} 3^m \Big)^d \subset \R^d.
\end{equation*}
If $1\le n<m$, then $\square_m$ can be partitioned (up to a set of zero Lebesgue measure) into exactly $3^{d(m-n)}$ subcubes which are $\Z^d$-translations of $\square_n$.

Now, for each $P\in \R^{d\times d}$, the subadditivity of $\mu$ and stationarity of the probability measure space $(\Omega,\mathcal{F},\mathbb{P})$ imply the monotonicity of $\E[\mu(\square_m,P)]$ in $m$, namely, $\E[\mu(\square_{m+1},P)] \le E[\mu(\square_m,P)]$ for all $m\ge 1$. By the monotone convergence theorem, we may define
\begin{equation*}
\widehat{\mu}(P):= \lim_{m\to \infty} \E[ \mu(\square_m,P) ].
\end{equation*}
Using the Proposition \ref{prop.mu}, one sees that for any $P\in \R^{d\times d}$, there is a unique symmetric matrix $\widehat{A}$ such that
$$\label{def.A-hat}
\widehat{\mu}(P) = \frac{1}{2}P\cdot \widehat{A}P,
$$
Moreover, $\Lambda^{-1}I \le \widehat{A} \le \Lambda I$. We will call $\widehat{A}$ the homogenized matrix of the Stokes system with stochastic coefficient matrix $A$.

The homogenized matrix $\widehat{A}$ defined above is sufficient for us to establish the qualitative homogenization for Stokes system. However, it is not enough for quantitative analysis. To this end, for each $Q\in \R^{d\times d}$, we define the dual subadditive quantity $\mu^*(D,Q)$ by
\begin{equation}\label{def.mu*}
\begin{aligned}
\mu^*(D,Q):& =\sup_{w\in H^{\text{sol}}(D)} \average_D \left(-\frac{1}{2}\nabla w\cdot A\nabla w + Q\nabla w\right),
\end{aligned}
\end{equation}
The quantity $\mu^*(D,Q)$ is the energy of its unique maximizer
\begin{equation*}
\nu^*(\cdot,D,Q):= \argmax_{w\in H^{\text{sol}}(D)}\mu^*(D,Q),
\end{equation*}
where $(\nu^*,\varsigma^*)$ turns out to be the Neumann corrector, i.e., the weak solution of
\begin{equation}\label{def.neumann-corrector}
\left\{
\begin{aligned}
\nabla\cdot(A\nabla \nu^*)+ \nabla \varsigma^* &=0 &\qquad\text{ in }D,\\
\nabla \cdot \nu^* &=0 &\qquad\text{ in }D,\\
n\cdot A\nabla \nu^*+ n\varsigma^*&= nQ &\qquad\text{ on }\partial D.
\end{aligned}
\right.
\end{equation}

Proceeding as \cite{AKM19}, we can verify
$\mu(D,P)+\mu^*(D,Q)\ge P\cdot Q$
and then define a crucial quantity
\begin{equation}\label{def.J}
J(D,P,Q) = \mu(D,P)+\mu^*(D,Q)- P\cdot Q,
\end{equation}
which allows to carry out quantitative analysis of correctors. To study the basic properties of $J$, we introduce the set of ``$A$-Stokes functions'' in $D$
\begin{equation*}
\mathcal{A}(D) = \{ u\in H^1(D): \nabla\cdot A \nabla u + \nabla \pi = 0 \text{ in } D \text{ for some } \pi\in L^2(D) \}.
\end{equation*}
Then, we have the following properties for $J(D,P,Q)$.
\begin{proposition}\label{prop.J1}
	Let $D$ be a bounded Lipschitz domain in $\R^d$ and $P,Q\in \R^{d\times d}$. Then
	\begin{equation*}
	J(D,P,Q) = \sup_{w\in \mathcal{A}(D)} \fint_{D} \bigg( -\frac{1}{2}\nabla w\cdot A\nabla w - P\cdot A\nabla w + Q\cdot \nabla w\bigg).
	\end{equation*}
	Moreover, the maximizer $v(\cdot,D,P,Q)$ is equal to $\nu^*(D,Q) - \nu(D,P)$.
\end{proposition}

\begin{proposition}\label{prop.J2}
	Let $D$ be a bounded Lipschtiz domain in $\R^d$. Then the quantity $J(D,P,Q)$ and its maximizer $v(\cdot,D,P,Q)$ satisfy the following properties:
	\begin{itemize}
		\item Representation as quadratic form. The mapping $(P,Q)\mapsto J(D,P,Q)$ is a quadratic form and there exist $A_D, A_{*D} \in \R^{d^2\times d^2}$ such that
		\begin{equation*}
		\Lambda^{-1} I \le A_{*D} \le A_D \le \Lambda I,
		\end{equation*}
		and
		\begin{equation*}
		J(U,P,Q) = \frac{1}{2}P\cdot A_D P + \frac{1}{2}Q\cdot A_{*D}^{-1} Q - P\cdot Q.
		\end{equation*}
		Moreover, the matrices $A_D$ and $A_{*D}$ are characterized by the following relationships:
		\begin{equation*}
		\begin{aligned}
		A_D P &= - \fint_{D} A\nabla v(\cdot,D,P,0),\\
		A_{*D}^{-1} Q &= \fint_D \nabla v(\cdot,D,0,Q).
		\end{aligned}
		\end{equation*}
		
		\item Subadditivity. Let $\{D_i\}_{i=1}^N \subset D$ be a partition of $D$ of bounded Lipschitz domains, in the sense that $D_i\cap D_j = \emptyset$ if $i\neq j$ and
		$$
		\left|D \setminus \bigcup\limits_{i=1}^N D_i\right| =0.
		$$
		Then, for every $P,Q\in\mathbb{R}^{d\times d}$,
		$$
		J(D,P,Q) \le\sum\limits_{i=1}^N \dfrac{|D_i|}{|D|}J(D_i,P,Q).
		$$ 
		
		\item Quadratic response. For every $w\in \mathcal{A}(D)$ and $P,Q\in \R^{d\times d}$,
		\begin{equation*}
		\begin{aligned}
		&\fint_D \frac{1}{2} \big(\nabla w - \nabla v(\cdot,D,P,Q)\big)\cdot A \big(\nabla w - \nabla v(\cdot,D,P,Q)\big) \\
		&\qquad = J(D,P,Q) - \fint_D \bigg(-\frac{1}{2} \nabla w\cdot A\nabla w - P\cdot A\nabla w + Q\cdot \nabla w \bigg).
		\end{aligned}
		\end{equation*}
	\end{itemize}
\end{proposition}

Proposition \ref{prop.J1} and \ref{prop.J2} may be proved by the similar method as (\cite[Lemma 2.1 and 2.2]{AKM19}). Then, following the argument there, we are able to establish all the quantitative estimates of the correctors. Most importantly, we can show the following rate of convergence for the Dirichlet correctors.
\begin{theorem}\label{thm.Dnu}
	Let $s\in (0,d)$. There exists $C(d,\Lambda)<\infty$ such that, for every $m\in \N$ and $\|P\|_2\le 1$,
	\begin{equation*}
	\begin{aligned}
	3^{-m}\|\nabla \nu(\cdot,\square_m,P)-P\|_{\underline{\widehat{H}}^{-1}(\square_m)}&+3^{-m}\|A\nabla \nu(\cdot,\square_m,P)- \widehat{A}P\|_{\underline{\widehat{H}}^{-1}(\square_m)}\\
	& \le C3^{-m\sigma(d-s)} + \mathcal{O}_1(C3^{-ms}).
	\end{aligned}
	\end{equation*}
\end{theorem}

Note that the pressure term does not appear explicitly in the variational method of Stokes system and therefore no estimate of the pressure can be derived directly. In Lemma \ref{lem.DwDpi}, we will see that a weaker estimate for the pressure may be reduced to the estimate of the displacement.

%\begin{proposition}
%There exists $C(d,\Lambda)<\infty$ such that, for every $m\in \N$,
%\begin{equation}
%\E[\mathcal{E}(m)] \le C3^{-(\frac{d}{4}\wedge 1)m}+C\sum\limits_{n=0}^m 3^{n-m}\omega(\lceil\tfrac{m}{2}\rceil),
%\end{equation}
%In particular, $\E[\mathcal{E}(m)]\rightarrow 0$, as $m\rightarrow \infty$.
%\end{proposition}
%
%\begin{proposition}(Multiscale Poincar\'e) Fix $m\in \N$ and, for each $n\in \N$ with $n\le m$. There exists a constant $C(d)<\infty$ such that, for every $f\in L^2(\square_m)$,
%\begin{equation}
%\|f\|_{\underline{\widehat{H}}^{-1}(\square_m)} \le C\|f\|_{\underline{L}^2(\square_m)}+C\sum\limits_{n=0}^{m-1}3^n\left(\frac{1}{|\mathcal{Z}_n|}\sum_{y\in \mathcal{Z}_n}|\langle f\rangle_{y+\square_n}|^2\right)^{1/2}.
%\end{equation}
%\end{proposition}
%
%\begin{lemma}
%There exists a $C(d,\Lambda)<\infty$ such that, for every $m\in\N$ and $u\in H^1(\square_m)$,
%\begin{equation}
%\|u-\langle u \rangle_{\square_m}\|_{\underline{L}^2(\square_m)} \le C\|\nabla u\|_{\underline{\widehat{H}}^{-1}(\square_m)}.
%\end{equation}
%\end{lemma}
%Which implies
%\begin{corollary}There exists a $C(d,\Lambda)<\infty$ such that, for every $m\in\N$ and $u\in H^1(\square_m)$,
%\begin{equation}
%\|u-\langle u \rangle_{\square_m}\|_{\underline{L}^2(\square_m)} \le C\|\nabla u\|_{\underline{L}^2(\square_m)}+C\sum\limits_{n=0}^{m-1}3^n\left(\frac{1}{|\mathcal{Z}_n|}\sum_{y\in \mathcal{Z}_n}|\langle \nabla u\rangle_{y+\square_n}|^2\right)^{1/2}.
%\end{equation}
%\end{corollary}

Next, we construct the ``finite volume corrector'' (an approximation of the true corrector) which is crucial in the study of the convergence rates in homogenization. For each $n\in \N$ and $P\in \R^{d\times d}$, define $(\Phi_{n,j}^\beta,\Pi_{n,j}^\beta)$ by
\begin{equation*}
\Phi_{n,j}^\beta(x) = \nu(x,\square_n,P_j^\beta) - P_j^\beta x, \qquad \Pi_{n,j}^\beta(x) = \pi(x,\square_n,P_j^\beta),
\end{equation*}
where $P_j^\beta = e^\beta \otimes e_j$.
Observe that $(\Phi_{n,j}^\beta, \Pi_{n,j}^\beta)$ solves
\begin{equation}
\left\{
\begin{aligned}
\nabla\cdot(A\nabla \Phi_{n,j}^\beta )+ \nabla \Pi_{n,j}^\beta &=-\nabla\cdot(A P_j^\beta ) &\qquad\text{ in }\square_n,\\
\nabla \cdot \Phi_{n,j}^\beta  &=0 &\qquad\text{ in }\square_n,\\
\Phi_{n,j}^\beta &= 0 &\qquad\text{ on }\partial \square_n.
\end{aligned}
\right.
\end{equation}

The following theorem is a rescaling of Theorem \ref{thm.Dnu}.
\begin{theorem}\label{thm.corrector.Phi}
	Fix $s\in (0,d)$. There exists $\sigma = \sigma(d,\Lambda) >0$ so that
	\begin{equation*}
	\e \norm{\Phi_{n,j}^\beta }_{L^2(\e\square_n)} + \norm{A(\cdot/\e) ( \nabla \Phi_{n,j}^\beta(\cdot/\e) +  P_j^\beta) - \widehat{A} P_j^\beta }_{H^{-1}(\e\square_n)} \le C\e^{\sigma(d-s)} + C\e^s \mathcal{O}_1(1),
	\end{equation*}
	where $\e = 3^{-n}$.
\end{theorem}

Again, the above theorem does not include the estimate of the pressure $\Pi_n$. The following lemma reduces the estimate of $\Pi_n$ to the estimate of $\Phi_n$.
\begin{lemma}\label{lem.DwDpi}
	Let $D$ be a bounded Lipschitz domain. Let $W \in L^2(D;\R^{d\times d})$ and $ \pi\in L^2_0(D)$ satisfy $\nabla\cdot W + \nabla \pi = 0$ in $D$. Then, there exists $\delta = \delta(d,\Lambda, \text{Lip}(D))>0$ so that
	\begin{equation*}
	\norm{\pi}_{\underline{H}^{-1}(D)} \le C\norm{W}_{\underline{H}^{-1}(D)}^\delta \norm{W}_{\underline{L}^2(D)}^{1-\delta},
	\end{equation*}
	where $C$ depends only on $d,\Lambda$ and $\text{Lip}(D)$.
\end{lemma}
\begin{proof}
	This can be proved by duality and the Meyers' estimate. By rescaling, we may assume $|D| = 1$. Let $\phi \in H^1(D)$ with $\int_D \phi = 0$ and let $(u,p)$ solve
	\begin{equation}
	\left\{
	\begin{aligned}
	\Delta u + \nabla p &= 0 \qquad &\txt{in }& D, \\
	\nabla\cdot u & = \phi \qquad &\txt{in }& D, \\
	u & = 0 \qquad &\txt{on }& \partial D.
	\end{aligned}
	\right.
	\end{equation}
	Note that $\phi \in H^1(D)$ implies $\phi\in L^{p_1}$ where $p_1 = 2d/(d-2)$. Since $D$ is Lipschitz, by the global version of Theorem \ref{thm.Meyers.Stokes}, there exists $p_0\in (2,p_1]$ so that
	\begin{equation}\label{est.Dup.phi}
	\norm{\nabla u}_{L^{p_0}(D)} + \norm{p}_{L^{p_0}(D)} \le C\norm{\phi}_{L^{p_0}(D)} \le C\norm{\phi}_{H^1(D)}.
	\end{equation}
	
	On the other hand, the interior $H^2$ regularity implies $u\in H^2_{\text{loc}}(D;\R^d)$. Moreover, if $D'\subset\!\subset D$
	\begin{equation}\label{est.D2u}
	\norm{\nabla^2 u}_{L^2(D')} \le \frac{C}{\text{dist}(D', \partial D)} \norm{\phi}_{H^1(D)}.
	\end{equation}
	
	Now, given any $t\in (0,1)$, define $D_t = \{ x\in D: \text{dist}(x,\partial D) \ge t \}$. Let $\eta_t\in C_0^\infty(D)$ be a cut-off function so that $\eta_t(x) = 0$ if $x\in D\setminus D_t$ and $\eta_t(x) = 1$ if $x\in D_{2t}$. Moreover, $|\nabla \eta_t(x)| \le Ct^{-1}$. For any matrix $P\in \R^{d\times d}$, the integration by parts yields
	\begin{equation*}
	\begin{aligned}
	\int_{D} \pi \phi &= \int_{D} \pi \nabla\cdot u \\
	& = -\int_{D} \nabla \pi \cdot u \\
	& = \int_{D} (\nabla\cdot W) \cdot u \\
	& = -\int_{D_t} W \cdot \eta_t \nabla u - \int_{D\setminus D_{2t} } W \cdot (1-\eta_t) \nabla u.
	\end{aligned}
	\end{equation*}
	The first integral is bounded by
	\begin{equation*}
	\norm{W}_{H^{-1}(D)} \norm{\eta_t \nabla u}_{H^1(D_t)} \le Ct^{-1} \norm{W}_{H^{-1}(D)} \norm{\phi}_{H^1(D)},
	\end{equation*}
	where we have used Theorem \ref{thm.Stokes.Energy} and (\ref{est.D2u}). The second integral is bounded by
	\begin{equation*}
	\begin{aligned}
	\norm{W}_{L^2(D)} \norm{\nabla u}_{L^2(D\setminus D_{2t})} &\le C|D\setminus D_{2t}|^{\frac{1}{2} - \frac{1}{p_0} } \norm{W}_{L^2(D)} \norm{\nabla u}_{L^{p_0}(D)} \\
	& \le Ct^{\frac{1}{2} - \frac{1}{p_0}} \norm{W}_{L^2(D)} \norm{\phi}_{H^1(D)},
	\end{aligned}
	\end{equation*}
	where we have used (\ref{est.Dup.phi}) in the last inequality. It follows that
	\begin{equation}\label{est.pi.dual}
	\bigg| \int_{D} \pi \phi \bigg| \le \Big( Ct^{-1} \norm{W}_{H^{-1}(D)} + Ct^{\frac{1}{2} - \frac{1}{p_0}} \norm{W}_{L^2(D)} \Big) \norm{\phi}_{H^1(D)}.
	\end{equation}
	Now choose
	\begin{equation*}
	t = \bigg( \frac{\norm{W}_{H^{-1}(D)}}{\norm{W}_{L^2(D)}} \bigg)^{\frac{1}{1+\frac{1}{2} - \frac{1}{p_0}}}, \qquad \beta = \frac{ \frac{1}{2} - \frac{1}{p_0} }{ 1+\frac{1}{2} - \frac{1}{p_0} }.
	\end{equation*}
	Then we obtain the desired estimate from (\ref{est.pi.dual}) by duality.
\end{proof}

\begin{theorem}\label{thm.corrector.Pi}
	Fix $s\in (0,d)$. There exists $\sigma >0, \delta>0$ so that
	\begin{equation*}
	\norm{\Pi_{n,j}^\beta(\cdot/\e) }_{H^{-1}(\e\square_n)} \le C\e^{\sigma(d-s)} + C(\e^s O_1(1))^\delta,
	\end{equation*}
	where $\e = 3^{-n}$.
\end{theorem}
\begin{proof}
	This follows from Theorem \ref{thm.corrector.Phi} and Lemma \ref{lem.DwDpi}.
\end{proof}

%\begin{theorem}
%	Let $\delta>0$, and $D \subset \square_0$ be a bounded Lipschitz domain. There esist an exponent $\beta(\delta, d,\Lambda)>0$ and a constant $C(\delta,D,d,\Lambda)<\infty$ such that for every $\varepsilon \in (0,1]$, $f\in W^{1,2+\delta}(D)$ and solution pairs $(u^\varepsilon,p^\e),\  (u^0,p^0)$ of Dirichlet problems
%	\begin{equation}
%	\left\{
%	\begin{aligned}
%	\nabla\cdot(A^\e \nabla u^\e)+\nabla p^\e &=0 &\text{ in }D,\\
%	\nabla\cdot u^\e &=\nabla\cdot f &\text{ in }D,\\
%	u^\e & = f &\text{ on }\partial D,\\
%	\end{aligned}
%	\right.
%	\text{ and }
%	\left\{
%	\begin{aligned}
%	\nabla\cdot(\widehat{A} \nabla u^0)+\nabla p^0 &=0 &\text{ in }D,\\
%	\nabla\cdot u^0 &=\nabla\cdot f &\text{ in }D,\\
%	u^0 & = f &\text{ on }\partial D,\\
%	\end{aligned}
%	\right.
%	\end{equation}
%	we have, for every $r\in (0,1)$, that
%	\begin{equation}
%	\begin{aligned}
%	\|u^\varepsilon -u^0\|_{L^2(D)} +\|\nabla u^\varepsilon-\nabla u^0\|_{\widehat{H}^{-1}(D)}&+\|p^\e-p^0\|_{\widehat{H}^{-1}(D)}+\|A^\e\nabla u^\e-\widehat{A}\nabla u^0\|_{\widehat{H}^{-1}(D)}\\
%	&\le C\|\nabla f\|_{L^{2+\delta}(D)}\left(r^\beta+r^{-(2+d/2)}\mathcal{E}'(\varepsilon)^{\frac{1}{2}}\right),
%	\end{aligned}
%	\end{equation}
%	and
%	\begin{equation}
%	\begin{aligned}
%	\|\frac{1}{2}\nabla u^\varepsilon \cdot A^\e \nabla u^\e -\frac{1}{2}\nabla u^0\cdot \widehat{A}\nabla u^0\|_{W^{-2,1}(D)}\le C\|\nabla f\|^2_{L^{2+\delta}(D)}\left(r^\beta+r^{-(2+d/2)}\mathcal{E}'(\varepsilon)^{\frac{1}{2}}\right),
%	\end{aligned}
%	\end{equation}
%\end{theorem}
\subsection{Convergence rates}

With Theorem \ref{thm.corrector.Phi} and \ref{thm.corrector.Pi}, we are able to prove the an algebraic rate of convergence for Stokes system in Lipschitz domains.

\begin{theorem}\label{thm.Stokes.rate}
	Let $\delta>0, s\in (0,d)$, and $D$ be a bounded Lipschitz domain. There exist exponents $\alpha, \beta>0$, a constant $C>0$ (depending only on $\delta, s, d, \Lambda$ and $\text{Lip}(D)$) and a random variable $\X = \X_s:\Omega \to [0,\infty)$ satisfying
	\begin{equation*}
	\X \le \mathcal{O}_1(C),
	\end{equation*}
	such that for every $\varepsilon \in (0,1]$ and the solution pairs $(u^\varepsilon,p^\e),\  (u^0,p^0)$ of the Stokes systems
	\begin{equation*}
	\left\{
	\begin{aligned}
	\nabla\cdot(A^\e \nabla u^\e)+\nabla p^\e &=0 &\text{ in }D,\\
	\nabla\cdot u^\e &=\Ag{f}_D &\text{ in }D,\\
	u^\e & = f &\text{ on }\partial D,\\
	\end{aligned}
	\right.
	\end{equation*}
	and
	\begin{equation*}
	\left\{
	\begin{aligned}
	\nabla\cdot(\widehat{A} \nabla u^0)+\nabla p^0 &=0 &\text{ in }D,\\
	\nabla\cdot u^0 &=\Ag{f}_D &\text{ in }D,\\
	u^0 & = f &\text{ on }\partial D,\\
	\end{aligned}
	\right.
	\end{equation*}
	with $u_0\in W^{1,2+\delta}(D;\R^d)$, we have
	\begin{equation*}
	\begin{aligned}
	&\|u^\varepsilon -u^0\|_{L^2(D)} +\|\nabla u^\varepsilon-\nabla u^0\|_{H ^{-1}(D)}+\|A^\e\nabla u^\e-\widehat{A}\nabla u^0\|_{H^{-1}(D)}+\|p^\e-p^0\|_{H^{-1}(D)}\\
	&\quad\le C\Big( \e^{\beta(d-s)} + (\e^s \X)^\alpha \Big) \|\nabla u_0\|_{L^{2+\delta}(D)}.
	\end{aligned}
	\end{equation*}
\end{theorem}

The proof of Theorem \ref{thm.Stokes.rate} uses the same idea from \cite[Theorem 1.17]{AKM19}, as well as some useful computation for Stokes system from \cite{Gu16}. We omit the detailed proof.

\section{Homogenization of Elasticity System}

In this section, we will establish the convergence rate for the linear system of nearly incompressible elasticity.

\subsection{Reduction to constant $\lambda$}
Let the probability measure space $(\Omega,\mathcal{F},\mathbb{P})$ satisfy (\ref{def.Omega}), \ref{def.stationarity} and (\ref{def.unit-range-dependence}). Let $u^\e_\lambda$ be the weak solution of
\begin{equation*}
\left\{
\begin{aligned}
\nabla\cdot (A^\e \nabla u^\e_\lambda) + \nabla (\lambda^\e \nabla\cdot u^\e_\lambda) &= 0 \qquad &\txt{in } &D, \\
u^\e_\lambda  &= f \qquad &\txt{on } &\partial D.
\end{aligned}
\right.
\end{equation*}
In order to reduce the system to the case with constant $\lambda$, use the splitting technique in (\ref{eq.reduction}) and (\ref{eq.splitting}), and write (\ref{eq.elasticity.Le}) as
\begin{equation}\label{eq.Le.const}
\left\{
\begin{aligned}
\nabla\cdot (\widetilde{A}^\e \nabla u^\e_\lambda) + \lambda_0 \nabla (\nabla\cdot u^\e_\lambda) &= 0 \qquad &\txt{in } &D, \\
u^\e_\lambda  &= f \qquad &\txt{on } &\partial D,
\end{aligned}
\right.
\end{equation}
where $\lambda(x) = \lambda_0 + b(x)$ and $\widetilde{A} = (\widetilde{a}_{ij}^{\alpha\beta})$ is defined by
\begin{equation}\label{def.splitting}
\widetilde{a}_{ij}^{\alpha\beta}(x) = a_{ij}^{\alpha\beta}(x) + b(x)\delta_{i}^\alpha \delta_{j}^\beta.
\end{equation}
Denote the map $(A,\lambda) \mapsto (\widetilde{A},\lambda_0)$ by $S$, i.e., $(\widetilde{A},\lambda_0) = S(A,\lambda)$.
Now, note that in (\ref{eq.Le.const}), $\lambda_0$ is a fixed constant and $\widetilde{A}$ is a new (tensor-valued) random variable on $\Omega$ satisfying the similar hypothesis as $A$. Define a new probability measure space $(\widetilde{\Omega}, \widetilde{\mathcal{F}}, \widetilde{\mathbb{P}})$ endowed by $(\Omega,\mathcal{F},\mathbb{P})$ via the map $S$, where
\begin{equation*}
\widetilde{\Omega} =\{ (\widetilde{A},\lambda_0) = S(A,\lambda): (A,\lambda)\in \Omega  \},
\end{equation*}
$\widetilde{\mathcal{F}} = S(\mathcal{F})$ and $\widetilde{\mathbb{P}}[\omega] = \mathbb{P}[S^{-1}(\omega)]$ for every $\omega\in \widetilde{\mathcal{F}}$.

The following proposition shows that the probability measure space $(\widetilde{\Omega}, \widetilde{\mathcal{F}}, \widetilde{\mathbb{P}})$ satisfies the same conditions as in Section 4.
\begin{proposition}\label{prop.tOmega}
	The probability measure space $(\widetilde{\Omega}, \widetilde{\mathcal{F}}, \widetilde{\mathbb{P}})$ satisfies the following conditions:
	\begin{itemize}
		
		\item Stationarity with respect to $\mathbb{Z}^d$-translations:
		\begin{equation*}
		\widetilde{\mathbb{P}}\circ T_z = \widetilde{\mathbb{P}}, \quad\text{where }(T_z(\widetilde{A},\lambda_0))(x)= (\widetilde{A}(x+z),\lambda_0).
		\end{equation*}
		
		\item Unit range of dependence:
		\begin{equation*}
		\begin{aligned}
		&\text{$\widetilde{\mathcal{F}}_D$ and $\widetilde{\mathcal{F}}_E$ are $\mathbb{P}$-independent for every Borel}\\
		&\qquad\text{subset pair $D, \, E\subset \mathbb{R}^d$ satisfying dist$(D,E)\ge 1$}.
		\end{aligned}
		\end{equation*}
		
	\end{itemize}
\end{proposition}

Because of the above reduction, without loss of generality, it suffices to consider the case with constant $\lambda$. In other words, we consider
\begin{equation}\label{eq.elasticity.LeConst}
\left\{
\begin{aligned}
\nabla\cdot (A^\e \nabla u^\e_\lambda) + \lambda \nabla ( \nabla\cdot u^\e_\lambda) &= 0 \qquad &\txt{in } &D, \\
u^\e_\lambda  &= f \qquad &\txt{on } &\partial D,
\end{aligned}
\right.
\end{equation}
where $\lambda\ge 0$ is constant and (\ref{def.Omega}) - (\ref{def.unit-range-dependence}) are satisfied.

\subsection{Small $\lambda$ case}
For $\lambda$ relatively small (compared to $\e^{-\sigma}$), the system of elasticity may be viewed as the elliptic system. In this case, the homogenization theory may be established by the standard method developed in \cite{AS16,AM16,AKM16,AKM19}. One of the goals in this subsection is to identify the structure of the homogenized operator depending on $\lambda$, which is not obvious.

Let $C_\lambda = (c_{\lambda,ij}^{\alpha\beta})$ be defined by
\begin{equation*}
c_{\lambda,ij}^{\alpha\beta}(x) = a_{ij}^{\alpha\beta}(x) + \lambda \delta^\alpha_i \delta^j_\beta.
\end{equation*}
Thus, the system (\ref{eq.elasticity.LeConst}) may be written as
\begin{equation}
\left\{
\begin{aligned}
\nabla\cdot (C^\e_\lambda \nabla u^\e_\lambda) &= 0 \qquad &\txt{in } &D, \\
u^\e_\lambda  &= f \qquad &\txt{on } &\partial D.
\end{aligned}
\right.
\end{equation}
Clearly, the random variable $C_\lambda$ also satisfies the stationarity and the unit-range dependence assumptions. However, in this case, the ellipticity constant of $C_{\lambda}$ depends on $\lambda$, namely,
\begin{equation}\label{cond.ellipticity.lambda}
\Lambda^{-1} |\xi|^2 \le  c_{\lambda,ij}^{\alpha\beta}(x)\xi_i^\alpha \xi_j^\beta \le (2\Lambda + \lambda) |\xi|^2 \quad \text{for any } x\in \R^d, \xi \in \R^{d\times d}. 
\end{equation}
In this case, we need to figure out how the homogenized operator and the convergence rate depend on the parameter $\lambda$. As in \cite{AKM19}, define
\begin{equation}\label{def.mu.lambda}
\begin{aligned}
\mu_\lambda(D,P) &: = \inf_{\nu \in \ell_P+H^1_0(D;\R^d)} \average_D \frac{1}{2} \nabla \nu  \cdot C_\lambda \nabla \nu \\
& = \inf_{\nu \in \ell_P+H^1_0(D;\R^d)} \average_D \bigg( \frac{1}{2} \nabla \nu  \cdot A \nabla \nu + \frac{1}{2}\lambda (\nabla\cdot \nu)^2 \bigg),
\end{aligned}
\end{equation}
where $\ell(P): = Px$ is an affine function. Recall that the unique minimizer $\nu_\lambda = \nu_\lambda(\cdot,D,P)$ (also called the Dirichlet corrector) is the solution of
\begin{equation}\label{eq.DCorrector}
\left\{
\begin{aligned}
\nabla\cdot (C_\lambda \nabla \nu_\lambda) = \nabla\cdot (A\nabla \nu_\lambda) + \lambda\nabla(\nabla\cdot \nu_\lambda) &= 0 \qquad &\txt{in } &D, \\
\nu_\lambda  &= \ell(P) \qquad &\txt{on } &\partial D.
\end{aligned}
\right.
\end{equation}

\begin{proposition}
	The minimum energy $\mu_\lambda(D,P)$ satisfies the following properties:
	
	\begin{itemize}
		\item[(i)] Representation as quadratic form: there exists an symmetric $A_{\lambda}(D)$ such that
		\begin{equation*}
		\Lambda^{-1}I \le A_{\lambda}(D) \le \Lambda I,
		\end{equation*}
		and for each $P\in \R^{d\times d}$
		\begin{equation*}
		\mu_\lambda(D,P) - \frac{1}{2}\lambda \txt{Tr}(P)^2 = \frac{1}{2} P\cdot A_{\lambda}(D) P.
		\end{equation*}
		
		\item[(ii)] Subadditivity: Let $\{D_i\}_{i=1}^N \subset D$ be a partition of $D$ of bounded Lipschitz domains, in the sense that $D_i\cap D_j = \emptyset$ if $i\neq j$ and
		$$
		\left|D \setminus \bigcup\limits_{i=1}^N D_i\right| =0.
		$$
		Then, for every $P\in\mathbb{R}^{d\times d}$,
		$$
		\mu_\lambda(D,P) \le\sum\limits_{i=1}^N \dfrac{|D_i|}{|D|}\mu_\lambda(D_i,P).
		$$
	\end{itemize}
\end{proposition}

\begin{proof}
	Part (ii) is the same as in \cite{AKM19} for elliptic equations. It suffices to prove Part (i). As in \cite{AKM19}, we know in priori that $\mu_\lambda(D,P)$ is a symmetric quadratic form of $P$. For $\nu\in \ell_P + H_0^1(D;\R^d)$, using the divergence theorem, we have
	\begin{equation*}
	\average_D \frac{1}{2}\lambda (\nabla\cdot \nu)^2 = \average_D \frac{1}{2}\lambda (\nabla\cdot \nu - \text{Tr}(P) )^2 + \frac{1}{2}\lambda \text{Tr}(P)^2.
	\end{equation*}
	It follows that
	\begin{equation}\label{eq.mu-TrP}
	\mu_\lambda(D,P) - \frac{1}{2}\lambda \txt{Tr}(P)^2 = \inf_{\nu \in \ell_P+H^1_0(D;\R^d)} \average_D \bigg( \frac{1}{2} \nabla \nu  \cdot A \nabla \nu + \frac{1}{2}\lambda (\nabla\cdot \nu - \text{Tr}(P) )^2 \bigg).
	\end{equation}
	Now, by choosing $\nu = \ell_P$, we have $\nabla\cdot \nu = \text{Tr}(P)$ and therefore
	\begin{equation}\label{est.upper}
	\mu_\lambda(D,P) - \frac{1}{2}\lambda \txt{Tr}(P)^2 \le \frac{1}{2} \average_D P\cdot A P \le \frac{1}{2}\Lambda|P|^2.
	\end{equation}
	On the other hand, let $\nu = \ell_P + w$ with $w\in H_0^1(D;\R^d)$. Then, the H\"{o}lder's inequality and the divergence theorem imply
	\begin{equation}\label{est.lower}
	\begin{aligned}
	\mu_\lambda(D,P) - \frac{1}{2}\lambda \txt{Tr}(P)^2 &\ge \inf_{w \in H^1_0(D;\R^d)} \average_D \frac{1}{2} \nabla (P + \nabla w)  \cdot A \nabla (P+\nabla w) \\
	& \ge \frac{1}{2\Lambda} \inf_{w \in H^1_0(D;\R^d)} \average_D |P + \nabla w|^2 \\
	& \ge \frac{1}{2\Lambda} \inf_{w \in H^1_0(D;\R^d)} \bigg| \average_D (P + \nabla w) \bigg|^2 \\
	& = \frac{1}{2\Lambda} |P|^2.
	\end{aligned}
	\end{equation}
	Since we already know $\mu_\lambda(D,P) - \frac{1}{2}\lambda \txt{Tr}(P)^2$ is a symmetric quadratic form, (\ref{est.upper}) and (\ref{est.lower}) leads to the desired representation and estimate.
\end{proof}

Let $\square_m$ be the triadic cube defined as in Section 4. Then the subadditivity property implies that the limit of $\E[\mu_\lambda(\square_m,P)]$ exists as $m\to \infty$. Denote this limit by $\overline{\mu}_\lambda(P)$. It turns out that the limit
\begin{equation*}
\overline{\mu}_\lambda(P) - \frac{1}{2}\lambda \txt{Tr}(P)^2 = \lim_{m\to \infty} \E \Big[ \frac{1}{2} P\cdot A_\lambda(\square_m) P \Big]
\end{equation*}
exists and is a quadratic form. Hence, there exists a constant $\overline{A}_\lambda$ satisfying
\begin{equation}\label{cond.hatA.lambda}
\Lambda^{-1}I \le\overline{A}_{\lambda} \le \Lambda I,
\end{equation}
so that
\begin{equation*}
\overline{\mu}_\lambda(P)  = \frac{1}{2} P\cdot \overline{A}_\lambda P + \frac{1}{2}\lambda \txt{Tr}(P)^2.
\end{equation*}
%\begin{proposition}
%	Let $\lambda \in [0,\infty)$. The homogenized operator of $\nabla\cdot A^\e\nabla + \lambda\nabla(\nabla\cdot)$ is $\nabla\cdot \overline{A}_\lambda \nabla+ \lambda\nabla(\nabla\cdot)$, where $\overline{A}_\lambda$ is a constant coefficient field satisfying the ellipticity condition (\ref{cond.hatA.lambda}) in dependent of $\lambda$.
%\end{proposition}
This particularly indicates that the homogenized operator of $\nabla\cdot A^\e\nabla + \lambda\nabla(\nabla\cdot)$ takes a form of $\nabla\cdot \overline{A}_\lambda \nabla+ \lambda\nabla(\nabla\cdot)$. In fact, we have the following rate of convergence.
\begin{theorem}\label{thm.small-lambda}
	Let $\delta>0$ and $D$ be a bounded Lipschitz domain. Let $s\in (0,d)$ and $\lambda \in [0,\infty)$. There exists a random variable $\X = \X_{s,\lambda}$ and constants $\alpha,\beta, C_0>0$ (independent of $\e$ and $\lambda$) satisfying
	\begin{equation}\label{est.X.lambda}
	\X \le \mathcal{O}_1(C_0),
	\end{equation}
	such that if $u^\e_\lambda$ and $u^0_\lambda$ are the weak solutions of
	\begin{equation}\label{eq.rate.e}
	\left\{
	\begin{aligned}
	\nabla\cdot (A^\e \nabla u^\e_\lambda) + \lambda \nabla ( \nabla\cdot u^\e_\lambda) &= 0 \qquad &\txt{in } &D, \\
	u^\e_\lambda  &= f \qquad &\txt{on } &\partial D,
	\end{aligned}
	\right.
	\end{equation}
	and
	\begin{equation}\label{eq.rate.0}
	\left\{
	\begin{aligned}
	\nabla\cdot (\overline{A}_\lambda \nabla u^0_\lambda) + \lambda \nabla ( \nabla\cdot u^0_\lambda) &= 0 \qquad &\txt{in } &D, \\
	u^0_\lambda  &= f \qquad &\txt{on } &\partial D,
	\end{aligned}
	\right.
	\end{equation}
	respectively, with $f\in W^{1,2+\delta}(D;\R^d)$, then
	\begin{equation}\label{est.SRate}
	\norm{u^\e_\lambda - u^0_\lambda}_{L^2(D)} + \norm{\nabla u^\e_\lambda - \nabla u^0_\lambda}_{H^{-1}(D)} \le C(\lambda+1) \big(\e^{\beta(d-s)} + \e^s \X \big) \norm{\nabla f}_{L^{2+\delta}(D)}.
	\end{equation}
\end{theorem}

\begin{proof}
	By viewing (\ref{eq.rate.e}) as an elliptic system with ellipticity constant $\Lambda+\lambda$, the result may be seen by examining the proof \cite[Theorem 2.18]{AKM19}.
\end{proof}

\begin{remark}
	Given $s\in (0,d)$, let $\sigma = \beta(d-s)/2$ and $s' = s+\sigma \in (0,d)$. If $\lambda \le \e^{-\sigma}$, the above theorem particularly implies
	\begin{equation}\label{est.YOs}
	\norm{u^\e_\lambda - u^0_\lambda}_{L^2(D)} + \norm{\nabla u^\e_\lambda - \nabla u^0_\lambda}_{H^{-1}(D)} \le \big(\e^{\beta(d-s')-\sigma} + (\e \X')^{s'-\sigma} \big) \norm{\nabla f}_{L^{2+\delta}(D)},
	\end{equation}
	where $\X' = (\X)^{1/(s'-\sigma)} = \X^{1/s} \le \mathcal{O}_s(C_0')$ (since $\X \le \mathcal{O}_1(C_0)$). Now, without loss of generality, by making $\beta$ smaller in (\ref{est.YOs}), we have
	\begin{equation}\label{est.u.X'}
	\norm{u^\e_\lambda - u^0_\lambda}_{L^2(D)} + \norm{\nabla u^\e_\lambda - \nabla u^0_\lambda}_{H^{-1}(D)} \le \big(\e^{\beta(d-s)} + (\e \X')^{s} \big) \norm{\nabla f}_{L^{2+\delta}(D)},
	\end{equation}
	with $\X' = \X'_{s,\lambda} \le \mathcal{O}_s(C)$. 
	
	For the same reason, if $\lambda\le \e^{-\sigma}$, (\ref{est.SRate}) also implies the rate of convergence for the pressure
	\begin{equation}\label{est.p.X'}
	\norm{\lambda \nabla\cdot u^\e_\lambda - \lambda \nabla\cdot u^0_\lambda}_{H^{-1}(D)} \le C\big(\e^{\beta(d-s)} + (\e \X')^{s} \big) \norm{\nabla f}_{L^{2+\delta}(D)},
	\end{equation}
	with the same $\X'$.
	
	We should point out that even though the homogenized matrix $\overline{A}_\lambda$ and the random variable $\X_{s,\lambda}$ may vary as $\lambda$ varies, both of them are fortunately in the classes that are in dependent of $\lambda$; see (\ref{cond.hatA.lambda}) and (\ref{est.X.lambda}).
\end{remark}

\subsection{Large $\lambda$ case}
To obtain a rate of convergence for the system of elasticity when $\lambda$ is relatively large, we have to employ the result of Stokes system. Let us state the main result of this subsection.

\begin{theorem}\label{thm.large-lambda}
	Let the same assumptions of Theorem \ref{thm.small-lambda} hold. There exist a random variable $\X'' = \X''_{s}$  and constants $\alpha,\beta,C_0>0$ (independent of $\e$ and $\lambda$) satisfying
	\begin{equation*}
	\X'' \le \mathcal{O}_s(C_0),
	\end{equation*}
	so that
	\begin{equation*}
	\begin{aligned}
	& \norm{u^\e_\lambda - u^0_\lambda}_{L^2(D)} + \norm{\lambda \nabla\cdot u^\e_\lambda - \lambda\nabla\cdot u^0_\lambda}_{H^{-1}(D)} \\
	&\qquad \le C \big( \lambda^{-1}+ \e^{\beta(d-s)}  + (\X'' \e)^{\alpha s} \big)\norm{\nabla f}_{L^{2+\delta}(D)}.
	\end{aligned}
	\end{equation*}
	where $u^\e_\lambda$ and $u^0_\e$ are the weak solutions of (\ref{eq.rate.e}) and (\ref{eq.rate.0}), respectively.
\end{theorem}

To prove this theorem, we need the following lemma.
\begin{lemma}\label{lem.barA-hatA}
	For any $\lambda\in (0,\infty)$, we havve $|\overline{A}_\lambda - \widehat{A}| \le C\lambda^{-1}$.
\end{lemma}

\begin{proof}
	Let $\mu(D,P)$ and $\mu_\lambda(D,P)$ be defined in (\ref{def.mu}) and (\ref{def.mu.lambda}), respectively. Let $\nu_\infty \in \ell_P + H_0^{\text{sol}}(D;\R^d)$ and $\nu_\lambda\in \ell_P + H_0^1(D;\R^d)$ be the minimizers of (\ref{def.mu}) and (\ref{def.mu.lambda}), respectively. Because $H_0^{\text{sol}}(D;\R^d) \subset H_0^1(D;\R^d)$, we have
	\begin{equation}\label{est.mulambda.mu}
	\begin{aligned}
	\mu_\lambda(D,P) & \le \average_D \bigg( \frac{1}{2} \nabla \nu_\infty  \cdot A \nabla \nu_\infty + \frac{1}{2}\lambda (\nabla\cdot \nu_\infty)^2 \bigg) \\
	& = \average_D \bigg( \frac{1}{2} \nabla \nu_\infty  \cdot A \nabla \nu_\infty \bigg) + \frac{1}{2}\lambda \text{Tr}(P)^2 \\
	& = \mu(D,P) + \frac{1}{2}\lambda \text{Tr}(P)^2,
	\end{aligned}
	\end{equation}
	where we have used the fact $\nabla\cdot v_\infty = \text{Tr}(P)$.
	
	On the other hand, by Lemma \ref{lem.lambda.rate} and rescaling,
	\begin{equation*}
	\bigg( \average_D |\nabla v_\lambda - \nabla v_\infty |^2 \bigg)^{1/2}  \le C\lambda^{-1} |P|.
	\end{equation*}
	Consequently,
	\begin{equation}\label{est.mu.mulambda}
	\begin{aligned}
	\mu(D,P) & =  \average_D \frac{1}{2} \nabla \nu_\infty  \cdot A\nabla \nu_\infty \\
	& = \average_D \frac{1}{2} \nabla \nu_\lambda  \cdot A\nabla \nu_\lambda + \average_D \frac{1}{2} (\nabla \nu_\infty - \nabla \nu_\lambda)  \cdot A( \nabla  \nu_\infty + \nabla \nu_\lambda) \\
	& \le \mu_\lambda(D,P) - \frac{1}{2}\lambda \text{Tr}(P)^2  + C\lambda^{-1}|P|^2,
	\end{aligned}
	\end{equation}
	where we have used (\ref{eq.mu-TrP}) in the last inequality.
	
	Combining (\ref{est.mulambda.mu}) and (\ref{est.mu.mulambda}),
	\begin{equation*}
	|\mu_\lambda(D,P) - \frac{1}{2}\lambda \text{Tr}(P)^2 - \mu(D,P)| \le C\lambda^{-1}|P|^2.
	\end{equation*}
	Now, let $D = \square_m$. Taking expectations and sending $m\to \infty$, in view of the definition of $\overline{A}_\lambda$ and $\widehat{A}$, we have
	\begin{equation*}
	\Big|\frac{1}{2} P\cdot \overline{A}_\lambda P - \frac{1}{2} P\cdot \widehat{A}P \Big| \le C\lambda^{-1}|P|^2.
	\end{equation*}
	This implies the desired estimate.
\end{proof}

\begin{proof}[Proof of Theorem \ref{thm.large-lambda}]
	Let $v^\e_\lambda$ be the weak solution of
	\begin{equation}\label{eq.rate.ve}
	\left\{
	\begin{aligned}
	\nabla\cdot (A^\e \nabla v^\e) + \nabla p^\e &= 0 \qquad &\txt{in }& D, \\
	\nabla\cdot v^\e & = \Ag{f}_D \qquad &\txt{in }& D, \\
	v^\e & = f \qquad &\txt{on }& \partial D,
	\end{aligned}
	\right.
	\end{equation}
	and let $v^0$ be the weak solution of
	\begin{equation}\label{eq.rate.v0}
	\left\{
	\begin{aligned}
	\nabla\cdot (\widehat{A} \nabla v^0) + \nabla p^0 &= 0 \qquad &\txt{in }& D, \\
	\nabla\cdot v^0 & = \Ag{f}_D \qquad &\txt{in }& D, \\
	v^0 & = f \qquad &\txt{on }& \partial D.
	\end{aligned}
	\right.
	\end{equation}
	It follows from Lemma \ref{lem.lambda.rate} that
	\begin{equation}\label{est.rate.1}
	\norm{u^\e_\lambda - v^\e}_{H^1(D)} + \norm{\lambda \nabla\cdot u^\e_\lambda - \average_D \lambda \nabla\cdot u^\e_\lambda - p^\e }_{L^2(D)} \le C\lambda^{-1} \norm{\nabla f}_{L^2(D)}.
	\end{equation}
	On the other hand, Theorem \ref{thm.Stokes.rate} implies
	\begin{equation}\label{est.rate.2}
	\norm{v^\e - v^0}_{L^2(D)} + \norm{p^\e - p^0}_{H^{-1}(D)} \le C\big( \e^{\beta(d-s)} + (\e \X'')^{\alpha s} \big) \|\nabla f\|_{L^{2+\delta}(D)},
	\end{equation}
	for some random variable $\X'' = \X''_s \le \mathcal{O}_s(C_0)$ with some absolute constant $C_0$.
	
	Next, using Lemma \ref{lem.lambda.rate} again, we obtain
	\begin{equation}\label{est.rate.3}
	\norm{v^0_\lambda - v^0}_{H^1(D)} + \norm{\lambda \nabla\cdot v^0_\lambda - \average_D \lambda \nabla\cdot v^0_\lambda - p^0 }_{L^2(D)} \le C\lambda^{-1} \norm{\nabla f}_{L^2(D)},
	\end{equation}
	where $v^0_\lambda$ is the weak solution of
	\begin{equation}\label{eq.hatA.rate}
	\left\{
	\begin{aligned}
	\nabla\cdot (\widehat{A} \nabla v^0_\lambda) + \lambda \nabla ( \nabla\cdot v^0_\lambda) &= 0 \qquad &\txt{in } &D, \\
	v^0_\lambda  &= f \qquad &\txt{on } &\partial D.
	\end{aligned}
	\right.
	\end{equation}
	
	Finally, comparing (\ref{eq.rate.0}) and (\ref{eq.hatA.rate})) and using Lemma \ref{lem.barA-hatA} and the energy estimate, we obtain
	\begin{equation}\label{est.rate.4}
	\norm{v^0_\lambda - u^0_\lambda}_{H^1(D)} + \norm{\lambda \nabla\cdot v^0_\lambda - \lambda\nabla\cdot u^0_\lambda }_{L^2(D)} \le C\lambda^{-1} \norm{\nabla f}_{L^2(D)},
	\end{equation}
	where we also used the fact
	\begin{equation*}
	\average_D (\lambda \nabla\cdot v^0_\lambda - \lambda\nabla\cdot u^0_\lambda ) = 0.
	\end{equation*}
	Combining (\ref{est.rate.1}), (\ref{est.rate.2}), (\ref{est.rate.3}) and (\ref{est.rate.4}), we obtain the announced estimate.
\end{proof}

\subsection{Global convergence rate}
Combing Theorem \ref{thm.small-lambda} and \ref{thm.large-lambda}, we obtain a global convergence rate uniform for any $\lambda \ge 0$.
\begin{theorem}\label{thm.global.rate}
	Let the same assumptions of Theorem \ref{thm.small-lambda} hold. There exist a random variable $\X = \X_{s,\lambda}$ and constants $\alpha,\beta,C>0$ (independent of $\e$ and $\lambda$) satisfying
	\begin{equation*}
	\X \le \mathcal{O}_s(C),
	\end{equation*}
	so that
	\begin{equation*}
	\begin{aligned}
	& \norm{u^\e_\lambda - u^0_\lambda}_{L^2(D)} + \norm{\lambda \nabla\cdot u^\e_\lambda - \lambda\nabla \cdot u^0_\lambda}_{H^{-1}(D)} \\
	&\qquad \le C \big(\e^{\beta(d-s)}  + (\X \e)^{\alpha s} \big)\norm{\nabla f}_{L^{2+\delta}(D)},
	\end{aligned}
	\end{equation*}
	where $u^\e_\lambda$ and $u^0_\e$ are the weak solutions of (\ref{eq.rate.e}) and (\ref{eq.rate.0}), respectively.
\end{theorem}

\begin{proof}
	Let $\X' = \X'_{s,\lambda} \le \mathcal{O}_s(C_0)$ be the random variable in (\ref{est.u.X'}) and (\ref{est.p.X'}). Let $\X'' = \X''_s \le \mathcal{O}_s(C_0)$ be the random variable in Theorem \ref{thm.large-lambda}.
	Choose $\X = \X_{s,\lambda} = \max \{ \X',\X'' \}$. Observe that $\X\le \mathcal{O}_s(C)$ for some absolute constant $C$. In fact, since $\X'\le \mathcal{O}_s(C_0)$ and $\X''\le \mathcal{O}_s(C_0)$,
	\begin{equation*}
	\E\big[ \exp\big((\X/C)^s\big) \big] \le \E\big[ \exp\big((\X'/C)^s\big) + \exp\big((\X''/C)^s\big) \big] \le 4,
	\end{equation*}
	as desired.
	
	Then, given any fixed $\lambda\ge 0$, by considering $\lambda < \e^{-\sigma}$ and $\lambda>\e^{-\sigma}$ separately with appropriate $\sigma$ depending on $s\in (0,d)$, Theorem \ref{thm.small-lambda} and \ref{thm.large-lambda} together implies that there exist $\alpha,\beta>0$ so that
	\begin{equation*}
	\begin{aligned}
	\norm{u^\e_\lambda - u^0_\lambda}_{L^2(D)} + \norm{\lambda \nabla\cdot u^\e_\lambda - \lambda\nabla \cdot u^0_\lambda}_{H^{-1}(D)} \le C\big( \e^{\beta(d-s)} + (\e \X)^{\alpha s} \big)\norm{\nabla f}_{L^{2+\delta}(D)},
	\end{aligned}
	\end{equation*}
	where $C$ is independent of $\e$ and $\lambda$.
\end{proof}

\section{Large-Scale Regularity: Interior Estimates}

In this section, we investigate the large-scale interior estimate for the system of elasticity
\begin{equation}\label{eq.int.B2}
\nabla\cdot (A^\e \nabla u^\e_\lambda) + \lambda\nabla ( \nabla\cdot u^\e_\lambda) = 0 \qquad\txt{in } B_2,
\end{equation}
where $B_2 = B_2(0)$ and $\lambda \ge 0$ is a constant.  

\subsection{Excess quantities and properties}
Let $u^\e_\lambda$ be a weak solution of (\ref{eq.int.B2}). We define two critical excess quantities. Define
\begin{equation*}
\begin{aligned}
\Phi(t) & =\frac{1}{t}\inf\limits_{q\in\R^d} \bigg( \average_{B_t } |u^\e_\lambda - q|^2
\bigg)^{1/2} \\
& \quad + \frac{1}{t} \norm{\lambda\nabla\cdot u_\lambda^\e - \average_{B_t} \lambda\nabla \cdot u_\lambda^\e }_{\underline{H}^{-1}(B_t)} + \sup_{k,\ell\in [1/4,1]} \bigg| \average_{B_{kt}} \lambda\nabla\cdot u_\lambda^\e -  \average_{B_{\ell t}} \lambda\nabla\cdot u_\lambda^\e\bigg|.
\end{aligned}
\end{equation*}
For any $v\in H^1(B_{t};\R^d)$, define
\begin{equation*}
\begin{aligned}
H(t;v) & = \frac{1}{t}\inf\limits_{\substack{M\in \R^{d\times d}\\ q\in\R^d}} \bigg( \average_{B_{t} } |v-Mx-q|^2 \bigg)^{1/2} \\
& \quad + \frac{1}{t} \norm{\lambda\nabla\cdot v - \average_{B_t} \lambda\nabla \cdot v }_{\underline{H}^{-1}(B_t)} + \sup_{k,\ell\in [1/4,1]} \bigg| \average_{B_{kt}} \lambda\nabla\cdot v -  \average_{B_{\ell t}} \lambda\nabla\cdot v\bigg|.
\end{aligned}
\end{equation*}
For simplicity, if $v = u^\e_\lambda$, we also write $H(r) = H(r;u^\e_\lambda)$. Then it is not hard to see 
\begin{equation}\label{est.basic.HPhi}
H(r) \le \Phi(r) \le C\bigg( \average_{B_r}|\nabla u^\e_\lambda|^2 \bigg)^{1/2} \le C\Phi(2r),
\end{equation}
where the second inequality follows from the Poincar\'{e} inequality and Lemma \ref{lem.Nabla}, and the third inequality follows from (\ref{est.WCaccioppoli.Int}).

Two useful properties of $H$ and $\Phi$ are given below.

\begin{lemma}\label{lem.h}
	There exists a function $h:(0,2) \mapsto [0,\infty)$ so that for any $r\in (0,1)$
	\begin{equation*}
	\left\{
	\begin{aligned}
	h(r) & \le C( H(r) + \Phi(r)) \\
	\Phi(r) &\le H(r) + h(r) \\
	\sup_{r\le s,t\le 2r} |h(s) - h(t)| &\le CH(2r).
	\end{aligned}\right.
	\end{equation*}
\end{lemma}
\begin{proof}
	Let $M_r$ be the matrix in $\R^{d\times d}$ that minimizes $H(r)$, namely,
	\begin{equation*}
	\begin{aligned}
	H(r) & = \frac{1}{r}\inf_{q\in \R^d} \bigg( \average_{B_{r} } |u^\e_\lambda -M_r x-q|^2 \bigg)^{1/2} \\
	&\quad + \frac{1}{t} \norm{\lambda\nabla\cdot u^\e_\lambda - \average_{B_r} \lambda\nabla \cdot u^\e_\lambda }_{\underline{H}^{-1}(B_r)} + \sup_{k,\ell\in [1/4,1]} \bigg| \average_{B_{kr}} \lambda\nabla\cdot u^\e_\lambda -  \average_{B_{\ell r}} \lambda\nabla\cdot u^\e_\lambda\bigg|.
	\end{aligned}
	\end{equation*}
	Define $h(r) = |M_r|$. In view of the definition, it is obvious that $\Phi(r) \le H(r) + h(r)$. Also,
	\begin{equation*}
	\begin{aligned}
	h(r) & \le \frac{C}{r} \inf_{q\in \R^d} \bigg( \average_{B_r} |M_r x + q|^2 \bigg)^{1/2} \\
	& \le \frac{C}{r} \inf_{q\in \R^d} \bigg( \average_{B_r} |u^\e_\lambda - M_r x - q|^2 \bigg)^{1/2} + \frac{1}{r} \inf_{q\in \R^d} \bigg( \average_{B_r} |u_\lambda^\e - q|^2 \bigg)^{1/2} \\
	&\le C( H(r) + \Phi(r)).
	\end{aligned}
	\end{equation*}
	Finally, if $r\le s,t\le 2r$,
	\begin{equation*}
	\begin{aligned}
	|h(s) - h(t)| & \le |M_s - M_t| \\
	&\le \frac{C}{r} \inf_{q\in \R^d} \bigg( \average_{B_r} |(M_s - M_t) x + q|^2 \bigg)^{1/2} \\
	& \le \frac{1}{r} \inf_{q\in \R^d} \bigg( \average_{B_r} |u^\e_\lambda - M_s x - q|^2 \bigg)^{1/2}  + \frac{1}{r} \inf_{q\in \R^d} \bigg( \average_{B_r} |u^\e_\lambda - M_t x - q|^2 \bigg)^{1/2} \\
	& \le \frac{C}{s}\inf\limits_{\substack{M\in \R^{d\times d}\\ q\in\R^d}} \bigg( \average_{B_s} |u^\e_\lambda - M x - q|^2 \bigg)^{1/2} + \frac{C}{t}\inf\limits_{\substack{M\in \R^{d\times d}\\ q\in\R^d}} \bigg( \average_{B_t} |u^\e_\lambda - M x - q|^2 \bigg)^{1/2}\\
	& \le \frac{C}{2r}\inf\limits_{\substack{M\in \R^{d\times d}\\ q\in\R^d}} \bigg( \average_{B_{2r}} |u^\e_\lambda - M x - q|^2 \bigg)^{1/2} \le CH(2r).
	\end{aligned}
	\end{equation*}
	The proof is complete.
\end{proof}
\begin{lemma}
	Suppose $u^\e_\lambda$ is a weak solution of (\ref{eq.int.B2}). There exists a constant $C$ so that for any $r\in (0,1)$, we have
	\begin{equation*}
	\sup_{s\in [r,2r]} \Phi(s) \le C\Phi(2r).
	\end{equation*}
\end{lemma}
\begin{proof}
	We estimate the three parts of $H(s)$ separately. First, it is obvious that
	\begin{equation*}
	\frac{1}{s}\inf_{q\in\R^d} \bigg( \average_{B_{s} } |u^\e_\lambda-q|^2 \bigg)^{1/2} \le \frac{C}{2r}\inf_{q\in\R^d} \bigg( \average_{B_{2r} } |u^\e_\lambda-q|^2 \bigg)^{1/2} \le C\Phi(2r).
	\end{equation*}
	Second, using the fact that $H_0^1(B_s) \subset H_0^1(B_{2r})$ and (\ref{est.H-1.L2}), we have
	\begin{equation*}
	\begin{aligned}
	& \frac{1}{s} \norm{\lambda\nabla\cdot u^\e_\lambda - \average_{B_s} \lambda\nabla \cdot u^\e_\lambda }_{\underline{H}^{-1}(B_s)} \\
	& \le \frac{1}{s} \norm{\lambda\nabla\cdot u^\e_\lambda - \average_{B_{2r}} \lambda\nabla \cdot u^\e_\lambda }_{\underline{H}^{-1}(B_s)} + \bigg| \average_{B_s} \lambda\nabla\cdot u^\e_\lambda - \average_{B_{2r}} \lambda\nabla\cdot u^\e_\lambda \bigg| \\
	& \le \frac{1}{s} \norm{\lambda\nabla\cdot u^\e_\lambda - \average_{B_{2r}} \lambda\nabla \cdot u^\e_\lambda }_{\underline{H}^{-1}(B_{2r})} + \sup_{k,\ell \in [1/4,1]}\bigg| \average_{B_{k2r}} \lambda\nabla\cdot u^\e_\lambda - \average_{B_{\ell 2r}} \lambda\nabla\cdot u^\e_\lambda \bigg| \\
	& \le 2\Phi(2r).
	\end{aligned}
	\end{equation*}
	Finally, it suffices to estimate
	\begin{equation*}
	\sup_{k,\ell \in [1/4,1]}\bigg| \average_{B_{ks}} \lambda\nabla\cdot u^\e_\lambda - \average_{B_{\ell s}} \lambda\nabla\cdot u^\e_\lambda \bigg| \le 2\sup_{k\in [1/4,1]} \bigg| \average_{B_{ks}} \lambda\nabla\cdot u^\e_\lambda - \average_{B_{r}} \lambda\nabla\cdot u^\e_\lambda \bigg|.
	\end{equation*}
	If $r\le ks \le 2r$, then by definition of $\Phi(2r)$, the above inequality is bounded by $2\Phi(2r)$. If $r/4 < ks < r$, we have
	\begin{equation*}
	\begin{aligned}
	\bigg| \average_{B_{ks}} \lambda\nabla\cdot u^\e_\lambda - \average_{B_{r}} \lambda\nabla\cdot u^\e_\lambda \bigg| & \le C\bigg( \average_{B_r}| \lambda\nabla\cdot u^\e_\lambda - \average_{B_{r}} \lambda\nabla\cdot u^\e_\lambda |^2 \bigg)^{1/2} \\
	& \le C\bigg( \average_{B_r} |\nabla u^\e_\lambda|^2 \bigg)^{1/2} \\
	& \le C\Phi(2r),
	\end{aligned}
	\end{equation*}
	where we have used Lemma \ref{lem.Nabla} and (\ref{est.WCaccioppoli.Int}). The desired estimate then follows readily.
\end{proof}

\subsection{Excess decay estimates}
This is done via a sequence of lemmas.
\begin{lemma}\label{lem.rater}
	For each $s\in (0,d)$, there exists a random variable $\X = \X_{s,\lambda}:\Omega\to [1,\infty)$ and a constant $C>0$ satisfying
	\begin{equation*}
	\X \le \mathcal{O}_s(C),
	\end{equation*}	
	such that for each $r\in (\e, 1)$, we have
	\begin{equation}\label{est.rate.Br}
	\bigg( \average_{B_r} |u^\e_\lambda - u^{0,r}_\lambda|^2 \bigg)^{1/2} + \norm{\lambda \nabla\cdot u^\e_\lambda - \lambda \nabla\cdot u^{0,r}_\lambda }_{\underline{H}^{-1}(B_r)} \le Cr \eta(\e \X/r) \Phi(4r),
	\end{equation}
	where
	\begin{equation}\label{def.Ys}
	\eta(\rho) = \rho^{\frac{1}{3}\min\{ \alpha s, \beta(d-s) \}},
	\end{equation}
	and $u_\lambda^{0,r}$ is the weak solution of
	\begin{equation}\label{eq.ulambda.0r}
	\left\{
	\begin{aligned}
	\nabla\cdot (\overline{A}_\lambda \nabla u^{0,r}_\lambda) + \lambda \nabla ( \nabla\cdot u^{0,r}_\lambda) &= 0 \qquad &\txt{in } &B_{r}, \\
	u^{0,r}_\lambda  &= u_\lambda^\e \qquad &\txt{on } &\partial B_{r}.
	\end{aligned}
	\right.
	\end{equation}
\end{lemma}
\begin{proof}
	By rescaling, it suffices to prove the result for $r = 1$. By the Meyers' estimate, we know that $u^\e_\lambda \in W^{1,2+\delta}(B_1;\R^d)$ for some $\delta>0$ and
	\begin{equation*}
	\norm{\nabla u^\e_\lambda}_{L^{2+\delta}(B_1)} \le C\norm{\nabla u^\e_\lambda}_{L^{2}(B_2)}.
	\end{equation*}
	Then the desired estimate follows immediately from Theorem \ref{thm.global.rate} and (\ref{est.WCaccioppoli.Int}).
\end{proof}

\begin{lemma}\label{lem.thetar}
	Let $u_\lambda^{0,r}$ be a solution of (\ref{eq.ulambda.0r}). There exists some $\theta \in (0,\frac{1}{4})$, depending only on $\Lambda$ and $d$, so that
	\begin{equation}\label{est.Hr.u0}
	H(\theta r; u_\lambda^{0,r}) \le \frac{1}{2} H(r; u_\lambda^{0,r}).
	\end{equation}
\end{lemma}

\begin{proof}
	First of all, for any $M\in \R^{d\times d}$ and $q\in \R^d$, $u_\lambda^{0,r}-Mx-q$ is a weak solution of (\ref{eq.ulambda.0r}) with constant coefficients. By the interior $C^{1,\alpha}$ estimate uniform in $\lambda$ (analogous to Theorem \ref{thm.C1a}), for any $\theta\in (0,1/4)$, we have
	\begin{equation*}
	H(\theta r; u_\lambda^{0,r}) \le C\theta^\alpha \inf\limits_{\substack{M\in \R^{d\times d}\\ q\in\R^d}} \bigg( \average_{B_{r/2}} |\nabla (u_\lambda^{0,r} -Mx-q) |^2 \bigg),
	\end{equation*}
	where we also used the fact (see (\ref{est.H-1.L2}))
	\begin{equation*}
	\norm{\lambda\nabla\cdot v - \average_{B_t} \lambda\nabla \cdot v }_{\underline{H}^{-1}(B_t)} \le Ct \bigg( \average_{B_t} |\lambda\nabla\cdot v - \average_{B_t} \lambda\nabla \cdot v|^2 \bigg)^{1/2}.
	\end{equation*}
	Combined with the generalized Caccioppoli inequality (\ref{est.WCaccioppoli.Int}), we obtain
	\begin{equation*}
	H(\theta r; u_\lambda^{0,r}) \le C\theta^\alpha  H(r;u_\lambda^{0,r}).
	\end{equation*}
	Finally, the desired estimate follows by choosing a proper $\theta$ so that $C\theta^\alpha \le 1/2$.
\end{proof}

\begin{lemma}\label{lem.Hr.iteration}
	There exists $\theta\in (0,\frac{1}{4})$, for each $r\in (\e\X,1)$,
	\begin{equation}\label{est.Hr}
	H(\theta r) \le \frac{1}{2} H(r) + C \sqrt[3]{\eta(\e \X/r)} \Phi(4r).
	\end{equation}
\end{lemma}
\begin{proof}
	This is a corollary of Lemma \ref{lem.rater} and Lemma \ref{lem.thetar}. Let $u_\lambda^{0,r} $ and $\theta\in (0,\frac{1}{4})$ be as in Lemma \ref{lem.thetar}. Observe that the triangle inequality for $H$ implies
	\begin{equation*}
	|H(t;f) - H(t;g)| \le H(t;f-g).
	\end{equation*}
	Applying this to (\ref{est.Hr.u0}), we get
	\begin{equation*}
	\begin{aligned}
	H(\theta r; u_\lambda^\e) & \le H(\theta r; u_\lambda^0) + H(\theta r; u_\lambda^\e - u_\lambda^0) \\
	& \le \frac{1}{2} H(r; u_\lambda^0) + H(\theta r; u_\lambda^\e - u_\lambda^0) \\
	& \le \frac{1}{2} H(r; u_\lambda^\e) + H(\theta r; u_\lambda^\e - u_\lambda^0) + H(r; u_\lambda^\e - u_\lambda^0).
	\end{aligned}
	\end{equation*}
	Then it suffices to estimate $H(\theta r; u_\lambda^\e - u_\lambda^0)$ with $\theta\in (0,1]$ by Lemma \ref{lem.rater}. Let us consider the three terms of the function $H(\theta r; u_\lambda^\e - u_\lambda^0)$ separately.
	
	First, it is clear from (\ref{est.rate.Br}) that
	\begin{equation}\label{est.H.1}
	\begin{aligned}
	\frac{1}{\theta r}\inf\limits_{\substack{M\in \R^{d\times d}\\ q\in\R^d}} \bigg( \average_{B_{\theta r} } |(u^\e_\lambda - u^{0,r}_\lambda)-Mx-q|^2 \bigg)^{1/2} & \le \frac{C_\theta}{r} \bigg( \average_{B_{r} } |u^\e_\lambda - u^{0,r}_\lambda|^2 \bigg)^{1/2} \\
	& \le C_\theta \eta(\e\X/r) \Phi(4r).
	\end{aligned}
	\end{equation}
	Next, since $H^1_0(B_{\theta r})$ is continuously embedded into $H^1_0(B_r)$, we trivially have
	\begin{equation}\label{est.H.2}
	\frac{1}{\theta r}\norm{\lambda \nabla\cdot u^\e_\lambda - \lambda \nabla\cdot u^{0,r}_\lambda }_{\underline{H}^{-1}(B_{\theta r})} \le \frac{1}{\theta r}\norm{\lambda \nabla\cdot u^\e_\lambda - \lambda \nabla\cdot u^{0,r}_\lambda }_{\underline{H}^{-1}(B_r)} \le C_\theta \eta(\e\X/r) \Phi(4r).
	\end{equation}
	Finally, it suffices to estimate
	\begin{equation}\label{eq.Ave.Br}
	\average_{B_{k\theta r}} (\lambda\nabla\cdot u^\e_\lambda - \lambda\nabla\cdot u^{0,r}_\lambda)
	\end{equation}
	for any $k\in (1/4,1)$. This follows from the following lemma:
	\begin{lemma}\label{lem.average2H-1}
		Let $F\in L^2(B_r)$. Then for any $\theta\in (0,1)$,
		\begin{equation*}
		\bigg| \average_{B_{\theta r}} F \bigg|^3 \le \frac{C_\theta}{r} \norm{F}_{\underline{H}^{-1}(B_{r})} \average_{B_r} |F|^2 .
		\end{equation*}
	\end{lemma}
	We postpone the proof of Lemma \ref{lem.average2H-1} and apply it to estimate (\ref{eq.Ave.Br})
	\begin{equation}\label{est.Ave.Bktr}
	\bigg| \average_{B_{k\theta r}} (\lambda\nabla\cdot u^\e_\lambda - \lambda\nabla\cdot u^{0,r}_\lambda)\bigg| \le C_{\theta k} \sqrt[3]{\eta(\e\X/r) \Phi(4r) } \bigg( \average_{B_r} |\lambda \nabla\cdot u^\e_\lambda -\lambda \nabla\cdot u^{0,r}_\lambda |^2 \bigg)^{1/3}.
	\end{equation}
	Now, let $t\in [3/2,2]$ be given as in Lemma \ref{lem.rater}. From the boundary condition of (\ref{eq.ulambda.0r})
	\begin{equation*}
	\average_{B_{tr}} (\lambda \nabla\cdot u^\e_\lambda -\lambda \nabla\cdot u^{0,r}_\lambda) = 0.
	\end{equation*}
	Consequently, by the triangle inequality and Lemma \ref{lem.Nabla}, one has
	\begin{equation*}
	\begin{aligned}
	& \bigg( \average_{B_r} |\lambda \nabla\cdot u^\e_\lambda -\lambda \nabla\cdot u^{0,r}_\lambda |^2 \bigg)^{1/2} \\
	&\le \bigg( \average_{B_{tr}} |\lambda \nabla\cdot u^\e_\lambda -\lambda \nabla\cdot u^{0,r}_\lambda |^2 \bigg)^{1/2} \\
	& \le \bigg( \average_{B_{tr}} |\lambda \nabla\cdot u^\e_\lambda -\average_{B_{tr}} \lambda \nabla\cdot u^\e_\lambda |^2 \bigg)^{1/2} + \bigg( \average_{B_{tr}} |\lambda \nabla\cdot u^{0,r}_\lambda -\average_{B_{tr}} \lambda \nabla\cdot u^{0,r}_\lambda |^2 \bigg)^{1/2} \\
	& \le C\bigg( \average_{B_{tr}} |\nabla u^\e_\lambda|^2 \bigg)^{1/2} \le C\Phi(4r).
	\end{aligned}
	\end{equation*}
	Inserting this into (\ref{est.Ave.Bktr}), we obtain
	\begin{equation}\label{est.H.3}
	\bigg| \average_{B_{k\theta r}} (\lambda\nabla\cdot u^\e_\lambda - \lambda\nabla\cdot u^{0,r}_\lambda)\bigg| \le C\sqrt[3]{\eta(\e\X/r)} \Phi(4r).
	\end{equation}
	Combining (\ref{est.H.1}), (\ref{est.H.2}) and (\ref{est.H.3}), we obtain the estimate for $H(\theta r; u_\lambda^\e - u_\lambda^0)$ and hence complete the proof.
\end{proof}

\begin{proof}[Proof of Lemma \ref{lem.average2H-1}]
	By rescaling, it suffices to prove the case $r = 1$. Let $\theta\in (0,1)$ be fixed and $t\in (0,\theta)$ to be determined. Suppose $\phi\in C^\infty_0(B_\theta)$ is a test function so that $\phi(x) = 1$ for $x\in B_{\theta-t}$ and $|\nabla \phi| \le C/t$. Then
	\begin{equation*}
	\begin{aligned}
	\bigg| \int_{B_\theta} F \bigg| &= \bigg| \int_{B_{\theta}} F\phi \bigg| + \bigg| \int_{B_\theta\setminus B_{\theta-t}} F(1-\phi)\bigg| \\
	& \le C \norm{F}_{\underline{H}^{-1}(B_1)} \norm{\phi}_{\underline{H}^1(B_1)} + |B_\theta\setminus B_{\theta-t}|^{1/2} \norm{F}_{L^2(B_1)} \\
	& \le Ct^{-1}\norm{F}_{\underline{H}^{-1}(B_1)} + Ct^{1/2}\norm{F}_{L^2(B_1)}.
	\end{aligned}
	\end{equation*}
	Choosing
	\begin{equation*}
	t = \frac{\norm{F}_{\underline{H}^{-1}(B_1)}^{2/3}}{\norm{F}_{L^2(B_1)}^{2/3}},
	\end{equation*}
	we obtain
	\begin{equation*}
	\bigg| \int_{B_\theta} F \bigg|^3 \le C\norm{F}_{\underline{H}^{-1}(B_1)} \norm{F}_{L^2(B_1)}^2,
	\end{equation*}
	as desired.
\end{proof}

\subsection{Iteration}
The following lemma is a generalization of \cite[Lemma 8.5]{Shen17}.
\begin{lemma}\label{lem.iteration}
	Suppose $\eta_i:(0,1] \mapsto [0,1]$, with $i = 1,2$, are increasing continuous functions so that $\eta_i(0) = 0$ and
	\begin{equation}\label{cond.Dini}
	\int_0^1 \frac{\eta_i(r)}{r} dr < \infty \qquad \text{for } i = 1,2.
	\end{equation}
	Let $H,\Phi,h:(0,2]\mapsto [0,\infty)$ be nonnegative functions. Suppose that there exist $\theta\in (0,1/4)$ and $C_0>0$ so that
	\begin{subequations}\label{est.HTP}
		\begin{align}
		H(\theta r) &\le \frac{1}{2} H(r) + C_0 \big\{ \eta_1(\e/r)+\eta_2(r)\big\} \Phi(4r) \label{subeqn:a}\\
		H(r) & \le C_0\Phi(r) \label{subeqn:b}\\ 
		h(r) & \le C_0( H(r) + \Phi(r))  \label{subeqn:c}\\
		\Phi(r) &\le C_0( H(r) + h(r)) \label{subeqn:d}\\
		\sup_{r\le t\le 2r} \Phi(t) & \le C_0\Phi(2r) \label{subeqn:h}\\
		\sup_{r\le s,t\le 2r} |h(s) - h(t)| &\le C_0 H(2r) \label{subeqn:e}
		\end{align}
	\end{subequations}
	for all $r\in  [\e,1]$. Then there is a constant $C>0$ depending only on $C_0$ and $\eta_i (i = 1,2)$ such that
	\begin{equation}\label{est.iteration}
	\int_{\e}^{2} \frac{H(r)}{r} dr + \sup_{\e\le r\le 2}\Phi(r) \le C\Phi(2).
	\end{equation}
\end{lemma}
\begin{proof}
	We start from an estimate of $h$. The assumption (\ref{subeqn:e}) on $h$ implies $h(r) \le h(2r) + CH(2r)$. Hence, given any $t\in (\e,1)$
	\begin{equation*}
	\begin{aligned}
	\int_t^{1} \frac{h(r)}{r} dr & \le \int_t^{1} \frac{h(2r)}{r} dr + C_0\int_t^{1} \frac{H(2r)}{r} dr \\
	& = \int_{2t}^{2} \frac{h(r)}{r} dr + C_0\int_{2t}^2 \frac{H(r)}{r} dr
	\end{aligned}
	\end{equation*}
	It follows in sequence from (\ref{subeqn:c}), (\ref{subeqn:b}) and (\ref{subeqn:h}) that
	\begin{equation*}
	\int_{t}^{2t} \frac{h(r)}{r} dr \le C\Phi(2) + C\int_{2t}^2 \frac{H(r)}{r} dr.
	\end{equation*}
	Hence, by using (\ref{subeqn:e}) again, for every $t\in (\e,1)$,
	\begin{equation}\label{est.ht}
	h(t) \le C\Phi(2) + C\int_{t}^2 \frac{H(r)}{r} dr.
	\end{equation}
	
	Let $\alpha>1$ be a large number and $\beta<1/2$ be a small number to be determined. Without loss of generality, assume $\e < \alpha^{-1}\beta$. Integrating (\ref{subeqn:a}) over the interval $[\alpha\e,\beta]$, we have
	\begin{equation}\label{est.H/r}
	\int_{\alpha \e}^{\beta} \frac{H(\theta r)}{r} dr \le \frac{1}{2}\int_{\alpha \e}^{\beta} \frac{H(r)}{r} dr + C_0 \int_{\alpha \e}^{\beta} \big\{ \eta_1(\e/r) + \eta_2(r) \big\} \Phi(4r) \frac{dr}{r}.
	\end{equation}
	Using the condition (\ref{subeqn:d}), we have
	\begin{equation*}
	\begin{aligned}
	\int_{\alpha \e}^{\beta} \big\{ \eta_1(\e/r) + \eta_2(r) \big\} \Phi(4r)\frac{dr}{r} \le C_0 \int_{\alpha\e}^{\beta} \big\{ \eta_1(\e/r) + \eta_2(r) \big\} (H(4r) + h(4r)) \frac{dr}{r}.
	\end{aligned}
	\end{equation*}
	Now, we observe that the monotonicity of $\eta_i$ and (\ref{est.ht}) imply
	\begin{equation*}
	\begin{aligned}
	\int_{\alpha\e}^{\beta} \big\{ \eta_1(\e/r) + \eta_2(r) \big\} H(4r) \frac{dr}{r}& = \int_{4\alpha\e}^{4\beta} \big\{ \eta_1(4\e/r) + \eta_2(r/4) \big\} H(r) \frac{dr}{r} \\
	&\le \big\{ \eta(\alpha^{-1}) + \eta(\beta) \big\} \int_{4\alpha\e}^{4\beta} H(r) \frac{dr}{r},
	\end{aligned}
	\end{equation*}
	and
	\begin{equation*}
	\begin{aligned}
	& \int_{\alpha\e}^{\beta} \big\{ \eta_1(\e/r) + \eta_2(r) \big\} h(4r) \frac{dr}{r} \\
	&\qquad \le C\Phi(2) \int_{\alpha\e}^{\beta} \big\{ \eta_1(\e/r) + \eta_2(r) \big\}\frac{dr}{r} + C\int_{\alpha\e}^{\beta} \big\{ \eta_1(\e/r) + \eta_2(r) \big\}\int_{r}^2 \frac{H(s)}{s} ds dr \\
	&\qquad \le C \Phi(2) + C\bigg\{ \int_{0}^{\alpha^{-1}} \frac{\eta_1(r)}{r} dr + \int_{0}^{\beta} \frac{\eta_2(r)}{r} dr \bigg\} \int_{\alpha \e}^2 \frac{H(s)}{s} ds.
	\end{aligned}
	\end{equation*}
	Combining the last four inequalities, we obtain
	\begin{equation*}
	\begin{aligned}
	& \int_{\theta\alpha \e}^{\theta\beta} \frac{H( r)}{r} dr \\
	&\le C\Phi(2) + \bigg[ \frac{1}{2} + C\bigg\{ \eta_1(\alpha^{-1}) + \eta_2(\beta) + \int_{0}^{\alpha^{-1}} \frac{\eta_1(r)}{r} dr + \int_{0}^{\beta} \frac{\eta_2(r)}{r} dr \bigg\} \bigg]\int_{\alpha \e}^2 \frac{H(r)}{r} dr.
	\end{aligned}
	\end{equation*}
	
	Next, by choosing $\alpha$ sufficiently large and $\beta$ sufficiently small so that
	\begin{equation*}
	\frac{1}{2} + C\bigg\{ \eta_1(\alpha^{-1}) + \eta_2(\beta) + \int_{0}^{\alpha^{-1}} \frac{\eta_1(r)}{r} dr + \int_{0}^{\beta} \frac{\eta_2(r)}{r} dr \bigg\} \le \frac{3}{4},
	\end{equation*}
	we have
	\begin{equation}\label{est.Hr.ab}
	\begin{aligned}
	\int_{\theta\alpha \e}^{2} \frac{H( r)}{r} dr &\le C\Phi(2) + 3 \int_{\theta\beta}^2 \frac{H(r)}{r} dr\le C\Phi(2),
	\end{aligned}
	\end{equation}
	where we also used (\ref{subeqn:b}) and (\ref{subeqn:h}) in the last inequality. In view of (\ref{est.ht}), this gives
	\begin{equation}\label{est.hr}
	h(r) \le C\Phi(2), \qquad \text{for any } r\in (\theta\alpha \e, 2).
	\end{equation}
	Therefore, for any $t\in (\theta\alpha \e, 2)$, by (\ref{subeqn:d}), (\ref{est.Hr.ab}) and (\ref{est.hr}),
	\begin{equation*}
	\int_t^{2t} \frac{\Phi(r)}{r} dr \le C_0\int_t^{2t} \frac{H(r)}{r} dr + C_0 \int_t^{2t} \frac{h(r)}{r} dr \le C\Phi(2).
	\end{equation*}
	In view of (\ref{subeqn:h}) again, this implies that
	\begin{equation}\label{est.Phir}
	\Phi(r) \le C\Phi(2), \qquad \text{for any } r\in (\theta\alpha \e, 2).
	\end{equation}
	Note that (\ref{est.Hr.ab}) and (\ref{est.Phir}) almost give the desired estimate (\ref{est.iteration}), except for the interval $(\e,\theta \alpha \e)$. However, since $\theta \alpha$ is a fixed number depending only on $C_0,\eta_1$ and $\eta_2$, by repeatedly using (\ref{subeqn:h}) finitely many times, we recover the estimate (\ref{est.Phir}) for $r\in (\e,\theta\alpha \e)$. Also, using (\ref{subeqn:b}), we recover 
	\begin{equation*}
	\int_{\e}^{\theta\alpha\e } \frac{H( r)}{r} dr \le C_0\int_{\e}^{\theta\alpha\e } \frac{\Phi( r)}{r} dr \le C_0(\theta \alpha-1) \sup_{r\in (\e,\theta\alpha )} \Phi(r) \le C\Phi(2).
	\end{equation*}
	This completes the proof.
\end{proof}

\begin{theorem}\label{thm.Lip}
	For any $s\in (0,d), \lambda\in [0,\infty)$, there exists a random variable $\X = \X_{s,\lambda}: \Omega\mapsto [1,\infty)$ satisfying
	\begin{equation*}
	\X \le \mathcal{O}_{s}(C),
	\end{equation*}
	such that if $r \in [\e \X, 1]$, then
	\begin{equation}\label{est.ue.Lip}
	\bigg( \average_{B_r} |\nabla u_\lambda^\e|^2 \bigg)^{1/2} + 	\bigg( \fint_{B_r} |\lambda \nabla\cdot u_\lambda^\e -  \fint_{B_2} \lambda \nabla\cdot u_\lambda^\e|^2 \bigg)^{1/2} \le C \bigg( \average_{B_2} |\nabla u_\lambda^\e|^2 \bigg)^{1/2}.
	\end{equation} 
\end{theorem}
\begin{proof}
	Let $H$ and $\Phi$ be defined as in Section 6.1. Note that the conditions of $H$ and $\Phi$ are verified by Lemma \ref{lem.h}, \ref{lem.thetar} and \ref{lem.Hr.iteration}. Let $\eta_1(r) = \sqrt[3]{\eta(r)} = r^{\sigma}$ with $\sigma>0$ and $\eta_2 = 0$. Applying Lemma \ref{lem.iteration} with $\e$ replaced by $\X\e$, we obtain
	\begin{equation}\label{est.Tr.T1}
	\int_{\e\X}^{2} \frac{H(r)}{r} dr + \sup_{\e\X \le r\le 2}\Phi(r) \le C\Phi(2).
	\end{equation}
	This, together with (\ref{est.basic.HPhi}), implies
	\begin{equation}\label{est.Lip.Br}
	\bigg( \average_{B_r} |\nabla u_\lambda^\e|^2 \bigg)^{1/2} \le C \bigg( \average_{B_2} |\nabla u_\lambda^\e|^2 \bigg)^{1/2}.
	\end{equation}
	
	To estimate the pressure, note that the first part of (\ref{est.Tr.T1}) gives
	\begin{equation*}
	\int_{\e\X}^{2} \sup_{k,\ell\in [1/4,1]} \bigg| \average_{B_{kt}} \lambda\nabla\cdot u^\e_\lambda - \average_{B_{\ell t}} \lambda\nabla\cdot u^\e_\lambda \bigg| \frac{dt}{t} \le C\bigg( \average_{B_2} |\nabla u_\lambda^\e|^2 \bigg)^{1/2}.
	\end{equation*}
	For a given $r\in (\e\X,1)$, there is an integer $N\ge 0$ so that $2^N r \in (1/2,1]$. Now, for any $0\le j\le N-1$, note that
	\begin{equation*}
	\bigg| \average_{B_{2^j r}} \lambda\nabla\cdot u^\e_\lambda - \average_{B_{2^{j+1} r}} \lambda\nabla\cdot u^\e_\lambda \bigg| \le 2\int_{2^{j+1} r}^{2^{j+2} r} \sup_{k,\ell\in [1/4,1]} \bigg| \average_{B_{kt}} \lambda\nabla\cdot u^\e_\lambda - \average_{B_{\ell t}} \lambda\nabla\cdot u^\e_\lambda \bigg| \frac{dt}{t}.
	\end{equation*}
	It follows that
	\begin{equation*}
	\begin{aligned}
	\bigg| \average_{B_{r}} \lambda\nabla\cdot u^\e_\lambda - \average_{B_{2^{N} r}}  \lambda\nabla\cdot u^\e_\lambda \bigg| & \le \sum_{j=0}^{N-1} \bigg| \average_{B_{2^j r}} \lambda\nabla\cdot u^\e_\lambda - \average_{B_{2^{j+1} r}}  \lambda\nabla\cdot u^\e_\lambda\bigg| \\
	& \le 2\int_{2r}^{2^{N+1} r} \sup_{k,\ell\in [1/4,1]} \bigg| \average_{B_{kt}} \lambda\nabla\cdot u^\e_\lambda - \average_{B_{\ell t}} \lambda\nabla\cdot u^\e_\lambda \bigg| \frac{dt}{t} \\
	& \le C\bigg( \average_{B_2} |\nabla u_\lambda^\e|^2 \bigg)^{1/2}.
	\end{aligned}
	\end{equation*}
	On the other hand, since $2^N r \in (1,2]$, Lemma \ref{lem.Nabla} implies
	\begin{equation*}
	\bigg| \average_{B_{2^N r}} \lambda\nabla\cdot u^\e_\lambda - \average_{B_{2}}  \lambda\nabla\cdot u^\e_\lambda \bigg| \le C\bigg( \average_{B_2} |\nabla u_\lambda^\e|^2 \bigg)^{1/2}.
	\end{equation*}
	Consequently, for any $r\in (\e\X,1)$,
	\begin{equation}\label{est.Pressure.rto2}
	\bigg| \average_{B_{r}} \lambda\nabla\cdot u^\e_\lambda - \average_{B_{2}}  \lambda\nabla\cdot u^\e_\lambda \bigg| \le  C\bigg( \average_{B_2} |\nabla u_\lambda^\e|^2 \bigg)^{1/2}.
	\end{equation}
	
	Finally, using (\ref{est.Pressure.rto2}), Lemma \ref{lem.Nabla} and (\ref{est.Lip.Br}), we obtain
	\begin{equation*}
	\begin{aligned}
	&\bigg( \fint_{B_r} |\lambda \nabla\cdot u_\lambda^\e -  \fint_{B_2} \lambda \nabla\cdot u_\lambda^\e|^2 \bigg)^{1/2} \\
	&\le \bigg( \fint_{B_r} |\lambda \nabla\cdot u_\lambda^\e -  \fint_{B_r} \lambda \nabla\cdot u_\lambda^\e|^2 \bigg)^{1/2} + \bigg| \fint_{B_r} \lambda \nabla\cdot u_\lambda^\e - \fint_{B_2} \lambda \nabla\cdot u_\lambda^\e \bigg| \\
	& \le C\bigg( \average_{B_r} |\nabla u_\lambda^\e|^2 \bigg)^{1/2} + C\bigg( \average_{B_2} |\nabla u_\lambda^\e|^2 \bigg)^{1/2} \\
	& \le C\bigg( \average_{B_2} |\nabla u_\lambda^\e|^2 \bigg)^{1/2},
	\end{aligned}
	\end{equation*}
	as desired.
\end{proof}

\begin{proof}[Proof of Theorem \ref{thm.main}]
	By (\ref{eq.reduction}) and (\ref{eq.splitting}), as well as the reduction in Section 5.1, the system may be reduced to the case with constant $\lambda = \lambda_0$. Then, by Theorem \ref{thm.Lip}, we obtain
	\begin{equation*}
	\bigg( \average_{B_r} |\nabla u_\lambda^\e|^2 \bigg)^{1/2} + \bigg( \fint_{B_r} |\lambda_0 \nabla\cdot u_\lambda^\e -  \fint_{B_2} \lambda_0 \nabla\cdot u_\lambda^\e|^2 \bigg)^{1/2}  \le C \bigg( \average_{B_2} |\nabla u_\lambda^\e|^2 \bigg)^{1/2}.
	\end{equation*}
	Finally, $\lambda_0$ may be replaced by $\lambda^\e$ due to the assumption (\ref{cond.incompr}) and the estimate of $|\nabla u^\e_\lambda|$ in the last inequality. This finishes the proof.
\end{proof}

\section{Large-Scale Regularity: Boundary Estimates}
In this section, we study the uniform large-scale estimate near the boundary which is of $\e$-scale $C^{1,\alpha}$ at $0\in \partial D$. Unfortunately, the previous argument for the interior estimate does not apply identically since the boundary below $\e$-scale is only Lipschitz, which means the local $C^{1,\alpha}$ regularity cannot be expected even for the homogenized system. To overcome this difficulty, we introduce a new idea of boundary perturbation to modify the excess decay iteration method which allows the boundary to be bumpy at microscopic scales.

\subsection{Boundary geometry}
Let $\alpha\in (0,1)$ be fixed and $D$ a bounded $\e$-scale $C^{1,\alpha}$ domain and $0\in \partial D$. As usual, define $D_t = D\cap B_t(0)$ and $\Delta_t = \partial D\cap B_t(0)$. By Definition \ref{def.C1a.e}, there exists a constant $C_0$ so that for any $t\in (\e,1)$, there exists a unit ``outer normal'' vector $n_t\in \mathbb{S}^{d-1}$ such that
\begin{equation*}
\begin{aligned}
T_t^{-} &:= \{x\in \R^d: x\cdot n_t < -C_0 t \zeta(t,\e) \} \cap B_t(0)\\
& \subset D_t \subset T_t^+:= \{x\in \R^d: x\cdot n_t < C_0 t\zeta(t,\e) \}\cap B_t(0),
\end{aligned}
\end{equation*}
where $\zeta(t,\e) = t^\alpha + (\e/t)^\alpha$. This particularly implies that both $T_t^-$ and $T_t^+$ approximate $D_t$ well at almost all scales with $\e\ll t \ll 1$. Moreover, 
\begin{equation*}
|T_t^+\setminus T_t^-|\lesssim t^{d} \zeta(t,\e) \simeq \zeta(t,\e) |D_t|.
\end{equation*}

The outer normal $n_t$ of the flat boundary of $T_t^{\pm}$ will play an important role in our proof. Intuitively, it represents a  macroscopically approximate direction perpendicular to the boundary near $0$ at $t$-scale and coincides with the usual outer normal if the boundary is smooth (i.e., $C^{1,\alpha}$). The following lemma shows that $n_t$ changes gently with $t\in (\e,1)$.
\begin{lemma}\label{lem.n_r}
	Let $\e\le s\le r\le 1$, then
	\begin{equation*}
	|n_r - n_s| \le C\frac{r\zeta(r,\e)}{s}.
	\end{equation*}
\end{lemma}
This is a simple geometric observation and the proof will be omitted.

\subsection{Excess quantities and properties}
Let $\lambda\ge 0$ and $u^\e_\lambda \in H^1(D_2;\R^d)$ be the weak solution of
\begin{equation}\label{eq.ue.bdry}
\left\{
\begin{aligned}
\nabla\cdot (A^\e \nabla u^\e_\lambda) +  \nabla (\lambda \nabla\cdot u^\e_\lambda) &= 0 \qquad &\txt{in } &D_{2}, \\
u_\lambda^\e  &=0  \qquad &\txt{on } & \Delta_{2}.
\end{aligned}
\right.
\end{equation}
We redefine $\Phi$ and $H$ in the boundary case as follows:
\begin{equation}\label{def.re.Phi}
\begin{aligned}
\Phi(t) & =\frac{1}{t} \bigg( \average_{D_t } |u^\e_\lambda|^2 \bigg)^{1/2} + \frac{1}{t} \norm{\lambda\nabla\cdot u^\e_\lambda - \average_{D_t} \lambda\nabla\cdot u^\e_\lambda }_{\underline{H}^{-1}(D_t)} \\
& \qquad + \sup_{k,\ell \in [1/4,1]} \bigg| \average_{D_{kt}} \lambda\nabla\cdot u^\e_\lambda - \average_{D_{\ell t}} \lambda\nabla\cdot u^\e_\lambda \bigg|
\end{aligned}
\end{equation}
and for $v\in H^1(D_t;\R^d)$
\begin{equation}\label{def.re.H}
\begin{aligned}
H(t;v) &= \frac{1}{t}\inf_{q\in\R^d} \bigg( \average_{D_{t} } |v - (n_t\cdot x)q|^2 \bigg)^{1/2} + \frac{1}{t} \norm{\lambda\nabla\cdot v - \average_{D_t} \lambda\nabla\cdot v }_{\underline{H}^{-1}(D_t)} \\
& \qquad + \sup_{k,\ell \in [1/4,1]} \bigg| \average_{D_{kt}} \lambda\nabla\cdot v - \average_{D_{\ell t}} \lambda\nabla\cdot v \bigg|.
\end{aligned}
\end{equation}
Put $H(t) = H(t;u^\e_\lambda)$ for short.

To apply Lemma \ref{lem.iteration}, as before, we need some basic properties of $\Phi$ and $H$. As in the interior case, we still have
\begin{equation}\label{est.HPhi}
H(r) \le \Phi(r) \qquad \text{and} \qquad \sup_{r\le s\le 2r}\Phi(s) \le C\Phi(2r).
\end{equation}
Also, we have the following key property which is slightly different from Lemma \ref{lem.h}.
\begin{lemma}\label{lem.h.bdry}
	There exists a function $h:(0,2) \mapsto [0,\infty)$ so that for any $r\in (0,1)$
	\begin{equation*}
	\left\{
	\begin{aligned}
	h(r) & \le C(H(r) + \Phi(r)) \\
	\Phi(r) &\le H(r) + h(r) \\
	\sup_{r\le s,t\le 2r} |h(s) - h(t)| &\le CH(2r) + C\zeta(r,\e)\Phi(2r).
	\end{aligned}\right.
	\end{equation*}
\end{lemma}
\begin{proof}
	The proofs for the first two inequality are similar to Lemma \ref{lem.h}. We only prove the third inequality here. 
	Let $q_r$ be the vector that minimizes $H(r)$, namely,
	\begin{equation}\label{def.re.H}
	\begin{aligned}
	H(r) &= \frac{1}{r} \bigg( \average_{D_{r} } |v - (n_r\cdot x)q_r|^2 \bigg)^{1/2} + \frac{1}{r} \norm{\lambda\nabla\cdot v - \average_{D_r} \lambda\nabla\cdot v }_{\underline{H}^{-1}(D_r)} \\
	& \qquad + \sup_{k,\ell \in [1/4,1]} \bigg| \average_{D_{kr}} \lambda\nabla\cdot v - \average_{D_{\ell r}} \lambda\nabla\cdot v \bigg|.
	\end{aligned}
	\end{equation}
	Define $h(r) = |q_r|$. Then $h(r) \le C(H(r) + \Phi(r)) \le C\Phi(r)$. Let $s,t\in [r,2r]$, one has
	\begin{equation}\label{est.qsqt}
	\begin{aligned}
	|q_s - q_t| &\le \frac{C}{r} \bigg( \average_{D_{r} } |(n_{2r}\cdot x)(q_s - q_t)|^2 \bigg)^{1/2} \\
	& \le \frac{C}{r} \bigg( \average_{D_{r} } |u^\e_\lambda - (n_{2r}\cdot x)q_s)|^2 \bigg)^{1/2} + \frac{C}{r} \bigg( \average_{D_{r} } |u^\e_\lambda - (n_{2r}\cdot x)q_t)|^2 \bigg)^{1/2}.
	\end{aligned}
	\end{equation}
	We estimate the first term. Using Lemma \ref{lem.n_r} and (\ref{est.HPhi}), we have
	\begin{equation*}
	\begin{aligned}
	& \frac{1}{r} \bigg( \average_{D_{r} } |u^\e_\lambda - (n_{2r}\cdot x)q_s)|^2 \bigg)^{1/2} \\
	&\le \frac{C}{r} \bigg( \average_{D_{s} } |u^\e_\lambda - (n_{s}\cdot x)q_s)|^2 \bigg)^{1/2} + |n_{2r}-n_s||q_s| \\
	& \le \frac{C}{s} \inf_{q\in \R^d}\bigg( \average_{D_{s} } |u^\e_\lambda - (n_{s}\cdot x)q)|^2 \bigg)^{1/2} + C\zeta(2r,\e) \Phi(2r) \\
	& \le \frac{C}{s}  \bigg( \average_{D_{s} } |u^\e_\lambda - (n_{s}\cdot x)q_{2r})|^2 \bigg)^{1/2} + C\zeta(2r,\e) \Phi(2r) \\
	& \le \frac{C}{2r} \bigg( \average_{D_{2r} } |u^\e_\lambda - (n_{2r}\cdot x)q_{2r})|^2 \bigg)^{1/2} + |n_{2r} - n_s| |q_{2r}| + C\zeta(2r,\e) \Phi(2r) \\
	& \le CH(2r) + C\zeta(r,\e) \Phi(2r),
	\end{aligned}
	\end{equation*}
	where we have used the fact $\zeta(2r,\e) \le C\zeta(r,\e)$ since $\zeta(r,\e) = r^\alpha+(\e/r)^\alpha$. The estimate for the second term of (\ref{est.qsqt}) is similar. Hence, for any $s,t\in [r,2r]$,
	\begin{equation*}
	|h(s) - h(t)| \le |q_s -q_t| \le CH(2r) + C\zeta(r,\e) \Phi(2r),
	\end{equation*}
	as desired.
\end{proof}

\subsection{Boundary perturbation}
We construct an approximate solution in $T_t^+$ at each scale $t > \e$ which admits a better regularity due to the smoothness of the boundary. Let $u_\lambda^\e\in H^1(D_{2};\R^d)$ solves (\ref{eq.ue.bdry}).
We extend the function $u^\e_\lambda$ to the entire $B_{2}$ by
\begin{equation}\label{def.tue}
\widetilde{u}^\e_\lambda(x) = \left\{ 
\begin{aligned}
& u^\e_\lambda(x)  &\quad &\text{for } x\in D_{2}\\
& 0  &\quad &\text{for } x\in B_{2}\setminus D_{2}.
\end{aligned}
\right.
\end{equation}
Then $\widetilde{u}^{\e}_\lambda\in H^1(B_{2};\R^d)$ and $\norm{\widetilde{u}^{\e}_\lambda}_{H^1(B_{r})} = \norm{u^{\e}_\lambda}_{H^1(D_{r})}$ for any $r\in (0,2)$. Now, fix $r\in (\e\X_s,1)$, let $v^\e_\lambda$ be the weak solution of
\begin{equation}\label{eq.v.approx}
\left\{
\begin{aligned}
\nabla\cdot (A^\e \nabla v^\e_\lambda) + \lambda \nabla ( \nabla\cdot v^\e_\lambda) &= 0 \qquad &\txt{in } &T_{2r}^+, \\
v_\lambda^\e  &=\widetilde{u}^{\e}_\lambda  \qquad &\txt{on } & \partial T_{2r}^+.
\end{aligned}
\right.
\end{equation}

The following is the key lemma.
\begin{lemma}\label{lem.u-v}
	There exists $\gamma>0$, depending only on $d,\Lambda$ and the Lipschitz constant of $\Delta_2$, so that for any $\lambda>0$,
	\begin{equation}\label{est.approx.Dr}
	\bigg( \average_{T_{2r}^+} |\nabla \widetilde{u}^\e_\lambda - \nabla v^\e_\lambda|^2 \bigg)^{1/2} + \bigg( \average_{T_{2r}^+} |\lambda \nabla\cdot \widetilde{u}^\e_\lambda - \lambda \nabla\cdot v^\e_\lambda|^2 \bigg)^{1/2} \le C\zeta(r,\e)^\gamma \bigg( \average_{D_{4r}} |\nabla u^\e_\lambda|^2 \bigg)^{1/2}.
	\end{equation}
\end{lemma}

\begin{proof}
	First of all, by the system (\ref{eq.v.approx}) and Theorem \ref{lem.energy.elasticity}, we have
	\begin{equation}\label{est.vel.T2r}
	\begin{aligned}
	&\norm{\nabla v^\e_\lambda}_{L^2(T_{2r}^+)}+\lambda \norm{\nabla\cdot v^\e_\lambda - \average_{T_{2r}^+} \nabla\cdot v^\e_\lambda }_{L^2(T_{2r}^+)} \\
	& \qquad \le C\norm{\nabla \widetilde{u}^\e_\lambda}_{L^2(T_{2r}^+)} = C\norm{\nabla u^\e_\lambda}_{L^2(D_{2r})}.
	\end{aligned}
	\end{equation}
	Let $\phi_r$ be a cut-off function so that $\phi_r(x) = 1$ on $\{ x\cdot n_{2r} \le -4C_0r\zeta(2r,\e) \}$ and $\phi_r(x) = 0$ on $\{x\cdot n_{2r} >-2C_0r\zeta(2r,\e) \}$ and $|\nabla \phi_r(x)| \le C[r\zeta(r,\e)]^{-1}$. Then it is not hard to show
	\begin{equation}\label{eq.tu-v}
	\begin{aligned}
	&\int_{B_{2r}} A^\e \nabla (\widetilde{u}^\e_\lambda -  v^\e_\lambda)\cdot \nabla(\widetilde{u}^\e_\lambda -  v^\e_\lambda) + \lambda \int_{B_{2r}} |\nabla\cdot (\widetilde{u}^\e_\lambda - v^\e_\lambda)|^2 \\
	& = \int_{B_{2r}} A^\e \nabla(\widetilde{u}^\e_\lambda - v^\e_\lambda)\cdot \nabla \Big[ ( \widetilde{u}^\e_\lambda - v^\e_\lambda) \phi_r \Big] + \lambda\int_{B_{2r}}\nabla\cdot (\widetilde{u}^\e_\lambda - v^\e_\lambda) \nabla\cdot \Big[ (\widetilde{u}^\e_\lambda -  v^\e_\lambda) \phi_r \Big] \\
	& \qquad + \int_{B_{2r}} A^\e \nabla(\widetilde{u}^\e_\lambda - v^\e_\lambda)\cdot \nabla \Big[ ( \widetilde{u}^\e_\lambda - v^\e_\lambda) (1-\phi_r) \Big] \\
	&\qquad + \lambda\int_{B_{2r}}\nabla\cdot (\widetilde{u}^\e_\lambda - v^\e_\lambda) \nabla\cdot \Big[ (\widetilde{u}^\e_\lambda -  v^\e_\lambda) (1-\phi_r) \Big].
	\end{aligned}
	\end{equation}
	Now observe that $(\widetilde{u}^\e_\lambda -  v^\e_\lambda) \phi_r \in H_0^1(T_{2r}^-;\R^d)$. Hence, by the integration by parts, the first two terms in the right-hand side of the last identity vanishes since $\widetilde{u}^\e_\lambda -  v^\e_\lambda$ is a weak solution in $T_{2r}^-$ which is a subset of $D_{2r} \subset T_{2r}^+$. So it suffices to estimate the third and forth terms. We will estimate the forth term only and the estimate of the third term is similar and easier. Note that $\widetilde{u}^\e_\lambda$ and $v^\e_\lambda$ are supported in $D_{2r}$ and $T_{2r}^+$, respectively. Thus $\widetilde{u}^\e_\lambda -  v^\e_\lambda \in H^1_0(B_{2r};\R^d)$ yields
	\begin{equation*}
	\int_{B_{2r}} \nabla\cdot \Big[ (\widetilde{u}^\e_\lambda -  v^\e_\lambda) (1-\phi_r) \Big] = 0.
	\end{equation*}
	Hence, by inserting constants in the forth term of (\ref{eq.tu-v}) and using the triangle inequality, the Cauchy–Schwarz inequality and (\ref{est.vel.T2r}), we obtain
	\begin{equation}\label{est.tuv.1-phi}
	\begin{aligned}
	&\bigg| \lambda\int_{B_{2r}}\nabla\cdot (\widetilde{u}^\e_\lambda - v^\e_\lambda) \nabla\cdot \Big[ (\widetilde{u}^\e_\lambda -  v^\e_\lambda) (1-\phi_r) \Big] \bigg| \\
	& \le \bigg\{ \lambda \bigg( \int_{D_{2r}} |\nabla\cdot \widetilde{u}^\e_\lambda - \average_{D_{2r}} \nabla\cdot \widetilde{u}^\e_\lambda|^2 \bigg)^{1/2} + \lambda \bigg( \int_{T_{2r}^+} |\nabla\cdot v^\e_\lambda - \average_{T_{2r}^+} \nabla\cdot v^\e_\lambda|^2 \bigg)^{1/2} \bigg\}  \\
	& \qquad \times \bigg( \int_{T_{2r}^+} \Big| \nabla\cdot \Big[ (\widetilde{u}^\e_\lambda -  v^\e_\lambda) (1-\phi_r) \Big] \Big|^2 \bigg)^{1/2} \\
	& \le C\norm{\nabla u^\e_\lambda}_{L^2(D_{2r})} \times \bigg( \int_{T_{2r}^+} \Big| \nabla\cdot \Big[ (\widetilde{u}^\e_\lambda -  v^\e_\lambda) (1-\phi_r) \Big] \Big|^2 \bigg)^{1/2}.
	\end{aligned}
	\end{equation}
	
	Next, by the Poincar\'{e} inequality
	\begin{equation*}
	\begin{aligned}
	\bigg( \int_{T_{2r}^+} \big|  \nabla\cdot\big[ \widetilde{u}^\e_\lambda (1-\phi_r) \big] \big|^2 \bigg)^{1/2} &\le \bigg( \int_{D_{2r}\cap  \{x\cdot n_{2r} \ge -4C_0r\zeta(2r,\e) \} } \big|  \nabla \widetilde{u}^\e_\lambda  \big|^2 \bigg)^{1/2} \\
	& \qquad + C[r\zeta(2r,\e)]^{-1} \bigg( \int_{D_{2r}\cap \{ x\cdot n_{2r} \ge -4C_0r\zeta(2r,\e) \} } \big| \widetilde{u}^\e_\lambda  \big|^2 \bigg)^{1/2} \\
	& \le C\bigg( \int_{D_{2r}\cap \{ x\cdot n_{2r} \ge -4C_0r\zeta(2r,\e) \} } \big|  \nabla \widetilde{u}^\e_\lambda  \big|^2 \bigg)^{1/2}.
	\end{aligned}
	\end{equation*}
	Note that $D_{2r}\cap \{ x\cdot n_{2r} \ge -4C_0r\zeta(2r,\e) \} \subset B_{2r} \cap \{  -4C_0r\zeta(2r,\e) \le x\cdot n_{2r} < 2C_0r\zeta(2r,\e) \}$, which implies $|D_{2r}\cap \{ x\cdot n_{2r} \ge -2C_0 r\zeta(2r,\e) \}| \le Cr^d\zeta(r,\e)$. It follows from the H\"{o}lder's inequality and the Meyers' estimate (Theorem \ref{thm.Meyers.Elasticity}) that
	\begin{equation*}
	\begin{aligned}
	& \bigg( \int_{D_{2r}\cap \{ x\cdot n_{2r} \ge -4C_0r\zeta(2r,\e) \} } \big|  \nabla \widetilde{u}^\e_\lambda  \big|^2 \bigg)^{1/2} \\
	&\qquad \le |D_{2r}\cap \{ x\cdot n_{2r} \ge -4C_0r\zeta(2r,\e) \}|^{\frac{1}{2} - \frac{1}{p_0}} \bigg( \int_{D_{2r} } \big|  \nabla \widetilde{u}^\e_\lambda  \big|^{p_0} \bigg)^{1/{p_0}} \\
	& \qquad \le C[\zeta(r,\e)]^{\frac{1}{2} - \frac{1}{p_0}}  \bigg( \int_{D_{4r} } \big|  \nabla \widetilde{u}^\e_\lambda  \big|^{2} \bigg)^{1/{2}}.
	\end{aligned}
	\end{equation*}
	Consequently, we arrive at
	\begin{equation*}
	\bigg( \int_{T_{2r}^+} \Big| \nabla\cdot \Big[ (\widetilde{u}^\e_\lambda -  v^\e_\lambda) (1-\phi_r) \Big] \Big|^2 \bigg)^{1/2} \le C[\zeta(r,\e)]^{\frac{1}{2} - \frac{1}{p_0}} \bigg( \int_{D_{4r} } \big|  \nabla \widetilde{u}^\e_\lambda  \big|^{2} \bigg)^{1/{2}}.
	\end{equation*}
	Substituting this into (\ref{est.tuv.1-phi}), we obtain the estimate of (\ref{eq.tu-v}). Particularly,
	\begin{equation}\label{est.Dtu-Dv}
	\bigg( \int_{T_{2r}^+} |\nabla \widetilde{u}^\e_\lambda - \nabla v^\e_\lambda|^2 \bigg)^{1/2} \le C\zeta(r,\e)^\gamma \bigg( \int_{D_{4r}} |\nabla u^\e_\lambda|^2 \bigg)^{1/2}.
	\end{equation}
	This implies the first part of (\ref{est.approx.Dr}). 
	
	To obtain the estimate for pressure, note that it follows easily from (\ref{est.Dtu-Dv}) and Lemma \ref{lem.DwDpi} that
	\begin{equation*}
	\bigg( \average_{D_{2r}} |\lambda \nabla\cdot \widetilde{u}^\e_\lambda - \lambda \nabla\cdot v^\e_\lambda|^2 \bigg)^{1/2} \le C\zeta(r,\e)^\gamma \bigg( \average_{D_{4r}} |\nabla u^\e_\lambda|^2 \bigg)^{1/2}.
	\end{equation*}
	On the other hand, in $T_{2r}^{+}\setminus D_{2r}$, using the Meyers' estimate for $\lambda\nabla\cdot v^\e_\lambda$ again and the fact $|T_{2r}^+\setminus D_{2r}| \le Cr\zeta(r,\e)$, we have
	\begin{equation*}
	\bigg( \average_{T_{2r}^+\setminus D_{2r}} |\lambda \nabla\cdot v^\e_\lambda|^2 \bigg)^{1/2} \le C\zeta(r,\e)^\gamma \bigg( \average_{D_{4r}} |\nabla u^\e_\lambda|^2 \bigg)^{1/2}.
	\end{equation*}
	Combining these estimates and using the fact $\widetilde{u}^\e_\lambda = 0$ in $T^+_{2r}\setminus D_{2r}$, we obtain the desired estimate.
\end{proof}

\begin{remark}\label{rmk.u-v}
	Since $\widetilde{u}^\e_\lambda - v^\e_\lambda = 0$ on $B_{2r}\setminus T_{2r}^+$, whose volume is comparable to $ r^d$, the Poincar\'{e} inequality implies
	\begin{equation}\label{est.T2r+.uv}
	\bigg( \average_{T_{2r}^+} | u^\e_\lambda - v^\e_\lambda|^2 \bigg)^{1/2} \le Cr\zeta(r,\e)^\gamma \bigg( \average_{D_{4r}} |\nabla u^\e_\lambda|^2 \bigg)^{1/2}.
	\end{equation}
\end{remark}

\subsection{Excess decay estimate and iteration}
Now we consider the homogenization of the approximate solution $v^\e_\lambda$. Let $v^0_\lambda$ be the weak solution of the corresponding homogenized system
\begin{equation}\label{eq.v0.approx}
\left\{
\begin{aligned}
\nabla\cdot (\overline{A}_\lambda \nabla v^0_\lambda) + \lambda \nabla ( \nabla\cdot v^0_\lambda) &= 0 \qquad &\txt{in } &T_{2r}^+, \\
v_\lambda^0  &=v^{\e}_\lambda  \qquad &\txt{on } & \partial (T_{2r}^+).
\end{aligned}
\right.
\end{equation}
Note that Theorem \ref{thm.Meyers.Elasticity} implies  $\widetilde{u}^\e_\lambda\in W^{1,2+\delta}(T_{2r}^+;\R^d)$ and $v^\e_\lambda \in W^{1,2+\delta}(T_{2r}^+;\R^d)$. Therefore, Theorem \ref{thm.global.rate} implies that there exists a random variable $\X = \X_{s,\lambda} \le O_s(C)$ with $s\in (0,d)$ such that
\begin{equation}\label{est.bdryRate}
\bigg( \average_{T_{2r}^+} |v^\e_\lambda - v^0_\lambda|^2 \bigg)^{1/2} + \norm{\lambda \nabla\cdot v^\e_\lambda - \lambda \nabla\cdot v^0_\lambda }_{\underline{H}^{-1}(T_{2r}^+)} \le Cr \eta(\e \X/r) \bigg( \fint_{D_{4r}} |\nabla u^\e_\lambda|^2 \bigg)^{1/2}.
\end{equation}

Next, we consider the smoothness of the homogenized solution. To this end, we will take a special form adapted to the boundary situation:
\begin{equation*}
\begin{aligned}
\widetilde{H}(r,\theta;v^0_\lambda) &= \frac{1}{\theta r}\inf_{q\in \R^d} \bigg( \average_{T_{2r}^+\cap B_{\theta r}} |v^0_\lambda - (n_{\theta r}\cdot x)q|^2 \bigg)^{1/2} \\
& \qquad + \frac{1}{\theta r} \norm{ \lambda\nabla\cdot v^0_\lambda - \average_{T_{2r}^+\cap B_{\theta r}}\lambda\nabla\cdot v^0_\lambda }_{\underline{H}^{-1}(T_{2r}^+\cap B_{\theta r})} \\
& \qquad+ \sup_{k,\ell \in [1/4,1]} \bigg| \average_{T_{2r}^+\cap B_{k\theta r}}\lambda\nabla\cdot v^0_\lambda - \average_{T_{2r}^+\cap B_{\ell\theta r}}\lambda\nabla\cdot v^0_\lambda \bigg|.
\end{aligned}
\end{equation*}
Recall that $n_{2r}$ is roughly the outward normal of the flat boundary of $T_{2r}^+$, which can be viewed as a large-scale normal vector of $\partial D$ at $0$. The linear function $(n_{\theta r}\cdot x)q$ is particularly effective in approximating $v_\lambda^0$, since $v^0_\lambda$ vanishes on the flat boundary and hence the tangent derivatives on the flat boundary vanishes.

\begin{lemma}\label{lem.bdry.v0}
	There exists a constant $\theta\in (0,1/4)$, depending only on $d$ and $\Lambda$, so that for any $r\in (\e,1)$
	\begin{equation*}
	\widetilde{H}(r,\theta;v^0_\lambda) \le \frac{1}{2} \widetilde{H}(r,2;v^0_\lambda) + C \zeta(r,\e) \bigg( \average_{T_{2r}^+} |\nabla v^0_\lambda|^2 \bigg)^{1/2}.
	\end{equation*}
\end{lemma}
\begin{proof}
	First of all, by the $C^{1,\alpha}$ regularity of $v^0_\lambda$ (Theorem \ref{thm.C1a}), we know
	\begin{equation}\label{est.C1a,v0}
	[\nabla v^0_\lambda]_{C^\alpha(T_{2r}\cap B_{r/2})} + [\lambda\nabla\cdot v^0_\lambda]_{C^\alpha(T_{2r}\cap B_{r/2})}\le Cr^{-\alpha} \bigg( \average_{T_{2r}^+\cap B_r} |\nabla v^0_\lambda|^2 \bigg)^{1/2}.
	\end{equation}
	Let $x_{2r}$ be the point on $\partial T_{2r}^+ \cap B_r$ (flat boundary) that is the closest to the origin. By our assumption, $|x_{2r}| \le Cr\zeta(r,\e)$. Clearly, since $v^0_\lambda$ is identically zero on the flat boundary, $v^0_\lambda(x_{2r}) = 0$ and the tangential derivative vanishes, i.e.,
	\begin{equation*}
	(I_{d\times d}-n_{2r}\otimes n_{2r}) \nabla v^0_\lambda(x_{2r}) = 0,
	\end{equation*}
	where $I_{d\times d}$ is the $d\times d$ identity matrix. It follows that
	\begin{equation}\label{eq.Dv.xr}
	\nabla v^0_\lambda(x_{2r}) = (n_{2r}\otimes n_{2r}) \nabla v^0_\lambda(x_{2r}) =  n_{2r} (n_{2r}\cdot \nabla v^0_\lambda(x_{2r})).
	\end{equation}
	Now, if $\theta\in (0,1/4)$ (to be determined) and $x\in T_{2r}^+ \cap B_{\theta r}$, (\ref{est.C1a,v0}) implies
	\begin{equation*}
	|v^0_\lambda(x) - v^0_\lambda(x_{2r}) - (x-x_{2r})\cdot \nabla v^0_\lambda(x_{2r})| \le C(\theta r)^{1+\alpha} r^{-\alpha} \bigg( \average_{T_{2r}^+\cap B_r} |\nabla v^0_\lambda|^2 \bigg)^{1/2}.
	\end{equation*}
	Combined with (\ref{eq.Dv.xr}), this implies
	\begin{equation*}
	\begin{aligned}
	|v_\lambda^0(x) - (x\cdot n_{2r})(n_{2r}\cdot \nabla v^0_\lambda(x_{2r}))| &\le C\theta^\alpha(\theta r) \bigg( \average_{T_{2r}^+\cap B_r} |\nabla v^0_\lambda|^2 \bigg)^{1/2} + |x_{2r}| \bigg( \average_{T_{2r}^+\cap B_r} |\nabla v^0_\lambda|^2 \bigg)^{1/2} \\
	& \le C\theta r \big( \theta^\alpha + \theta^{-1} \zeta(r,\e) \big) \bigg( \average_{T_{2r}^+\cap B_r} |\nabla v^0_\lambda|^2 \bigg)^{1/2}.
	\end{aligned}
	\end{equation*}
	Moreover, in view of Lemma \ref{lem.n_r}, we may replace $(x\cdot n_{2r})$ by $(x\cdot n_{\theta r})$ with the same error as above. Consequently, one arrives at
	\begin{equation*}
	\frac{1}{\theta r}\inf_{q\in \R^d} \bigg( \average_{T_{2r}^+\cap B_{\theta r}} |v^0_\lambda - (n_{\theta r}\cdot x)q|^2 \bigg)^{1/2} \le C \big( \theta^\alpha + \theta^{-1} \zeta(r,\e) \big) \bigg( \average_{T_{2r}^+\cap B_r} |\nabla v^0_\lambda|^2 \bigg)^{1/2}.
	\end{equation*}
	Together with the estimates of the pressure (which is obvious due to (\ref{est.C1a,v0})), we have
	\begin{equation*}
	\widetilde{H}(r,\theta; v^0_\lambda) \le C\big( \theta^\alpha + \theta^{-1} \zeta(r,\e) \big) \bigg( \average_{T_{2r}^+\cap B_r} |\nabla v^0_\lambda|^2 \bigg)^{1/2}.
	\end{equation*}
	
	Next, observe that $v^0_\lambda - n_{2r}\cdot (x-x_{2r}) q$ is a weak solution vanishing on the flat boundary for any $q\in \R^d$. Then (\ref{est.WCacc.bdry}) gives
	\begin{equation*}
	\begin{aligned}
	\bigg( \average_{T_{2r}^+\cap B_r} |\nabla v^0_\lambda|^2 \bigg)^{1/2} &\le \frac{C}{r} \inf_{q\in \R^d} \bigg( \average_{T_{2r}^{+} \cap B_{3r/2} } | v^0_\lambda - n_{2r}\cdot (x-x_{2r}) q|^2 \bigg)^{1/2} \\
	& \qquad + \frac{C}{r} \norm{\lambda \nabla\cdot v^0_\lambda - \average_{T_{2r}^{+}}\lambda \nabla\cdot v^0_\lambda  }_{\underline{H}^{-1}(T_{2r}^{+})} \\
	& \qquad + C\sup_{t\in [1/2,2]} \bigg| \average_{T_{2r}^{+}\cap B_{kr}} \lambda\nabla \cdot v^0_\lambda - \average_{T_{2r}^{+}\cap B_{\ell r}} \lambda\nabla \cdot v^0_\lambda \bigg| \\
	& \le C\widetilde{H}(r,2;v^0_\lambda) + C|x_{2r}|\bigg( \average_{T_{2r}^+} |\nabla v^0_\lambda|^2 \bigg)^{1/2}.
	\end{aligned}
	\end{equation*}
	Combining the previous two estimates, we obtain
	\begin{equation*}
	\widetilde{H}(r,\theta; v^0_\lambda) \le C\theta^\alpha \widetilde{H}(r,\theta; v^0_\lambda) + C_\theta \zeta(r,\e) \bigg( \average_{T_{2r}^+} |\nabla v^0_\lambda|^2 \bigg)^{1/2}.
	\end{equation*}
	Choosing $\theta \in (0,1/4)$ small enough, we obtain the desired estimate.
\end{proof}

\begin{lemma}\label{lem.approx.ve}
	There exists $\theta\in (0,1/4)$ so that for any $r\in (\e\X,1)$
	\begin{equation*}
	\widetilde{H}(r,\theta;v^\e_\lambda) \le \frac{1}{2} \widetilde{H}(r,2;v^\e_\lambda) + C\Big( \zeta(r,\e) + \sqrt[3]{\eta(\e\X/r)} \Big) \bigg( \average_{D_{4r}} |\nabla u^\e_\lambda|^2 \bigg)^{1/2}.
	\end{equation*}
\end{lemma}
\begin{proof}
	This follows from a similar argument as Lemma \ref{lem.Hr.iteration} by using (\ref{est.bdryRate}) and Lemma \ref{lem.bdry.v0}.
\end{proof}

Combining (\ref{lem.u-v}) and Lemma \ref{lem.approx.ve}, we obtain
\begin{lemma}
	There exists $\theta\in (0,1/4)$ so that for each $r\in (\e\X,1)$, we have
	\begin{equation*}
	H(\theta r) \le \frac{1}{2}H(r) + C\Big( \zeta(r,\e)^{\gamma}+ \sqrt[3]{\eta(\e\X/r)}   \Big) \Phi(8r).
	\end{equation*}
\end{lemma}
\begin{proof}
	By the triangle inequality, Lemma \ref{lem.approx.ve} and (\ref{est.WCacc.bdry}),
	\begin{equation}\label{est.H1234}
	\begin{aligned}
	H(\theta r) = H(\theta r; u^\e_\lambda) &\le H(\theta r; v^\e_\lambda) + H(\theta r; u^\e_\lambda - v^\e_\lambda) \\
	& \le \widetilde{H}(r,\theta;v^\e_\lambda) + |\widetilde{H}(r,\theta;v^\e_\lambda) - H(\theta;v^\e_\lambda)| + H(\theta r; u^\e_\lambda - v^\e_\lambda) \\
	& \le \frac{1}{2}\widetilde{H}(r,2;v^\e_\lambda) + C\Big( \zeta(r,\e)^{\gamma}+ \sqrt[3]{\eta(\e\X/r)}   \Big) \Phi(8r) \\
	& \qquad + |\widetilde{H}(r,\theta;v^\e_\lambda) - H(\theta r;v^\e_\lambda)| + H(\theta r; u^\e_\lambda - v^\e_\lambda)\\
	& \le \frac{1}{2}H(2r;u^\e_\lambda) + C\Big( \zeta(r,\e)^{\gamma}+ \sqrt[3]{\eta(\e\X/r)}   \Big) \Phi(8r) \\
	& \qquad + |\widetilde{H}(r,\theta;v^\e_\lambda) - H(\theta r;v^\e_\lambda)|  + H(\theta r; u^\e_\lambda - v^\e_\lambda)\\ 
	&\qquad + |\widetilde{H}(r,2;v^\e_\lambda) - H(2r;v^\e_\lambda)| + H(2r; u^\e_\lambda - v^\e_\lambda)\\
	&= \frac{1}{2}H(2r;u^\e_\lambda)+ C\Big( \zeta(r,\e)^{\gamma}+ \sqrt[3]{\eta(\e\X/r)}   \Big) \Phi(8r) \\
	& \qquad + \RN{1} + \RN{2}  + \RN{3} + \RN{4}.
	\end{aligned}
	\end{equation}
	It suffices to estimate $\RN{1}$ and $\RN{2}$, while the estimates of $\RN{3}$ and $\RN{4}$ are the same.
	
	Part (i): Estimate of $\RN{1}$. We will compare the three terms of $\widetilde{H}(r,\theta;v^\e_\lambda)$ and $H(\theta r; v^\e_\lambda)$ separately. Note that the quantities $\widetilde{H}$ and $H$ are defined for the same function $v^\e_\lambda$ over different domains $D_{\theta r}$ and $T_{2r}\cap B_{\theta r}$. Recall that $D_{\theta r} \subset T_{2r}^+ \cap B_{\theta r}$ and $|T_{2r}^+ \cap B_{\theta r} \setminus D_{\theta r}| \le Cr^d \zeta(r,\e) \approx \zeta(r,\e) |D_{\theta r}| \approx \zeta(r,\e) |T_{2r}\cap B_{\theta r}|$. This fact and the Meyers' estimate will be our main ingredients for the estimate of $\RN{1}$. We begin with the first term of $H(\theta r; v^\e_\lambda)$:
	\begin{equation*}
	\begin{aligned}
	&\frac{1}{\theta r}\inf_{q\in\R^d} \bigg( \average_{D_{\theta r} } |v^\e_\lambda  - (n_{\theta r}\cdot x)q|^2 \bigg)^{1/2} \\ &\le \frac{1}{\theta r}\inf_{q\in\R^d} \bigg( \average_{T_{2r}^+ \cap B_{\theta r} } |v^\e_\lambda  - (n_{\theta r}\cdot x)q|^2 \bigg)^{1/2} \bigg( \frac{|T_{2r}^+\cap B_{\theta r}|}{|D_{\theta r}|} \bigg)^{1/2} \\
	& \le \frac{1}{\theta r}\inf_{q\in\R^d} \bigg( \average_{T_{2r}^+ \cap B_{\theta r} } |v^\e_\lambda  - (n_{\theta r}\cdot x)q|^2 \bigg)^{1/2} \big(1+C\zeta(r,\e)\big)^{1/2} \\
	& \le \frac{1}{\theta r}\inf_{q\in\R^d} \bigg( \average_{T_{2r}^+ \cap B_{\theta r} } |v^\e_\lambda  - (n_{\theta r}\cdot x)q|^2 \bigg)^{1/2} + C\zeta(r,\e)^{1/2} \bigg( \average_{T_{2r}^+ \cap B_{\theta r} } |\nabla v^\e_\lambda|^2 \bigg)^{1/2} \\& \le \frac{1}{\theta r}\inf_{q\in\R^d} \bigg( \average_{T_{2r}^+ \cap B_{\theta r} } |v^\e_\lambda  - (n_{\theta r}\cdot x)q|^2 \bigg)^{1/2} + C\zeta(r,\e)^{1/2} \Phi(4r),
	\end{aligned}
	\end{equation*}
	where we also used the energy estimate of $v^\e_\lambda$ and the generalized Caccioppoli inequality (\ref{est.WCacc.bdry}) in the last inequality. To prove the other direction, let $q_{\theta r}$ be the vector that minimizes
	\begin{equation*}
	\frac{1}{\theta r}\inf_{q\in\R^d} \bigg( \average_{D_{\theta r} } |v^\e_\lambda  - (n_{\theta r}\cdot x)q|^2 \bigg)^{1/2}.
	\end{equation*}
	Then
	\begin{equation*}
	|q_{\theta r}| \le \frac{C}{\theta r} \bigg( \average_{D_{\theta r} } |(n_{\theta r}\cdot x)q_{\theta r}|^2 \bigg)^{1/2} \le \frac{C}{\theta r} \bigg( \average_{D_{\theta r} } |v^\e_\lambda|^2 \bigg)^{1/2} \le C\Phi(4r).
	\end{equation*}
	Hence,
	\begin{equation*}
	\begin{aligned}
	&\frac{1}{\theta r}\inf_{q\in\R^d} \bigg( \average_{T_{2r}^+ \cap B_{\theta r} } |v^\e_\lambda  - (n_{\theta r}\cdot x)q|^2 \bigg)^{1/2} \\
	& \qquad \le \frac{1}{\theta r} \bigg( \average_{T_{2r}^+ \cap B_{\theta r} } |v^\e_\lambda  - (n_{\theta r}\cdot x)q_{\theta r}|^2 \bigg)^{1/2} \\
	& \qquad \le \frac{1}{\theta r} \bigg( \average_{D_{\theta r} } |v^\e_\lambda  - (n_{\theta r}\cdot x)q_{\theta r}|^2 \bigg)^{1/2}\bigg( \frac{|D_{\theta r}|}{|T_{2r}^+\cap B_{\theta r}|} \bigg)^{1/2} \\
	& \qquad\qquad + \frac{1}{\theta r} \bigg( \frac{1}{|T_{2r}^+\cap B_{\theta r}|} \int_{T_{2r}^+ \cap B_{\theta r} \setminus D_{\theta r} } |v^\e_\lambda  - (n_{\theta r}\cdot x)q_{\theta r}|^2 \bigg)^{1/2} \\
	& \qquad \le \frac{1}{\theta r} \inf_{q\in \R^d} \bigg( \average_{D_{\theta r} } |v^\e_\lambda  - (n_{\theta r}\cdot x)q|^2 \bigg)^{1/2} + C\zeta(r,\e) \Phi(8r).
	\end{aligned}
	\end{equation*}
	Consequently,
	\begin{equation}\label{est.L2.Dtr}
	\begin{aligned}
	& \bigg| \frac{1}{\theta r}\inf_{q\in\R^d} \bigg( \average_{D_{\theta r} } |v^\e_\lambda  - (n_{\theta r}\cdot x)q|^2 \bigg)^{1/2} - \frac{1}{\theta r}\inf_{q\in\R^d} \bigg( \average_{T_{2r}^+ \cap B_{\theta r} } |v^\e_\lambda  - (n_{\theta r}\cdot x)q|^2 \bigg)^{1/2}\bigg|\\
	& \le C\zeta(r,\e)^{1/2} \Phi(8r).
	\end{aligned}
	\end{equation}
	
	Next, we consider the second part of $\widetilde{H}$ and $H$. We would like to show
	\begin{equation}\label{est.H-1}
	\begin{aligned}
	&\bigg| \norm{\lambda\nabla\cdot v^\e_\lambda - \average_{D_{\theta r}} \lambda\nabla\cdot v^\e_\lambda }_{\underline{H}^{-1}(D_{\theta r})} - \norm{\lambda\nabla\cdot v^\e_\lambda - \average_{T_{2r}^+ \cap B_{\theta r}} \lambda\nabla\cdot v^\e_\lambda }_{\underline{H}^{-1}(T_{2r}^+ \cap B_{\theta r})} \bigg| \\
	& \le C\zeta(r,\e)^\gamma \Phi(4r).
	\end{aligned}
	\end{equation}
	By the definition of $\underline{H}^{-1}$,
	\begin{equation*}
	\begin{aligned}
	&\norm{\lambda\nabla\cdot v^\e_\lambda - \average_{D_{\theta r}} \lambda\nabla\cdot v^\e_\lambda }_{\underline{H}^{-1}(D_{\theta r})} \\
	& \le  \norm{\lambda\nabla\cdot v^\e_\lambda - \average_{ T_{2r}^+ \cap B_{\theta r}} \lambda\nabla\cdot v^\e_\lambda }_{\underline{H}^{-1}(D_{\theta r})} + \bigg|\average_{ T_{2r}^+ \cap B_{\theta r}} \lambda\nabla\cdot v^\e_\lambda - \average_{D_{\theta r}} \lambda\nabla\cdot v^\e_\lambda  \bigg| \\
	& \le  \frac{|T_{2r}^+\cap B_{\theta r}|}{|D_{\theta r}|}\norm{\lambda\nabla\cdot v^\e_\lambda - \average_{ T_{2r}^+ \cap B_{\theta r}} \lambda\nabla\cdot v^\e_\lambda }_{\underline{H}^{-1}(T_{2r}^+ \cap B_{\theta r})} \\
	& \qquad + \bigg|\average_{ T_{2r}^+ \cap B_{\theta r}} \lambda\nabla\cdot v^\e_\lambda - \average_{D_{\theta r}} \lambda\nabla\cdot v^\e_\lambda  \bigg|.
	\end{aligned}
	\end{equation*}
	The second term on the right-hand side of the last inequality is bounded by $\zeta(r,\e)^{1/2} \Phi(4r)$ due to the fact $|T_{2r}^+ \cap B_{\theta r} \setminus D_{\theta r}| \le Cr^d\zeta(r,\e)$. Also observe that $|T^+_{2r}\cap B_{\theta r}|/|D_{\theta r}| = 1+ C\zeta(r,\e)$. It follows that
	\begin{equation}\label{est.H-1.Dtr}
	\begin{aligned}
	& \norm{\lambda\nabla\cdot v^\e_\lambda - \average_{D_{\theta r}} \lambda\nabla\cdot v^\e_\lambda }_{\underline{H}^{-1}(D_{\theta r})}\\ & \qquad \le \norm{\lambda\nabla\cdot v^\e_\lambda - \average_{T_{2r}^+ \cap B_{\theta r}} \lambda\nabla\cdot v^\e_\lambda }_{\underline{H}^{-1}(T_{2r}^+ \cap B_{\theta r})} + C\zeta(r,\e)^{1/2}\Phi(4r).
	\end{aligned}
	\end{equation}
	Conversely, to prove the other direction, we claim the following fact
	\begin{equation*}
	\norm{F}_{\underline{H}^{-1}(T_{2r}^+ \cap B_{\theta r})} \le \norm{F}_{\underline{H}^{-1}(D_{\theta r})} \frac{|D_{\theta r}|}{|T_{2r}^+ \cap B_{\theta r}| } + C\zeta(r,\e)^{\gamma} \bigg( \average_{T_{2r}\cap B_{\theta r}} |F|^p \bigg)^{1/p},
	\end{equation*}
	with $\gamma = 1/2-1/p$ and $p>2$. Applying the above fact to $F = \lambda\nabla\cdot v^\e_\lambda - \average_{T_{2r}^+ \cap B_{\theta r}} \lambda\nabla\cdot v^\e_\lambda$ and using the Meyers estimate of $v^\e_\lambda$, we have
	\begin{equation*}
	\begin{aligned}
	& \norm{\lambda\nabla\cdot v^\e_\lambda - \average_{T_{2r}^+ \cap B_{\theta r}} \lambda\nabla\cdot v^\e_\lambda }_{\underline{H}^{-1}(T_{2r}^+ \cap B_{\theta r})} \\ 
	&\le  \norm{\lambda\nabla\cdot v^\e_\lambda - \average_{T_{2r}^+ \cap B_{\theta r}} \lambda\nabla\cdot v^\e_\lambda }_{\underline{H}^{-1}(D_{\theta r})} + C\zeta(r,\e)^\gamma \Phi(4r) \\
	& \le  \norm{\lambda\nabla\cdot v^\e_\lambda - \average_{D_{\theta r}} \lambda\nabla\cdot v^\e_\lambda }_{\underline{H}^{-1}(D_{\theta r})} + C\zeta(r,\e)^\gamma \Phi(4r).
	\end{aligned}
	\end{equation*}
	This, combined with (\ref{est.H-1.Dtr}), gives (\ref{est.H-1}).

	Finally, for any $k,\ell \in [1/4,1]$, it is not hard to show
	\begin{equation*}
	\begin{aligned}
	& \bigg| \bigg| \average_{D_{k\theta r}} \lambda\nabla\cdot v^\e_\lambda - \average_{D_{\ell \theta r}} \lambda\nabla\cdot v^\e_\lambda \bigg| - \bigg| \average_{T_{2r}^+ \cap B_{k\theta r}} \lambda\nabla\cdot v^\e_\lambda - \average_{T_{2r}^+ \cap B_{\ell \theta r}} \lambda\nabla\cdot v^\e_\lambda \bigg| \bigg| \\
	& \le 2\sup_{k\in [1/4,1]} \bigg| \average_{D_{k\theta r}} \lambda\nabla\cdot v^\e_\lambda - \average_{T_{2r}^+ \cap B_{k\theta r}} \lambda\nabla\cdot v^\e_\lambda \bigg| \\
	& \le C\zeta(r,\e)^{1/2} \Phi(4r).
	\end{aligned}
	\end{equation*}
	
	This, together with (\ref{est.L2.Dtr}) and (\ref{est.H-1}), implies
	\begin{equation*}
	|\widetilde{H}(r,\theta;v^\e_\lambda) - H(\theta r; v^\e_\lambda)| \le C\zeta(r,\e)^\gamma \Phi(8r),
	\end{equation*}
	which is the desired estimate of $\RN{1}$.
	
	Part (ii): Estimate of $\RN{2}$. Recall that
	\begin{equation*}
	\begin{aligned}
	H(\theta r; u^\e_\lambda - v^\e_\lambda) &= \frac{1}{\theta r}\inf_{q\in\R^d} \bigg( \average_{D_{\theta r} } |u^\e_\lambda - v^\e_\lambda - (n_{\theta r}\cdot x)q|^2 \bigg)^{1/2} \\
	& \qquad + \frac{1}{\theta r} \norm{\lambda\nabla\cdot (u^\e_\lambda - v^\e_\lambda) - \average_{D_{\theta r}} \lambda\nabla\cdot (u^\e_\lambda - v^\e_\lambda) }_{\underline{H}^{-1}(D_{\theta r})} \\
	& \qquad + \sup_{k,\ell \in [1/4,1]} \bigg| \average_{D_{k\theta r}} \lambda\nabla\cdot (u^\e_\lambda - v^\e_\lambda) - \average_{D_{\ell \theta r}} \lambda\nabla\cdot (u^\e_\lambda - v^\e_\lambda) \bigg|.
	\end{aligned}
	\end{equation*}
	Then it follows from Lemma \ref{lem.u-v} and Remark \ref{rmk.u-v} that
	\begin{equation*}
	H(\theta r; u^\e_\lambda - v^\e_\lambda) \le C\zeta(r,\e)^\gamma \Phi(8r).
	\end{equation*}
	
	Hence, we have proved the estimates for $\RN{1}$ and $\RN{2}$. Similar argument gives the same estimates for $\RN{3}$ and $\RN{4}$. Therefore,  the desired estimate follows from (\ref{est.H1234}).
\end{proof}

Finally, following the similar argument as the interior estimate and using the full version of Lemma \ref{lem.iteration}, we obtain
\begin{theorem}\label{thm.bdry.Lip}
	For any $s\in (0,d),\lambda \in [0,\infty)$, there exists a random variable $\X = \X_{s,\lambda}: \Omega\mapsto [1,\infty)$ satisfying
	\begin{equation*}
	\X \le \mathcal{O}_{s}(C),
	\end{equation*}
	such that if $r \in [\e \X, 1]$, then
	\begin{equation}
	\bigg( \average_{D_r} |\nabla u_\lambda^\e|^2 \bigg)^{1/2} + \bigg( \fint_{D_r} |\lambda \nabla\cdot u_\lambda^\e -  \fint_{D_2} \lambda \nabla\cdot u_\lambda^\e|^2 \bigg)^{1/2} \le C \bigg( \average_{D_2} |\nabla u_\lambda^\e|^2 \bigg)^{1/2}.
	\end{equation}
\end{theorem}
\begin{proof}
	Let $H,\Phi$ and $h$ be defined as before. In view of Lemma \ref{lem.h.bdry}, the condition (\ref{subeqn:e}) in Lemma \ref{lem.iteration} is not satisfied exactly. However, if we set $H^*(r) = H(r) + \zeta(r,\e)\Phi(r)$. Thus, $H^*,\Phi$ and $h$ satisfy
	\begin{subequations}\label{est.HTP.bdry}
		\begin{align}
		H^*(\theta r) &\le \frac{1}{2} H(r) + C_0 \big\{ \zeta(r,\e)^\gamma + \zeta(\theta^{-1}r,\e) + \sqrt[3]{\eta(\e\X/r)} \big\} \Phi(8r) \label{subeqn:1}\\
		H^*(r) & \le C_0\Phi(r) \label{subeqn:2}\\ 
		h(r) & \le C_0( H^*(r) + \Phi(r))  \label{subeqn:3}\\
		\Phi(r) &\le C_0( H^*(r) + h(r)) \label{subeqn:4}\\
		\sup_{r\le t\le 2r} \Phi(t) & \le C_0\Phi(2r) \label{subeqn:5}\\
		\sup_{r\le s,t\le 2r} |h(s) - h(t)| &\le C_0 H^*(2r). \label{subeqn:6}
		\end{align}
	\end{subequations}
	Recall that for any $r\in (\e\X,1)$
	\begin{equation*}
	\zeta(r,\e)^\gamma + \zeta(\theta^{-1}r,\e) + \sqrt[3]{\eta(\e\X/r)} \le Cr^{\alpha\gamma} + C(\e/r)^{\alpha\gamma} + (\e\X/r)^{\rho/3}.
	\end{equation*}
	Thus, Lemma \ref{lem.iteration} applies (with $\e$ replaced by $\e\X$) and gives
	\begin{equation*}
	\int_{\e\X}^{2} \frac{H^*(r)}{r} dr + \sup_{\e\X \le r\le 2}\Phi(r) \le C\Phi(2).
	\end{equation*}
	Now, we may proceed as the proof of Theorem \ref{thm.Lip} and obtain the desired estimate.
\end{proof}

\begin{proof}[Proof of Theorem \ref{thm.main1}]
	The proof is the same as Theorem \ref{thm.main} by using Theorem \ref{thm.bdry.Lip}.
\end{proof}
\bibliographystyle{amsplain}
\bibliography{Lib}

\begin{flushleft}
Shu Gu,
Department of Mathematics,
Fordham University,
New York, NY 10458,
USA.

E-mail: sgu31@fordham.edu
\end{flushleft}

\begin{flushleft}
Jinping Zhuge,
Department of Mathematics,
University of Chicago,
Chicago, IL 60637,
USA.

E-mail: jpzhuge@math.uchicago.edu
\end{flushleft}

\medskip

\noindent \today
\end{document}